\documentclass[10pt,a4paper]{article}
\usepackage[margin=3cm]{geometry}

\newcommand{\splitspace}{\hspace{4em}}
\newcommand{\figscale}{.775}

\usepackage[utf8]{inputenc}
\usepackage{microtype}
\usepackage[T1]{fontenc}

\usepackage[shortlabels]{enumitem}
\usepackage{booktabs}
\usepackage{multicol}
\usepackage[font=small,margin=8ex]{caption}

\usepackage{graphicx,xcolor}
\usepackage[colorlinks=true,hyperindex=true]{hyperref}

\definecolor{newdarkblue}{RGB}{24,64,184}
\definecolor{newgreen}{RGB}{80,170,55}
\hypersetup{
    colorlinks,%
    citecolor=newdarkblue,%
    filecolor=red,%
    linkcolor=newdarkblue,%
    urlcolor=newgreen,%
    pdfnewwindow=true,%
    pdfstartview={FitH}
}

\usepackage{amsmath,amsfonts,amssymb,amsthm}
\numberwithin{equation}{section}

\newtheorem{theorem}{Theorem}[section]
    \newtheorem{proposition}[theorem]{Proposition}
    \newtheorem{lemma}[theorem]{Lemma}

    \newtheorem{thmx}{Theorem}

    \newtheorem{thmzero}{Theorem}

\theoremstyle{definition}
    \newtheorem{definition}[theorem]{Definition}
    \newtheorem{remark}[theorem]{Remark}
    \newtheorem{example}[theorem]{Example}

\makeatletter
\def\namedlabel#1#2{\begingroup
   \def\@currentlabel{#2}%
   \label{#1}\endgroup
}
\makeatother

\newcommand{\assref}[1]{{\ref{ass:#1}}}

\newcommand{\secref}[1]{Section~{\ref{#1}}}
\newcommand{\appref}[1]{Appendix~{\ref{#1}}}

\def\rr{{\mathbb R}}
\def\zz{{\mathbb Z}}
\def\nn{{\mathbb N}}

\def\cA{{\cal A}}
\def\cF{{\cal F}}
\def\cL{{\cal L}}
\def\cP{{\cal P}}

\def\P{{\mathbb P}}
\def\Q{{\mathbb Q}}

\def\ilpq{I_{\Q}}
\def\qlpq{q_{\Q}}

\def\ilpp{I_{\P}}
\def\qlpp{q_{\P}}

\def\ir{I_R}
\def\uir{\underline I_R}
\def\oir{\overline I_R}
\def\qr{q_R}

\def\oqr{\overline q_R}

\def\iv{I_V}
\def\uiv{\underline I_V}
\def\oiv{\overline I_V}
\def\qv{q_V}

\def\oqv{\overline q_V}

\def\iw{I_W}
\def\uiw{\underline I_W}
\def\oiw{\overline I_W}
\def\qw{q_W}
\def\uqw{\underline q_W}
\def\oqw{\overline q_W}

	\makeatletter
	\providecommand*{\diff}%
	{\@ifnextchar^{\DIfF}{\DIfF^{}}}
	\def\DIfF^#1{%
	\mathop{\mathrm{\mathstrut d}}%
	\nolimits^{#1}\gobblespace}
	\def\gobblespace{%
	\futurelet\diffarg\opspace}
	\def\opspace{%
	\let\DiffSpace\!%
	\ifx\diffarg(%
	\let\DiffSpace\relax
	\else
	\ifx\diffarg%
	\let\DiffSpace\relax
	\else
	\ifx\diffarg\{%
	\let\DiffSpace\relax
	\fi\fi\fi\DiffSpace}

    \def\dd{\diff}

\DeclareMathOperator{\supp}{supp}
\newcommand{\Exp}[1]{\mathrm{e}^{#1}}
\def\complement{\mathsf{c}}
\def\shift{T}
\def\Sc{h_{\textnormal{c}}}
\def\Sr{h_{\textnormal{r}}}
\def\htop{h_{\textnormal{top}}}
\def\ptop{p_{\textnormal{top}}}
\def\potone{\varphi}
\def\pottwo{\psi}

\title{Large deviations of return times and related\\ entropy estimators on shift spaces}
\author{
    No\'e Cuneo\textsuperscript{1} 
    and 
    Renaud Raqu\'epas\textsuperscript{2}
    }
\date{}

\begin{document}

\maketitle

\begin{center}
\small
\begin{tabular}{c c c}
   1. Universit\'e Paris Cit\'e and Sorbonne Universit\'e, && 2. New York University\\
   CNRS, Laboratoire de Probabilit\'es, && Courant Institute of Mathematical Sciences \\
   Statistique et Mod\'elisation && 251 Mercer Street \\
   F-75013 Paris, France &&  New York, NY 10012, United States \\
\end{tabular}
\end{center}

\begin{abstract}
    We prove the large deviation principle for several entropy and cross entropy estimators based on return times and waiting times on shift spaces over finite alphabets. We consider shift-invariant probability measures satisfying some decoupling conditions which imply no form of mixing nor ergodicity. We establish precise relations between the rate functions of the different estimators, and between these rate functions and the corresponding pressures, one of which is the R\'enyi entropy function. For the most commonly used definition of return times, the large-deviation rate function is proved to be nonconvex, except in marginal cases.  The results apply in particular to irreducible Markov chains, equilibrium measures for Bowen-regular potentials, $g$-measures, invariant Gibbs states for absolutely summable interactions in statistical mechanics, and also to probability measures which may be far from Gibbsian, including some hidden Markov models and repeated quantum measurement processes.

    \smallskip


    \smallskip

\end{abstract}

\setcounter{tocdepth}{2}
\tableofcontents

\linespread{1.1}

\section{Introduction}\label{sec:intro}

Return times play a fundamental role in the theory of dynamical systems. In the specific context of a one-sided shift space over a finite alphabet, there is a vast literature on the connection between return times and entropy. 
One of the landmark results in this direction is the following theorem, due to Wyner and Ziv~\cite{WZ89} and to Ornstein and Weiss~\cite{OW93}: if~$\P$ is an ergodic probability measure, then for $\P$-almost every sequence~$x$, the time~$R_n(x)$ it takes for the first~$n$ letters to reappear down the sequence~$x$ (in the same order) grows exponentially with~$n$ at a rate equal to the Kolmogorov--Sinai entropy~$h(\P)$.
In other words, return times, once properly rescaled, provide a sequence $(\tfrac 1n \ln R_n)_{n\in\nn}$ of universal entropy estimators. Here, ``universal'' means that the definition of the random variable~$R_n$ itself makes no reference to the measure~$\P$;  computing $R_n(x)$ requires no explicit information about the marginals of the measure~$\P$ whose entropy is being estimated. Refinements of this behavior in the form of central limit theorems, laws of the iterated logarithm, large deviation principles (LDPs) and multifractal analysis have been an active area of research since the 1990s; see e.g.~\cite{Ko98,CGS99,FW01,O03,Cu05,Jo06,CFMVY19,AACG22}. 

In the present paper, we will prove a full LDP for the sequence $(\tfrac 1n \ln R_n)_{n\in\nn}$, and give an expression of the rate function and pressure in terms of those of the LDP accompanying the celebrated Shannon--McMillan--Breiman (SMB) theorem. We will also carry this analysis for a nonoverlapping notion of return times along a sequence (see the definition of~$V_n$ below) and for waiting times (see the definition of $W_n$ below) involving a pair of sequences as in~\cite{WZ89,Sh93,MS95,Ko98,CDEJR22}. The latter will be related to recent results on the large deviations of $(-\tfrac 1n \ln \Q_n)_{n\in\nn}$ with respect to~$\P$, where $\Q_n$ is the $n$-th marginal of a second shift-invariant probability measure~$\Q$. For each of these results, the assumptions on the measures involve the notions of decoupling of~\cite{CJPS19,CDEJR22}, in turn inspired by~\cite{LPS95,Pf02}.  Roughly speaking, these decoupling assumptions take the form of bounds on how strongly different parts of the sequences may depend on each other. However, unlike mixing conditions, they do not require distant symbols to be asymptotically independent. As we shall see, the decoupling assumptions cover many standard classes of probability measures, including irreducible Markov chains, equilibrium measures for sufficiently regular potentials, $g$-measures, invariant Gibbs states for summable interactions in statistical mechanics, as well as some less standard measures which may be far from Gibbsian in any sense.

To the best of the authors' knowledge, the large deviations of the sequence~$(\tfrac 1n \ln R_n)_{n\in\nn}$ were previously only understood locally (i.e.\ in a bounded interval containing the entropy $h(\P)$), and only for measures satisfying the Bowen--Gibbs property~\cite{CGS99,AACG22}.
The improvements we provide are three-fold. First, we extend the LDP to the whole real line, and describe the possible lack of convexity of the rate function which prevents the G\"artner--Ellis method used in~\cite{AACG22} from producing a global result. Second, we go past the Bowen--Gibbs property by focussing on the class of \emph{decoupled measures} for which the LDP accompanying the SMB theorem and its analogue for cross entropy can be derived from results recently established in~\cite{CJPS19}. Third, we extend this analysis to the related estimators $(\frac 1n \ln V_n)_{n\in\nn}$ and $(\frac 1n \ln W_n)_{n\in\nn}$ to be introduced below.

Our main results also provide a sharp version of the large-deviation upper bounds on $(\tfrac 1n \ln R_n)_{n\in\nn}$ that were obtained in \cite{JB13} for mixing measures on shift spaces, and in \cite{CRS18} for more general dynamical systems.

\paragraph*{Organization of the paper.} In the remainder of \secref{sec:intro}, we discuss our setup and main results. Theorem~\ref{thm:CJPS}, which collects some large-deviation properties of $(-\tfrac 1n \ln \Q_n)_{n\in\nn}$ with respect to $\P$, is viewed as a starting point. Our main new results, Theorems~\ref{thm:mainthmA},~\ref{thm:mainthmB} and \ref{thm:mainthmC}, deal with the large deviations of $(\tfrac 1n \ln W_n)_{n\in\nn}$, $(\tfrac 1n \ln V_n)_{n\in\nn}$ and $(\tfrac 1n \ln R_n)_{n\in\nn}$ respectively. In \secref{sec:toymodel}, we outline the proof on the basis of a  toy model consisting of a mixture of geometric random variables, which we connect to the widely studied problem of exponential approximations of hitting and return times.

In \secref{sec:examplesdisc}, we discuss a selection of applications and provide concrete formulae for the rate functions and pressures whenever possible. It seems that even in the simplest examples (Bernoulli and Markov measures), some of our results are actually new. We also compare our results to those of \cite{CGS99,AACG22} in \secref{sec:equilbowen}, and explain how Theorems~\ref{thm:mainthmA} and~\ref{thm:mainthmC} prove a conjecture stated in \cite{AACG22}.

In \secref{sec:key-set}, we first establish some sharp (at the exponential scale) estimates on the waiting times $W_n$ (\secref{sec:Wn-estimates}), which we then extend to~$R_n$ and~$V_n$ in \secref{sec:Wn-to-Rn}.

\secref{sec:RLfuncts} starts with some reminders about (weak) LDPs, and a brief review of the notion of Ruelle--Lanford functions. We then identify the Ruelle--Lanford functions of $(\tfrac 1n \ln W_n)_{n\in\nn}$, $(\tfrac 1n \ln V_n)_{n\in\nn}$ and $(\tfrac 1n \ln R_n)_{n\in\nn}$ in order to obtain the corresponding weak LDPs. 

In \secref{sec:full}, we complete the proofs of Theorems~\ref{thm:mainthmA}--\ref{thm:mainthmC}: the weak LDPs of the previous section are promoted to full LDPs, the Legendre--Fenchel duality relations between the rate functions and the associated pressures are established, and the (lack of) convexity of the rate function of $(\tfrac 1n \ln R_n)_{n\in\nn}$ is characterized.

In \secref{sec:LLN-v2}, we gain some insight into the set where the rate functions vanish by combining our LDPs and the corresponding almost sure convergence results (law of large numbers) in the literature.

In order to make the paper accessible to readers who may not be familiar with the tools that we will require from both large-deviation theory and dynamical systems,
we include several technical, yet rather standard results in the appendices.
We provide in \appref{app:subshiftsupp} some definitions regarding subshifts and measures of maximal entropy, which are useful for some of the examples in \secref{sec:examplesdisc} as well as for characterizing the convexity of the rate function of $(\tfrac 1n \ln R_n)_{n\in\nn}$.
In \appref{app:decoupling}, we present technical results about our decoupling assumptions, and in particular some sufficient conditions for them to hold. Finally, in \appref{app:weak-to-full}, we give tools to promote weak LDPs to full ones, and we prove Theorem~\ref{thm:CJPS}.

\paragraph*{Notational conventions.} We adopt the convention that~$\nn = \{1,2,3,4,\dotsc\}$. Unless otherwise stated, measures are probability measures. We use the conventions $\ln 0 = -\infty$ and $0 \cdot (\pm \infty) = 0$ so that, in particular, $0 \cdot \ln 0 = 0$. Given a function $f:\rr \to  \rr \cup \{\infty\}$, we denote by $f^*: \rr \to \rr\cup\{\infty\}$ the Legendre--Fenchel transform (convex conjugate) of $f$ defined by $f^*(\alpha) = \sup_{s\in \rr}(\alpha s - f(s))$. The ball of radius~$\epsilon>0$ around the point~$s$ in a metric space is denoted $B(s,\epsilon)$.

\paragraph*{Acknowledgements.} The authors would like to thank Tristan Benoist, Jean-René Chazottes, Vaughn Climenhaga, Bastien Fernandez, Jérôme Rousseau
and Charles-\'Edouard Pfister
for stimulating discussions, relevant references and valuable advice about the manuscript. They are specifically grateful to Vojkan Jak\v{s}i\'{c} for introducing them to this problem in the framework of entropy information theory lectures at McGill University in 2020, for which RR was a course assistant.  
RR acknowledges financial support from the {Natural Sciences and Engineering Research Council of Canada} and from the {Fonds de recherche du Qu\'ebec\,---\,Nature et technologies}. 
Part of this work was done while RR was a post-doctoral researcher at CY Cergy Paris Universit\'e and supported by LabEx MME-DII ({Investissements d'Avenir} program of the French government). 
Part of this work was done during a stay of both authors at the {Centre de recherches math\'ematiques} and McGill University
supported by the {Agence Nationale de la Recherche} of the French government through the grant NONSTOPS (ANR-17-CE40-0006) and by the Centre de recherches math\'ematiques through a CNRS-CRM ``Appel \`a mobilit\'e Qu\'ebec--France'' grant.

\subsection{Setup and hypotheses}
\label{ssec:setup}

We consider the one-sided shift space $\Omega :=\cA^{\nn} = \{x = (x_k)_{k\in\nn} : x_k \in \cA \text{ for all } k \in \nn \}$ for some finite alphabet~$\cA$. The shift map $\shift : \Omega \to\Omega$ is defined by $(\shift x)_{n} := x_{n+1}$. As usual,~$\Omega$ is equipped with the product topology constructed from the discrete topology on~$\cA$, and $\cF$ denotes the corresponding Borel $\sigma$-algebra.

We use common notations such as: $x_k^n$ for the letters $x_k, x_{k+1}, \dots, x_n$ of the sequence~$x\in\Omega$; $[u]$ for the \emph{cylinder} consisting in the sequences~$x$ with $x_1^n=u$; $\P_n$ for the marginal of the probability measure $\P$ on~$\cA^n$, i.e.\ $\P_n(u) = \P([u])$ for all $u\in \cA^n$. We let $\Omega_{\textnormal{fin}} := \bigcup_{n\in\nn}\cA^n$ be the set of words of finite length. The length of a word $u\in \cA^{n}$ is denoted by $|u|=n$. The concatenation of $u,v\in \Omega_{\textnormal{fin}}$ is denoted by $uv$, and the concatenation of $m$ copies of~$u$ is denoted by $u^m$. We use $\cF_n$ for the $\sigma$-algebra generated by the cylinders sets $[u]$ with $u\in \cA^n$, and $\cF_{{\textnormal{fin}}} := \bigcup_{n\in\nn}\cF_n \subset \cF$ for the set of events involving only finitely many coordinates in~$\Omega$. 

We denote by $\cP_{\textnormal{inv}}(\Omega)$ the set of shift-invariant Borel probability measures on $\Omega$. For $\P\in \cP_{\textnormal{inv}}(\Omega)$ and $n\in \nn$, the support of $\P_n$ is $\supp \P_n := \{u\in \cA^n: \P_n(u)>0\}$. Moreover, $\supp \P := \{x\in \Omega: \P_n(x)>0\text{ for all }n\in\nn\}$ is the support of~$\P$, which is a subshift of~$\Omega$. We denote by $\htop(\supp\P)$ the topological entropy of $\supp \P$, which satisfies $\htop(\supp \P)\leq \ln |\cA|$; see \appref{app:subshiftsupp}.%
    \footnote{We use ``$|\,\cdot\,|$'' both for the length of words and the cardinality of discrete sets.}

We now define the three sequences of interest, namely $(R_n)_{n\in \nn}$, $(V_n)_{n\in \nn}$ and $(W_n)_{n\in \nn}$.
The {\em return times} $R_n:\Omega \to \nn$ are defined as
\[
    R_n(x) := \inf\{k\in\nn : \shift^k x \in [x_1^n] \}= \inf\{k\in\nn : x_{k+1}^{k+n}  = x_1^n \}
\]
for all $x\in \Omega$ and $n\in \nn$, and their nonoverlapping counterparts $V_n:\Omega \to \nn$ are defined as
$$
    V_n(x) := \inf\{k\in\nn : \shift^{k+n-1}x \in [x_1^n] \} =  \inf\{k\in\nn : x_{n+k}^{2n+k-1} = x_1^n \}
$$
instead.
The {\em waiting times} $W_n: \Omega\times \Omega \to \nn$ are defined as
$$
    W_n(x,y) := \inf\{k\in\nn : \shift^{k-1}y \in [x_1^n] \}=  \inf\{k\in\nn : y_{k}^{k+n-1} = x_1^n \}
$$
for all $(x,y)\in \Omega\times \Omega$ and $n\in\nn$. In the literature, the function $W_n(x, \cdot\,)$ is often called {\em hitting time} of the set $[x_1^n]$.

We next briefly recall some results about almost sure convergence of the sequences $(\frac 1n \ln R_n)_{n\in \nn}$, $(\frac 1n \ln V_n)_{n\in \nn}$ and $(\frac 1n \ln W_n)_{n\in \nn}$, which justify their role of as (cross) entropy estimators. The LDPs we prove in the present paper complement these almost sure convergence results, without relying on them, nor implying them in general; see \secref{sec:LLN-v2} for an extended discussion.

First, if $\P \in \cP_{\textnormal{inv}}(\Omega)$, then
\begin{equation}\label{eq:asconvergencern}
    \lim_{n\to\infty}\frac 1n \ln R_n(x) = \lim_{n\to\infty}\frac 1n \ln V_n(x) = h_\P(x),
\end{equation}
for $\P$-almost every $x\in \Omega$, where $h_\P$ is the local entropy function defined by
\begin{equation}
    \label{eq:SMB}  
    h_\P(x) := \lim_{n\to\infty} -\frac{1}{n} \ln \P_n(x_1^n).
\end{equation}
By virtue of the SMB theorem, the last limit exists and satisfies $h_\P(x) = h_\P(\shift x)$ for $\P$-almost every~$x$, and its integral with respect to $\P$ is the \emph{Kolmogorov--Sinai entropy} (also commonly called \emph{specific entropy} or simply \emph{entropy}), defined as
$$
h(\P) := \lim_{n\to\infty}-\frac{1}{n} \sum_{u\in\cA^n} \P_n(u) \ln \P_n(u).$$
When~$\P$ is ergodic, we have $h_\P(x) = h(\P)$ for $\P$-almost every $x\in \Omega$, and in this case \eqref{eq:asconvergencern} can be traced back to~\cite{WZ89,OW93}. The relations \eqref{eq:asconvergencern} are less standard when $\P$ is merely assumed to be shift invariant; see~\secref{ssec:as-remix} for references.

Consider now a pair $(\P,\Q$) of shift-invariant probability measures, possibly with $\P=\Q$. 
As we will discuss in \secref{ssec:as-remix}, under some assumptions more general than those of Theorem~\ref{thm:mainthmA} below, it has recently been proved in \cite{CDEJR22} that
\begin{equation*}
    \lim_{n\to\infty} \frac 1n \ln W_n(x,y) = h_\Q(x)
\end{equation*}
for~$(\P\otimes\Q)$-almost all pairs $(x,y)\in \Omega\times\Omega$, where $h_\Q$ is defined as in \eqref{eq:SMB} and integrates with respect to~$\P$ to the specific \emph{cross entropy}
of~$\P$ relative to~$\Q$, i.e.\
\begin{equation}
\label{eq:def-cross}
    \Sc(\P|\Q) := \lim_{n\to\infty} -\frac{1}{n} \sum_{u \in {\cal A}^n} \P_n(u) \ln \Q_n(u).
\end{equation}
If $\P$ is, in addition, assumed to be ergodic, then $h_\Q(x) =  \Sc(\P|\Q)$ for $\P$-almost every $x\in \Omega$. We stress that the $\P$-almost sure existence of~$h_\Q$ and the existence of the limit in~\eqref{eq:def-cross} do not follow from mere shift invariance when $\P\neq\Q$. 
Note that, with $\Sr(\P|\Q)$ the specific \emph{relative entropy}, we have $\Sc(\P|\Q) = h(\P) + \Sr(\P|\Q)$ whenever both sides are well defined, and that $\Sc(\P|\P) = h(\P)$ is always well defined. The numbers $\Sr(\P|\Q)$ and $\Sc(\P|\Q)$ are both relevant in information theory (see e.g.~\cite[\S 5.4]{CoTh06}\footnote{Unfortunately, there is a lack of consensus on terminology. In particular, in \cite{CoTh06}, ``cross entropy'' is only mentioned as a synonym for relative entropy, and what we call cross entropy here plays a central role in Theorem~5.4.3 of \cite{CoTh06} but is not given any specific name therein. As far as the authors are aware, the terminology adopted in the present paper reflects the most common usage in modern applications.}) and are commonly used in classification tasks (see e.g.~\cite[\S{3.13}]{GBC16} or~\cite[\S{5.1.6}]{Mur}).

The \emph{pressure} $\qr$ associated with the sequence $(\tfrac 1n \ln R_n)_{n\in\nn}$ is the function of~$\alpha \in \rr$ defined by
\begin{equation}\label{eq:defqr}
    \qr(\alpha) := \lim_{n\to\infty}\frac 1n \ln \int \Exp{\alpha \ln R_n(x)}\dd\P(x)
\end{equation}
when the limit exists. The function $\qr$ is also referred to as the ``rescaled cumulant-generating function'' or ``$L^\alpha$-spectrum'' in the literature. Similarly, we define
\begin{equation}\label{eq:defqv}
    \qv(\alpha) := \lim_{n\to\infty}\frac 1n \ln \int \Exp{\alpha \ln V_n(x)}\dd\P(x)
\end{equation}
and
\begin{equation}\label{eq:defqw}
    \qw(\alpha) := \lim_{n\to\infty}\frac 1n \ln \iint \Exp{\alpha \ln W_n(x,y)}\dd\P(x)\dd\Q(y)
\end{equation}
when the limits exist.

The above pressures will be expressed in terms of the pressure of the sequence $(-\frac 1n \ln \Q_n)_{n\in\nn}$, which we define as
\begin{equation}\label{eq:limitqpq}
\begin{split}
    \qlpq(\alpha) &:=  \lim_{n\to\infty}\frac 1n \ln \int \Exp{-\alpha \ln \Q_n(x_1^n)} \dd\P(x) \\
    &\phantom{:}=  \lim_{n\to\infty} \frac 1n \ln\sum_{u\in \supp \P_n} \Q_n(u)^{-\alpha}\P_n(u)
\end{split}
\end{equation}
when the limit exists. It will be part of our assumptions below that $\P_n \ll \Q_n$ for every $n\in \nn$, i.e.\ that $\P_n$ is absolutely continuous with respect to $\Q_n$, so the integral and sum in \eqref{eq:limitqpq} are well defined and finite for each $n$. 
When $\Q = \P$, the summand in the rightmost expression of~\eqref{eq:limitqpq} is $\P_n(u)^{1-\alpha}$; up to some sign and normalization convention, the pressure $\qlpp$ thus coincides with the R\'enyi entropy function of $\P$.

A common route to establishing the LDP, which is followed in particular in \cite{AACG22}, is to first study the pressure in detail, and then derive the LDP using an adequate version of the G\"artner--Ellis theorem. Intrinsic to this approach are the differentiability of the pressure and the convexity of the rate function, both of which fail in our setup; see \secref{sec:equilbowen} for further discussion of the results in \cite{AACG22}. Our approach goes in the opposite direction: first the LDP is established, and only then is some version of Varadhan's lemma used to describe the pressure. This method allows to consider significantly more general measures, and to obtain the full LDP with a possibly nonconvex rate function. This path was already followed in \cite{CJPS19} in order to establish, in particular, the LDP for $(-\tfrac 1n \ln \Q_n)_{n\in\nn}$ with respect to $\P$ and the properties of $\qlpq$.

\paragraph{Assumptions.} In order to state the decoupling assumptions below, we require a sequence $(\tau_n)_{n\in\nn}$ in $\nn\cup\{0\}$ and a sequence $(C_n)_{n\in\nn}$ in $[1,\infty)$, assumed to be fixed and to satisfy
\begin{equation}\label{eq:tauncn}
    \lim_{n\to\infty} \frac{\tau_n}{n} = \lim_{n\to\infty} \frac{\ln C_n}{n} = 0.
\end{equation}
We will freely write  $\tau_n = o(n)$ and $C_n = \Exp{o(n)}$ or speak of an ``$o(n)$-sequence'' and an ``$\Exp{o(n)}$-sequence'' when referring to the conditions \eqref{eq:tauncn}.

\begin{definition}[UD, SLD, JSLD, admissible pair]
    \namedlabel{ass:ud}{UD}\namedlabel{ass:sld}{SLD}\namedlabel{ass:jsld}{JSLD}
    Let $\P\in \cP_{\textnormal{inv}}(\Omega)$. We say that $\P$ satisfies the {\em upper decoupling} assumption ({UD}) if for all $n,m\in\nn$, $u\in \cA^n$, $v\in \cA^m$ and $\xi\in \Omega^{\tau_n}$,
    \begin{equation}\label{eq:defUD}
		\P_{n+\tau_n+m}\left(u\xi v\right) \leq {C_n} \P_n(u)\P_m(v).        
    \end{equation}
    We say that $\P$ satisfies the {\em selective lower decoupling} assumption ({SLD}) if for all $n,m\in\nn$, $u\in \cA^n$ and $v\in \cA^m$, there exist  $0\leq \ell \leq \tau_n$ and $\xi\in\cA^\ell$ such that 
    \begin{equation}\label{eq:SLDeq}
		\P_{n+\ell+m}\left(u\xi v\right) \geq C_n^{-1} \P_n(u)\P_m(v).    
    \end{equation}
	A pair of measures $(\P, \Q)$ with $\P, \Q\in \cP_{\textnormal{inv}}(\Omega)$ is said satisfy the {\em joint selective lower decoupling} assumption ({JSLD}) if for all $n,m\in\nn$, $u\in \cA^n$ and $v\in \cA^m$, there exist  $0\leq \ell \leq \tau_n$ and $\xi\in\cA^\ell$ such that 
	\begin{equation}\label{eq:jslddef}
		\P_{n+\ell+m}\left(u\xi v\right) \geq C_n^{-1} \P_n(u)\P_m(v)
        \qquad\text{and}\qquad
        \Q_{n+\ell+m}\left(u\xi v\right) \geq C_n^{-1} \Q_n(u)\Q_m(v).		
	\end{equation}
	Finally, a pair of measures $(\P, \Q)$ with $\P, \Q\in \cP_{\textnormal{inv}}(\Omega)$ is said to be {\em admissible} if $\P_n \ll \Q_n$ for all $n\in \nn$, if $\P$ and $\Q$ satisfy~\assref{ud}, and if the pair $(\P, \Q)$ satisfies~\assref{jsld}.
\end{definition}

\begin{remark}\label{rem:tau0sld}
	If the pair $(\P, \Q)$ satisfies~\assref{jsld}, then obviously both $\P$ and $\Q$ satisfy~\assref{sld}, but the converse need not hold, since in~\eqref{eq:jslddef} both inequalities are required to hold for the {\em same} $\xi$. If $\tau_n =0$ for all $n$, then~\assref{jsld} becomes equivalent to~$\P$ and~$\Q$ both satisfying~\assref{sld}.
\end{remark}

\begin{remark}\label{rem:changetaun}
    We can always increase the constants $C_n$ and $\tau_n$ (see Lemma~\ref{lem:increasecntaun}), so there is no loss of generality in taking the same sequences $(\tau_n)_{n\in \nn}$ and $(C_n)_{n\in \nn}$ for~\assref{jsld} (resp. \assref{sld}) and~\assref{ud}, as well as for both measures $\P$ and $\Q$.
\end{remark}

Next, the following numbers will play an important role:
\begin{equation}
\label{eq:defgammaplus}
        \gamma_+ := \limsup_{n\to\infty}\frac 1n \sup_{u\in \cA^n}\ln \P_n(u)
        \qquad\text{and}\qquad
        \ \gamma_- := \liminf_{n\to\infty}\frac 1n \inf_{u\in \supp \P_n}\ln \P_n(u).
\end{equation}
One easily shows that 
\begin{equation}\label{eq:ineqgammapm}
    0 \leq -\gamma_+ \leq h(\P) \leq  \htop(\supp\P) \leq -\gamma_- \leq \infty;
\end{equation}
see \appref{app:subshiftsupp} for a proof and a definition of the topological entropy $\htop$.
For some results, $\gamma_+$ will be required to be well approximated by periodic sequences:
\begin{definition}[PA]
\label{def:pa}
    \namedlabel{ass:pa}{PA}  
        A measure $\P\in \cP_{\textnormal{inv}}(\Omega)$ satisfies the {\em periodic approximation} assumption ({PA})
        if for every $\epsilon > 0$, there exists $p \in \nn$ and $u\in \cA^p$ such that
        \begin{equation*}
            \liminf_{n\to\infty} \frac 1{np} \ln  \P_{np }\left(u^n\right) \geq \gamma_+ - \epsilon.
        \end{equation*}
\end{definition}

\begin{remark}\label{rem:tau0pa}
	If $\P$ satisfies~\assref{sld} with $\tau_n=0$, then automatically~\assref{pa} holds; a more general sufficient condition for \assref{pa} to hold is given in Lemma~\ref{lem:pafromarray}. 
\end{remark}

We discuss at the end of \secref{sec:atorigin} possible ways to weaken~\assref{pa}.

\paragraph{Starting point.} Our analysis is built on top of the LDP for the sequence $(-\frac 1n \ln \Q_n)_{n\in\nn}$ viewed as a family of random variables on $(\Omega,\P)$.\footnote{While $\Q_n$ was defined as a measure on~$\cA^n$, it is reinterpreted here in the straightforward way as a measurable function on~$\Omega$.}
The following theorem summarizes the large-deviation properties of this sequence and essentially follows from the results of \cite{CJPS19}, where the corresponding weak LDP is proved. Taking this weak LDP for granted, the proof of the full theorem only requires some minor adjustments to the arguments in \cite{CJPS19}; for completeness we provide the details in  \appref{sec:proof-CJPS}.
In some important applications, the conclusions of the theorem are well known and follow from standard methods; see Sections~\ref{sec:exiid}--\ref{sec:equilbowen} below. The terminology regarding (weak and full) LDPs, rate functions and exponential tightness is summarized in \secref{sec:RLfuncts}.

\begin{thmzero}[LDP for $-\frac 1n \ln \Q_n$]\label{thm:CJPS}
    If $(\P, \Q)$ is an admissible pair, then the following hold:
    \begin{enumerate}[i.]
        \item The sequence $(-\tfrac 1n \ln \Q_n)_{n\in\nn}$ on $(\Omega, \P)$ satisfies the LDP with a convex rate function $\ilpq : \rr \to [0,\infty]$ satisfying $\ilpq(s) = \infty$ for all $s<0$. 
        \item The limit defining $\qlpq$ in~\eqref{eq:limitqpq}
        exists in $(-\infty, \infty]$ for all $\alpha\in\rr$ and defines a nondecreasing, convex, lower semicontinuous function $\qlpq:\rr \to (-\infty,\infty]$ satisfying $\qlpq(0)=0$.
        Moreover, the following Legendre--Fenchel duality relations hold:
        \begin{equation}\label{eq:LdualIpq}
            \qlpq = \ilpq^* 
            \qquad\text{and}\qquad
            \ilpq = \qlpq^*.
        \end{equation}
        \item Assume now that $\Q=\P$. Then, in addition to the above, the following hold:
        \begin{enumerate}[a.]
            \item The sequence $(-\frac 1n \ln \P_n)_{n\in \nn}$ on $(\Omega, \P)$ is exponentially tight and $\ilpp$ is a good rate function satisfying
            \begin{equation}\label{eq:Isinterval}
                \begin{alignedat}{3}
                    \ilpp(s) &\in [ s-\htop(\supp \P), s] \qquad && \text{if } s\in [-\gamma_+, -\gamma_-],\\
                    \ilpp(s) & =\infty && \text{otherwise,}                    
                \end{alignedat}
            \end{equation}
            where $[-\gamma_+, -\gamma_-]$ is understood to be $[-\gamma_+, \infty)$ if $\gamma_- = -\infty$.
            \item The limit superior (resp. inferior) defining $\gamma_+$ (resp. $\gamma_-$) is actually a limit.
            \item  For every $\alpha \leq 1$,
            \begin{equation}\label{eq:qppbdd}
                \qlpp(\alpha)\leq \qlpp(1) = \htop(\supp\P).
            \end{equation}
            \item Either $\gamma_- > -\infty$ and $\qlpp(\alpha)<\infty$ for all $\alpha\in \rr$, or $\gamma_- = -\infty$ and $\qlpp(\alpha) = \infty$ for all $\alpha > 1$. 
        \end{enumerate}
    \end{enumerate}
\end{thmzero}

\begin{remark}
    As we shall see in \secref{sec:LLN-v2}, under the assumptions of Theorem~\ref{thm:CJPS}, the limit defining $\Sc(\P|\Q)$ exists in $[0, \infty]$, and we have $\ilpq(\Sc(\P|\Q))=0$ when $\Sc(\P|\Q)<\infty$ and $\lim_{s\to\infty}\ilpq(s)=0$ when $\Sc(\P|\Q)=\infty$. In particular,  $\ilpp(h(\P))=0$ when $\Q = \P$. These conclusions will extend to the other rate functions to be introduced below; see again \secref{sec:LLN-v2}.
\end{remark}

\begin{remark}
    In the course of the proof, we shall establish and use the relation
    \begin{equation}\label{eq:energyentropy}    
        \ilpp(s) = s - \lim_{\epsilon\to 0} \limsup_{n\to\infty}\frac 1n \ln \left|\left\{u\in \cA^n:-\frac 1n \ln \P_n(u)\in B(s,\epsilon)\right\}\right|,
    \end{equation}   
    which can be seen as a classical ``energy-entropy competition''. Note that the logarithm is either nonnegative or $-\infty$, which explains why we have either $\ilpp\leq s$ or $\ilpp=\infty$ in~\eqref{eq:Isinterval}.
\end{remark}

\begin{remark}
	A standard consequence of~\eqref{eq:LdualIpq} and \eqref{eq:Isinterval}, which can also be derived from the definition of $\qlpp$ directly, is that if $\P$ satisfies~\assref{ud} and~\assref{sld}, then
        \begin{equation}\label{eq:gammapmlimit}
            -\gamma_\mp = \lim_{\alpha \to \pm\infty}\frac{\qlpp(\alpha)}\alpha.
        \end{equation}
    Readers who are unfamiliar with such properties of Legendre--Fenchel transforms may benefit from reading the introductions in Section~2.3 of \cite{DZ} and Chapter~VI of \cite{Ell}, or the in-depth exposition in \cite{Roc}.
\end{remark}

\begin{remark}
	The assumption that the alphabet $\cA$ is finite is important, as many of our estimates rely on the constant $|\cA|$ in a way which does not seem easy to circumvent. We do not expect the results to remain true on countably infinite alphabets without further assumptions on the measures at hand. This is a matter we would like to investigate in future research.
\end{remark}

\subsection{Main results}\label{sec:mainresults}

    We shall establish the LDP and express the rate function and pressure for the sequence $(\tfrac 1n \ln W_n)_{n\in\nn}$ in terms of $\ilpq$ and $\qlpq$.
    Similarly, the rate functions and pressures for $(\tfrac 1n \ln V_n)_{n\in\nn}$ and $(\tfrac 1n \ln R_n)_{n\in\nn}$ will be expressed in terms of $\ilpp$ and $\qlpp$.
       
    \begin{thmx}[LDP for $\frac 1n \ln W_n$]\label{thm:mainthmA}
    If $(\P, \Q)$ is an admissible pair, then the following hold:
    \begin{enumerate}[i.]
        \item  The sequence $(\tfrac 1n \ln W_n)_{n\in\nn}$ satisfies the LDP with respect to~$\P\otimes\Q$ with the convex rate function $\iw$ given by
        \begin{equation}\label{eq:defiwexplicit}
            \iw(s) :=
            \begin{cases}
                \infty &\text{if }s<0,\\
                \inf_{r\geq s}(r-s+\ilpq(r)) &\text{if }s\geq 0.
            \end{cases}
        \end{equation}
        \item For all~$\alpha\in\rr$, the limit in~\eqref{eq:defqw} exists in $(0, \infty]$,
        \begin{equation}\label{eq:formqW}
            \qw(\alpha) = \max\{\qlpq(\alpha), \qlpq(-1)\},
        \end{equation}
        and the Legendre--Fenchel duality relations $\qw = \iw^*$ and $\iw = \qw^*$ hold.
        \item If $\Q=\P$, then $\iw$ is a good rate function and $(\tfrac 1n \ln W_n)_{n\in\nn}$ is an exponentially tight family of random variables.
    \end{enumerate}
\end{thmx}

\begin{remark}\label{rem:sstar} 
    The relations \eqref{eq:defiwexplicit} and \eqref{eq:formqW} can be written as
    \begin{equation}\label{ex:formiwqwm1}
        \iw(s)=\begin{cases}
            s_0-s + \ilpq(s_0)&\text{if } s< s_0,\\
            \ilpq(s)&\text{if } s\geq s_0,
        \end{cases}
        \qquad \text{and}\qquad \qw(\alpha) = \begin{cases}
            \qlpq(-1)&\text{if } \alpha < -1,\\
            \qlpq(\alpha)&\text{if } \alpha \geq -1,
        \end{cases}
    \end{equation}
    where $s_0$ is any point such that $-1$ belongs to the subdifferential of $\ilpq$ at $s_0$, or equivalently, such that $s_0$ belongs to the subdifferential of $\qlpq$ at $-1$. In nontrivial cases, $q_W$ is not differentiable at~$\alpha = -1$. The situation is depicted on the basis of an example in Figure~\ref{fig:BernoulliPQ}.
\end{remark}

\begin{figure}[htb]

    \centering 

    \begin{tabular}{p{0.45\textwidth} p{0.45\textwidth}}
        \vspace{0pt} \includegraphics[scale=\figscale]{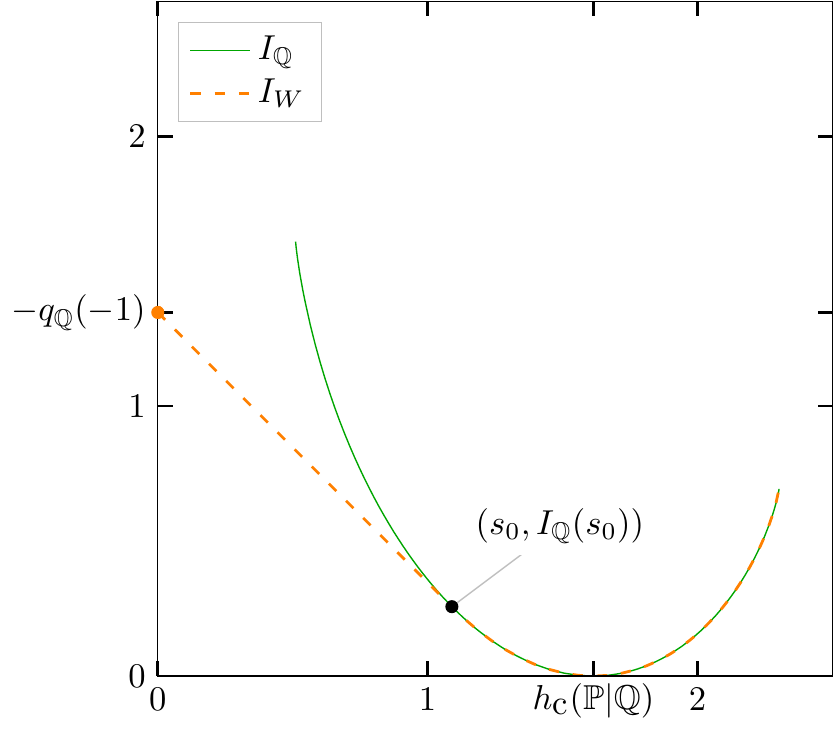} &
        \vspace{0pt} \includegraphics[scale=\figscale]{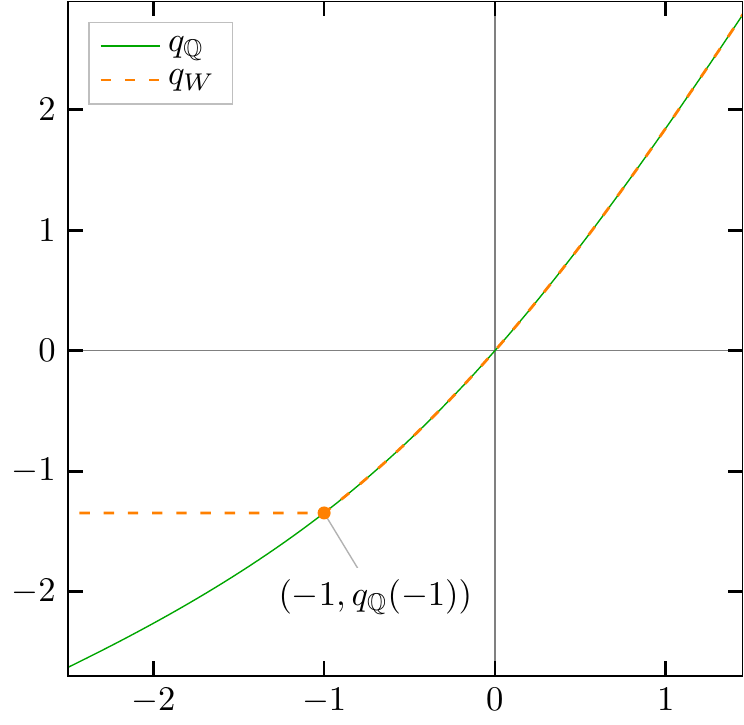}
    \end{tabular}
      
    \caption{The rate functions and pressures of Theorems~\ref{thm:CJPS} and~\ref{thm:mainthmA} for some Bernoulli measures $\P$ and~$\Q$; see \secref{sec:exiid}. The rate functions are infinite wherever not drawn.}
    \label{fig:BernoulliPQ}
\end{figure}

The next two theorems involve only one measure~$\P$. We remark that the pair $(\P, \P)$ is admissible if and only if $\P$ satisfies both~\assref{ud} and~\assref{sld}. As a consequence, the conclusions of Theorem~\ref{thm:CJPS} hold under the assumptions of Theorems~\ref{thm:mainthmB} and~\ref{thm:mainthmC}; in particular $\ilpp$ and $\qlpp$ are well defined, and the numbers $\gamma_+\in (-\infty, 0]$ and $\gamma_-\in [-\infty,0]$ defined in \eqref{eq:defgammaplus} are actual limits.

\begin{thmx}[LDP for $\frac 1n \ln V_n$]\label{thm:mainthmB}
    If $\P\in \cP_{\textnormal{inv}}(\Omega)$ satisfies~\assref{ud} and~\assref{sld}, then the following hold:
    \begin{enumerate}[i.]
        \item The sequence $(\tfrac 1n \ln V_n)_{n\in\nn}$ is exponentially tight and satisfies the LDP with respect to~$\P$ with the good, convex rate function $\iv$ given by
        \begin{equation}\label{eq:defivexplicit}
             \iv(s) := 
             \begin{cases}
                 \infty &\text{if }s<0,\\
                 \inf_{r\geq s}(r-s+\ilpp(r)) &\text{if }s\geq 0.
             \end{cases}
         \end{equation}
         \item For all $\alpha\in\rr$, the limit in~\eqref{eq:defqv} exists in $(0, \infty]$,
        \begin{equation}\label{eq:formqV}
            \qv(\alpha) = \max\{\qlpp(\alpha), \qlpp(-1)\},
        \end{equation}
        and the Legendre--Fenchel duality relations $\qv = \iv^*$ and $\iv = \qv^*$ hold. Moreover, 
        \begin{equation}\label{eq:reliv0qlppm1}
            \iv(0) = -\qlpp(-1) \geq - \gamma_+.
        \end{equation}
    \end{enumerate}
\end{thmx}

\begin{remark}\label{eq:iveqiw}
    The rate function $\iv$ corresponds to $\iw$ in the special case $\Q = \P$. While one can, of course, choose $\Q = \P$ in Theorem~\ref{thm:mainthmA} (in this case it suffices that $\P$ satisfies~\assref{ud} and~\assref{sld}), this special case is not equivalent to Theorem~\ref{thm:mainthmB}. Indeed, $W_n$ and~$V_n$ are still distinct in their definition and underlying probability space. It is known that the range of applicability of almost sure entropy estimation via $W_n$ is strictly smaller than that via~$V_n$ (or $R_n$); see~\cite{OW93} and~\cite[\S{4}]{Sh93}.
\end{remark}

\begin{thmx}[LDP for $\frac 1n \ln R_n$]\label{thm:mainthmC} 
    If $\P\in \cP_{\textnormal{inv}}(\Omega)$ satisfies~\assref{ud},~\assref{sld} and~\assref{pa}, then the following hold:
    \begin{enumerate}[i.]
        \item The sequence $(\tfrac 1n \ln R_n)_{n\in\nn}$ is exponentially tight and satisfies the LDP with respect to~$\P$ with the good (possibly nonconvex) rate function $\ir$ given by
        \begin{equation}\label{eq:defirexplicit}
            \ir(s) := 
            \begin{cases}
            \infty &\text{if }s<0,\\
            -\gamma_+ &\text{if }s = 0,\\
            \inf_{r\geq s}(r-s+\ilpp(r)) &\text{if }s>0.
            \end{cases}
        \end{equation}
        
        \item For all $\alpha \in \rr$, the limit in~\eqref{eq:defqr} exists in $(0, \infty]$ and
            \begin{equation}\label{eq:formqR}
                \qr(\alpha) = \max\{\qlpp(\alpha), \gamma_+\}.
            \end{equation}
        Moreover, $\qr = \ir^*$.
        \item We have the following relations:
        $\ir$ is convex $\iff$ $\ir = \iv$ $\iff$ $\qr = \qv$ $\iff$ $\gamma_+ = \qlpp(-1)$ $\iff$ $\ilpp(-\gamma_+) = 0$ $\iff$ $\qlpp(\alpha) = -\gamma_+\alpha$ for all $\alpha \leq 0$ $\impliedby$ $\gamma_+= \gamma_-$ $\iff$ $\ilpp(s) = \infty$ for all $s \in \rr\setminus\{-\gamma_+\}$$\iff$ $h(\P) = \htop(\supp\P)$. 
        Moreover, if $\tau_n=O(1)$ and $C_n=O(1)$, then all the above properties are actually equivalent.
    \end{enumerate}
\end{thmx}

By their definition, the rate functions $\ir$ and $\iv$ may differ only at~$0$, and by~\eqref{eq:reliv0qlppm1} we have $\ir(0) = -\gamma_+\leq \iv(0)$; see Figure~\ref{fig:Bernoullisimple} for an illustration. 
Part~iii gives many technical conditions equivalent to $\ir = \iv$, the most notable of which is the convexity of $\ir$. The situation where $\ir = \iv$ should be seen as quite degenerate; in the generic case, we have $\ir(0)<\iv(0)$, as strikingly illustrated by the examples in Sections~\ref{sec:exiid} and \ref{sec:markov}. The generic inequality $\ir(0)<\iv(0)$ is due to fact that the definition of $R_n$ allows for overlaps (characterized by $R_n < n$) which are excluded in $V_n$. These overlaps, which make the LDP for $(\frac 1n \ln R_n)_{n\in \nn}$ highly interesting, were widely studied in other contexts; see the discussion and references in \secref{sec:toymodel}. With help of \assref{pa}, we will show in \secref{sec:atorigin} that $\P\{x:R_n(x) < n\}$ asymptotically decreases like $\Exp{n\gamma_+}$, which explains the equality $\ir(0)=-\gamma_+$.

\begin{figure}[htb]
    \centering 

    \begin{tabular}{p{0.45\textwidth} p{0.45\textwidth}}
        \vspace{0pt} \includegraphics[scale=\figscale]{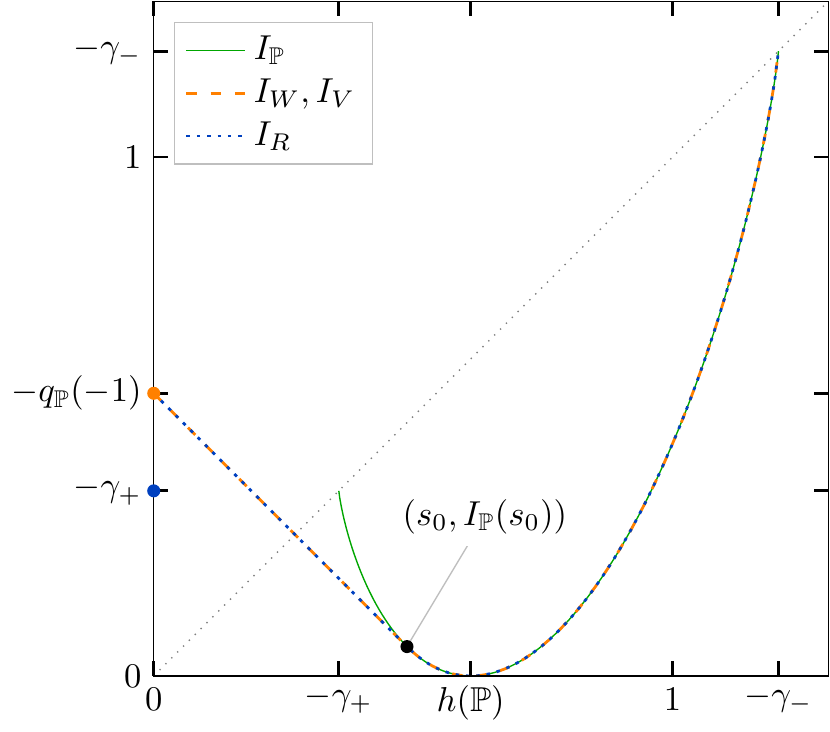} &
        \vspace{0pt} \includegraphics[scale=\figscale]{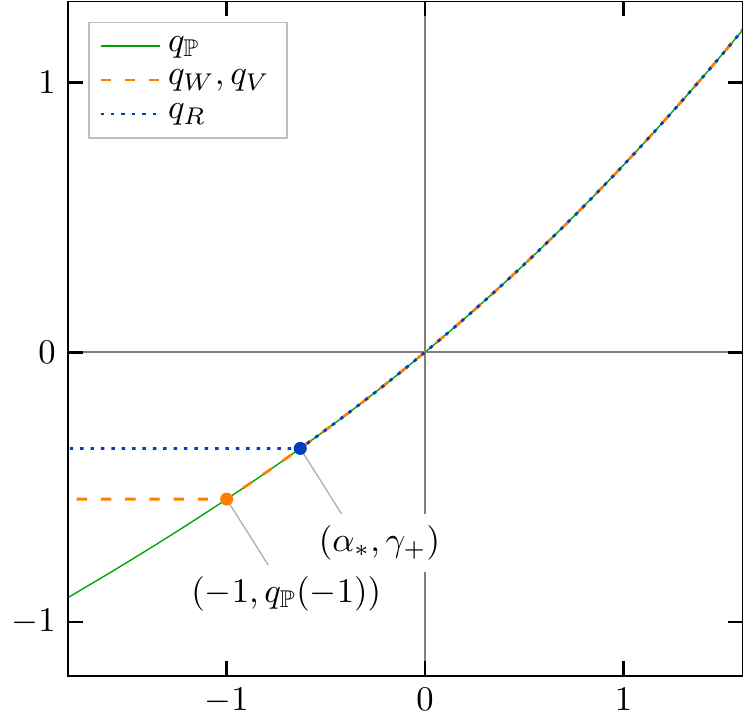}
    \end{tabular}
      
    \caption{The rate functions and pressures of Theorems~\ref{thm:CJPS} and~\ref{thm:mainthmA}--\ref{thm:mainthmC} for some Bernoulli measure~$\P$; see \secref{sec:exiid}.  The rate functions are infinite wherever not drawn.}
    \label{fig:Bernoullisimple}
\end{figure}

\begin{remark}\label{rem:gammap0}
    If $\gamma_+ = 0$, then $\ilpp(0) = 0$ by~\eqref{eq:Isinterval}, and thus $\qlpp(\alpha) = \ilpp^*(\alpha) = 0$ for all $\alpha \leq 0$. In this case, we easily see that $\iv = \ir = \ilpp$ and $\qv = \qr = \qlpp$. 
\end{remark}

\begin{remark}\label{rem:alphastar}The relation \eqref{eq:formqR} can be written as 
    \begin{equation}\label{eq:formirqrm1}
        \qr(\alpha) = \begin{cases}
        \gamma_+&\text{if } \alpha < \alpha_*,\\
        \qlpp(\alpha)&\text{if } \alpha \geq \alpha_*,
    \end{cases}
    \end{equation}
    where $\alpha_*\leq 0$ is such that $\qlpp(\alpha_*)=\gamma_+$. If $\gamma_+ = 0$, then we can take any $\alpha_*\leq 0$ in view of Remark~\ref{rem:gammap0}. 
    If $\gamma_+<0$, then $\qlpp(-1)\leq \gamma_+ < 0 = \qlpp(0)$ by~\eqref{eq:reliv0qlppm1}, so the point $\alpha_*$ is unique and satisfies $\alpha_*\in [-1, 0)$, with $\alpha_* = -1$ if and only if $\iv = \ir$.
\end{remark}

\subsection{Outline of the proof for a toy model}\label{sec:toymodel}

We present a natural toy model, consisting of a mixture of geometric distributions, which not only provides the correct rate functions for $W_n$ and $V_n$, but also illustrates the method of our proof. We then modify the toy model in order to take overlaps (corresponding to $R_n < n$) into account and guess the correct rate function for $R_n$.

Approximations of waiting times and return times by geometric random variables (or exponential random variables in the scaling limit) 
have been widely studied, typically under mixing assumptions; see e.g.~\cite{GS97,CGS99,HSV1999,AG01,HV10} and \cite{AAG21} for a recent overview and more exhaustive references. As often emphasized in the literature, possible overlaps, which are tightly related to Poincar\'e recurrence times (see \cite{AV08,AC15,AAG21}), play a crucial role in such approximations. We shall further comment on Poincar\'e recurrence times at the end of \secref{sec:atorigin}.

Our decoupling assumptions do not seem to imply any of the very sharp exponential approximations that are available in the literature. In comparison, the geometric approximation that we prove in \secref{sec:Wn-estimates} are quite loose in the sense that scaling factors and error terms may grow subexponentially. Yet, it suffices to establish the LDPs of interest. 

\paragraph*{Geometric approximation.} 
We start with the following interpretation of $W_n(x,y)$: first $x$ is drawn at random according to the law $\P$ and $y$ is drawn at random according to the law $\Q$, independently of~$x$. Then, for each $k\in\nn$, we check whether $y_{k}^{k+n-1} = x_1^n$ or not. Once~$x_1^n$ is given, shift invariance implies that $\Q\{y\in \Omega: y_{k}^{k+n-1} = x_1^n\} = \Q_n(x_1^n)$ for each $k\in\nn$, and we can view $W_n$ as the time of the first ``success'' in a series of attempts (indexed by~$k$). 
If the attempts were mutually independent, $W_n$ would be a geometric random variable with random parameter $p_n = \Q_n(x_1^n)$.
Of course, these attempts are not independent: even if~$\Q$ is a Bernoulli measure, the attempts $k$ and $k'$ are only independent for $|k'-k| \geq n$. 
However, it turns out, due to our decoupling assumptions, that the asymptotic behavior of $W_n$ at the scale that is relevant to our LDP is accurately captured by this simplified geometric model. 

We now define the toy model properly. Let $\widetilde W_n$ the random variable whose law $\nu_n$ on $\nn$ is given by
\begin{equation}\label{eq:sumanun}
    \nu_n(k) 
        := \sum_{u\in \cA^n} \P_n(u)\Q_n(u)(1-\Q_n(u))^{k-1},
\end{equation}
for every $k\in \nn$. This is a simple mixture of geometric distributions, motivated by the above discussion. The next proposition describes the large deviations of $\widetilde W_n$.
Since it is introduced only for illustration purposes, and since the actual proof of the proposition relies on estimates which are similar to\,---\,but simpler than\,---\,those we provide for $W_n$ in the main body of the paper, we limit ourselves to sketching the proof. The interested reader will easily be able to fill in the details.

\begin{proposition}\label{prop:ptoymodel}
    If the pair $(\P, \Q)$ is admissible, then the sequence $(\frac 1n \ln \widetilde W_n)_{n\in\nn}$ satisfies the LDP with the rate function~$\iw$ defined in \eqref{eq:defiwexplicit}.
\end{proposition}

\begin{proof}[Sketch of the proof.]
    For each $s\in M_n :=\{\frac 1n \ln k: k\in \nn\}$, the probability that $\frac 1n \ln \widetilde W_n$ equals $s$ is given by
    \begin{equation}\label{eq:distrhon}
        \nu_n(\Exp{ns}) = \sum_{u\in \cA^n} \P_n(u)\Q_n(u)(1-\Q_n(u))^{\Exp{ns}-1} = \int_{(0,\infty)} \Exp{-nr}(1-\Exp{-nr})^{\Exp{ns}-1} \dd\mu_n(r),
    \end{equation}
    where we denote by $\mu_n$ the distribution of $-\frac 1n \ln \Q_n$ with respect to $\P$.
    Now using that $\Exp{-nr}$ is very small for~$n$ large, a formal first order Taylor expansion yields 
    \begin{align}\label{eq:cutoffexp}
        (1-\Exp{-nr})^{\Exp{ns}-1} 
        \sim  \exp(-(\Exp{ns}-1)\Exp{-nr})  \sim  \exp(-\Exp{n(s-r)}).
    \end{align}
    We thus have a cut-off phenomenon: the above vanishes superexponentially when $r<s$, and is very close to $1$ when $r>s$. Formal substitution into~\eqref{eq:distrhon} yields 
    $$
        \nu_n(\Exp{ns}) \sim  \int_{(s,\infty )} \Exp{-nr} \dd\mu_n(r).
    $$
    We remark that this cut-off argument corresponds to retaining, in the sum in~\eqref{eq:sumanun}, only the $u\in \cA^n$ such that $\Q_n(u) \lesssim \Exp{-ns}$. 
    
    Now, for $s>0$ and $\epsilon$ small, the set $B(s,\epsilon)\cap M_n$ contains approximately $\Exp{n(s+\epsilon)}-\Exp{n(s-\epsilon)}\sim \Exp{ns}$ points, so
    \begin{equation}\label{eq:nuenbeforesup}
        \nu_n\left\{k\in \nn: \frac 1n\ln k \in B(s,\epsilon)\right\} \sim  \Exp{ns}\nu_n(\Exp{ns}) \sim  \int_{(s,\infty )} \Exp{n(s-r)} \dd\mu_n(r).
    \end{equation}
    Formally, Theorem~\ref{thm:CJPS} says that $\dd\mu_n(r)\sim \Exp{-n\ilpq(r)}\dd r$, and a saddle-point approximation in~\eqref{eq:nuenbeforesup} then yields
    $$
        \frac 1n \ln \nu_n\left\{k\in \nn: \frac 1n\ln k \in B(s,\epsilon)\right\} \sim \sup_{r\geq s}(s-r-\ilpq(r)) = -\iw(s).
    $$
    The same conclusion applies to the case $s=0$ without the need for any cut-off argument. This local formulation of the LDP (i.e.\ concerning only small balls) implies the weak LDP (see Section~\ref{sec:RLfuncts}), which can in turn be promoted to a full LDP using the arguments of Section~\ref{sec:full}.
\end{proof}

The fact that we consider $W_n(x,y)$ for $x$ and $y$ that are mutually independent was an important ingredient of the above argument. When moving on to~$V_n$ and~$R_n$ (and replacing $\Q$ with~$\P$), we look for occurrences of~$x_1^n$, not in an independent sample~$y$, but in~$x$ itself, which introduces more dependence. 
Conditioned on the event $[u]$ for some $u\in \cA^n$, the random variable~$V_n$ corresponds to the first success time of a series of attempts, where the $k$-th attempt is successful if $x_{n+k}^{2n+k-1} = u$. 
Contrary to the case of~$W_n$, the conditional success probability given $[u]$ of the $k$-th attempt is not simply given by $\P_n(u)$ since the coordinates $x_1^n$ and $x_{n+k}^{2n+k-1}$ are not independent in general.\footnote{They are, for example, if $\P$ is a Bernoulli measure; see Remark~\ref{rem:VW-IID}.} However, by our decoupling assumptions, and since the intervals $[1,n]$ and $[n+k, 2k+k-1]$ do not overlap, this added dependence will not actually alter the asymptotics. The arguments of Proposition~\ref{prop:ptoymodel} then suggest that $V_n$ obeys the LDP with the rate function $\iv$ of \eqref{eq:defivexplicit}.

\paragraph*{Return time and overlaps.} The picture for $R_n$ is more complicated: while the dependence between $x_{k+1}^{k+n}$ and $x_1^n$ will not significantly alter our estimates for~$k$ very large, this dependence will play a major role when $k< n$, due to overlap; see the examples in \secref{sec:examplesdisc}.
We shall prove in \secref{sec:atorigin} that, under the assumptions of Theorem~\ref{thm:mainthmC},
\begin{equation}\label{eq:limgammapl}
\lim_{n\to\infty}\frac 1n \ln\P\{x:R_n(x) < n\} = \gamma_+.
\end{equation}
This suggests the following picture: with probability close to $\Exp{n\gamma_+}$ we have ``very quick return'' ($R_n < n$), and with probability close to $1-\Exp{n\gamma_+}$ we start a series of trials as in the case of $V_n$. As before, let us pretend these trials are independent, and that their probability of success is exactly $\P_n(x_1^n)$. Noting that the asymptotics of $\frac 1n \ln R_n$
is not affected by adding to~$R_n$ quantities of order~$n$, we are led to introduce a toy model $\widetilde R_n$ for $R_n$ whose law~$\rho_n$ on~$\nn$ is given by
\begin{equation}\label{eq:defrhon}
    \rho_n(k) 
        := \Exp{n\gamma_+}\delta_{k,1}+ (1-\Exp{n\gamma_+})\nu_n(k)
\end{equation}
for every~$k\in\nn$, where $\nu_n$ is as in~\eqref{eq:sumanun} with $\Q =\P$. A straightforward adaptation of Proposition~\ref{prop:ptoymodel}, taking the first term in~\eqref{eq:defrhon} into account at $s=0$, and also using that $\gamma_+ \geq \sup_{r\geq 0}(-r-\ilpp(r))$ by Theorem~\ref{thm:CJPS}.iii.a., shows that $\widetilde R_n$ satisfies the LDP with the rate function $\ir$ of~\eqref{eq:defirexplicit}.

\paragraph*{Turning the sketch into a proof.}  We now briefly comment on how the above sketch will be made rigorous in the main body of the paper. In \secref{sec:Wn-estimates}, we establish that for fixed $x$, the random variable~$W_n(x, \cdot\,)$ on $(\Omega,\Q)$ is indeed well approximated by a geometric random variable at the exponential scale. The above cut-off argument is made rigorous in Proposition~\ref{prop:main-Wn-bound}. The corresponding estimates for $R_n$ and~$V_n$ are presented in \secref{sec:Wn-to-Rn}.

The saddle-point approximation mentioned above is made rigorous by the variation of Varadhan's lemma provided in Lemma~\ref{lem:modVaradhan}, which then leads to the weak LDP for the sequences of interest. As with the above toy model, we need to treat the cases~$s > 0$ (\secref{sec:posvalues}) and $s = 0$ (\secref{sec:atorigin}) separately\,---\,the latter is particularly subtle for~$R_n$ and~\assref{pa} will be needed to establish~\eqref{eq:limgammapl}. With the weak LDPs at hand, the main results are proved in \secref{sec:full}.

\section{Examples}
\label{sec:examplesdisc}

Decoupling assumptions are satisfied by many important classes of examples, and have allowed to simplify and unify the proofs of various large deviation principles that existed in many, sometimes rather technical, forms in the literature. The range of applicability of~\assref{ud} and~\assref{sld} has already been discussed in~\cite{CJPS19,BCJPPExamples,CDEJR22,CDEJR23}.

We describe in this section some classes of examples that serve as illustrations of different features of our main results. Sections~\ref{sec:exiid}--\ref{sec:lackexpotight} each assume some level of familiarity with the specific examples on the reader's part, and may be skipped entirely without affecting the reader's ability to understand the proofs of the main theorems. We now briefly summarize the role of each of these examples.

\begin{enumerate}
    \item In the Bernoulli (IID) case (\secref{sec:exiid}), formulae for the different pressures and rate functions can be quickly derived, and are easy to understand. To the best of our knowledge, the global aspect (i.e.\ without any restriction to a strict subinterval of $\rr$) of the LDPs in Theorems~\ref{thm:mainthmA}--\ref{thm:mainthmC} as well as the ability to consider distinct measures~$\P$ and~$\Q$ in Theorem~\ref{thm:mainthmA} are new even in this most basic class of examples.
    
    \item By going from Bernoulli measures to Markov measures (\secref{sec:markov}), we start seeing the benefit of inserting the words~$\xi$ in the formulation of~\assref{sld}, as they allow to deal with irreducible Markov chains whose transition probabilities are not all positive. Markov measures also make the role of periodicity in~\assref{pa} very clear.
    
    \item Discussing our assumptions and results in the setup of equilibrium measures for potentials enjoying Bowen's regularity condition (\secref{sec:equilbowen}) allows us to compare our results to existing ones, most notably to those of~\cite{AACG22} where the pressures $\qw$ and $\qr$ were studied. 

	\item We then discuss two situations in which Bowen's regularity condition can be lifted. While carrying distinct history and intuition, they both reveal the same two aspects of our decoupling assumptions. First, they  make a case that allowing a certain amount of growth for the sequence~$(C_n)_{n\in\nn}$ in~\assref{ud} and~\assref{sld} is beneficial. Second, they show that our assumptions apply in phase-transition situations, and in particular do not imply ergodicity. These generalizations are:
	\begin{enumerate}[i.]
        \item equilibrium measures for absolutely summable interactions in statistical mechanics (\secref{sec:spinchains});
        \item equilibrium measures for $g$-functions, i.e.\ $g$-measures (\secref{sec:gmeasures}).
    \end{enumerate}
    While our results do not seem to apply to the class of {\em weak Gibbs} measures in full generality (see \secref{sec:new-g-sm}), the measures discussed in Sections~\ref{sec:spinchains} and \ref{sec:gmeasures} are weak Gibbs.

    \item Finally, the so-called class of \emph{hidden Markov models} (\secref{sec:lackexpotight}) shows that our assumptions apply to measures which are far from Gibbsian. Hidden Markov models also provide examples of pairs of measures $(\P, \Q)$ where the distributions of $(\frac 1n \ln W_n)_{n\in \nn}$ lack exponential tightness, showing that $\iw$ (see Theorem~\ref{thm:mainthmA}) need not be a good rate function when $\Q\neq \P$. To the authors' knowledge, even the conclusions of Theorem~\ref{thm:CJPS} in this setup are new. 
\end{enumerate}

As abundantly discussed in \cite{BJPP1,CJPS19,BCJPPExamples,CDEJR22,CDEJR23}, repeated quantum measurement processes give rise to a very rich class of measures satisfying our decoupling assumptions, yet displaying remarkable singularities\,---\,some being, once again, far from Gibbsian. For reasons of space, we do not repeat such a discussion in the present paper.

\subsection{Bernoulli measures}\label{sec:exiid}

Consider the simple case $\P = \P_1^{\otimes \nn}$ and $\Q = \Q_1^{\otimes \nn}$, where $\P_1$ and $\Q_1$ are measures on $\cA$, with $\P_1 \ll \Q_1$. \assref{sld}, \assref{ud} and~\assref{pa} obviously hold with $C_n = 1$ and $\tau_n = 0$ for all $n$.
Note that, as a random variable on $(\Omega, \P)$, the map $x\mapsto -\tfrac 1n \ln\Q_n(x_1^n) = \frac 1n\sum_{i=1}^n (-\ln \Q_1(x_i))$ is simply the average of IID random variables supported on the finite set $\{-\ln \Q_1(a): a\in \supp \P_1\} \subset \rr$. By independence, $\qlpq(\alpha)$ is easily seen to coincide with the cumulant-generating function:
\begin{equation}\label{eq:qlpqcgf}
    \qlpq(\alpha) = 
    \ln \int_{\cA}  \Exp{-\alpha \ln \Q_1 (a)}\dd \P_1(a)=\ln \sum_{a\in\supp\P_1} \P_1(a) \Q_1 (a)^{-\alpha}.
\end{equation}

The LDP proved in Theorem~\ref{thm:CJPS} then follows from standard results. We mention two methods which lead to different expressions of the rate function $\ilpq$.

\medskip

\noindent\textit{Method 1.} Since $\ln \Q_n(x_1^n)$ is a sum of IID random variables, Cram\'er's theorem~\cite[\S{2.2.1}]{DZ} yields the stated LDP with a rate function given by the Legendre--Fenchel transform of the cumulant-generating function~\eqref{eq:qlpqcgf}, i.e.\ with a rate function $\ilpq = \qlpq^*$. This can be seen as a special case of the G\"artner--Ellis theorem~\cite[\S{2.3}]{DZ}, which applies here since by~\eqref{eq:qlpqcgf} the function $\qlpq$ is differentiable (and actually real-analytic).

\medskip

\noindent\textit{Method 2.} 
One can instead appeal to a combination of Sanov's theorem and the contraction principle~\cite[\S{2.1.1--2.1.2}]{DZ} to obtain the LDP with rate function
\begin{equation}
\label{eq:cont-Sanov}   
    \ilpq(s) = \inf_{\mu \in L_s} H_{\mathrm{r}}(\mu|\P_1),
\end{equation}
where~$H_{\mathrm{r}}$ denotes the relative entropy and $L_s$ is the set of probability measures~$\mu\ll \P_1$ (hence also satisfying $\mu \ll \Q_1$) on~$\cA$ subject to the constraint
\begin{equation}
\label{eq:constraint-Sanov}  
    -\sum_{a\in\cA} \mu(a) \ln \Q_1(a) = s.
\end{equation}
When $s \notin [-\ln\max_{a\in \cA} \Q_1(a), -\ln \min_{a\in \supp \Q_1} \Q_1(a)]$, the set $L_s$ is empty and the infimum is set to~$\infty$ by convention.
Note that $\ilpq$ vanishes at the unique point $s = -\sum_{a\in\cA} \P_1(a) \ln \Q_1(a)= \Sc(\P|\Q)$, where the infimum in~\eqref{eq:cont-Sanov} is attained at $\mu = \P_1$.

\medskip

We now turn to Theorems~\ref{thm:mainthmB} and~\ref{thm:mainthmC}, and we discuss in particular the convexity of $\ir$. The relevant quantities are easily expressed in terms of~$\P_1$:
\begin{align*}
    \gamma_+ &= \ln \max_{a\in\supp\P_1} \P_1(a),   
        & \htop(\supp\P) &= \ln |\supp \P_1|, \\
    \gamma_- &= \ln \min_{a\in\supp \P_1} \P_1(a),   
        & \qlpp(-1) &= \ln \sum_{a\in\supp\P_1} \P_1(a)^2.
  \end{align*}
  In view of these expressions, $\gamma_- \leq
\qlpp(-1) \leq \gamma_+$ with strict inequalities \emph{unless}~$\P_1$ is constant on its support. To see this, notice that $ \sum_{a\in\cA} \P_1(a)^2$ is the expectation of the function $a\mapsto \P_1(a)$ with respect to the measure $\P_1$ on $\cA$. To discuss the convexity of $\ir$ we distinguish two cases.

\medskip 

\noindent {\em Singular case}. If $\P_1$ is constant on its support, then $\gamma_- =
\qlpp(-1) = \gamma_+$, so $\ir$ is convex by Theorem~\ref{thm:mainthmC}.iii.
Moreover, we readily obtain 
$
    h(\P) = \htop(\supp \P),
$ 
so $\P$ is indeed the measure of maximal entropy on its support, in accordance with Theorem~\ref{thm:mainthmC}.iii. Next, since $-\frac 1n \ln \P_n(x_1^n) = h(\P)$ almost surely, the rate function $\ilpp$ vanishes at~$h(\P)$ and is infinite everywhere else. Dual to this, $\qlpp(\alpha) = h(\P)\alpha$ for all $\alpha \in \rr$. In the quite extreme case where $|\supp \P_1| = 1$, i.e.\ if $\P$ is a Dirac measure on an orbit of period 1, we find $h(\P) = 0$ and $\qlpp(\alpha) = 0$ for all $\alpha\in \rr$. 

\medskip 

\noindent {\em Generic case}. If $\P_1$ is not constant on its support, then $\gamma_- < \qlpp(-1) < \gamma_+$, and thus the rate function $\ir$ is nonconvex by Theorem~\ref{thm:mainthmC}.iii. In the present setup, it is easy to understand why $\ir(0) < \iv(0)$, as we now discuss. Let $\epsilon$ be small and $n$ be large. On the one hand, if $\hat a\in \cA$ is such that $\ln \P_1(\hat a) = \gamma_+$, we find $R_n(x) = 1$ for all $x\in [\hat a^{n+1}]$, so
\begin{equation}\label{eq:Pberlefthand}   
\begin{split}
    \P\{x : R_n(x) \leq \Exp{\epsilon n}\} 
        &\geq \P\left\{x : R_n(x) = 1 \right\} \\
        &\geq \P_{n+1}(\hat a^{n+1}) \\
        &=  \exp\left((n+1) \gamma_+\right).
\end{split}
\end{equation}
On the other hand, $V_n(x) = k$ implies that $x\in [u]\cap \shift^{-n-k+1}[u]$ for some $u\in \cA^n$, so
\begin{align*}
    \P\{x : V_n(x) = k\}  &\leq \sum_{u\in\cA^n} \P([u] \cap \shift^{-n-k+1}[u]) = \sum_{u\in\cA^n} \P_n(u)^2 \\
    & = \Big(\sum_{a\in \cA}\P_1(a)^2\Big)^n = \exp\left(n \qlpp(-1)\right),
\end{align*}
where we have used~\eqref{eq:qlpqcgf} with $\Q=\P$.
Thus,
\begin{equation}\label{eq:Pberotherhand}
    \P\{x : V_n(x) \leq \Exp{\epsilon n}\} \leq   \sum_{k=1}^{\lfloor\Exp{\epsilon n}\rfloor} \P\{x : V_n(x) = k\} \leq  \exp\left(n\epsilon + n \qlpp(-1)\right).
\end{equation}
Therefore, for $\epsilon$ small enough, the right-hand side of~\eqref{eq:Pberlefthand} decays exponentially faster than the right-hand side of~\eqref{eq:Pberotherhand} since $\qlpp(-1)<\gamma_+$. 
The estimates~\eqref{eq:Pberlefthand} and~\eqref{eq:Pberotherhand}, despite the rather crude inequalities in~\eqref{eq:Pberlefthand}, turn out to be sharp at the exponential scale. Indeed, we find $\iv(0) = -\qlpp(-1)$ and $\ir(0) = -\gamma_+$ in Theorems~\ref{thm:mainthmB} and~\ref{thm:mainthmC}.

\begin{remark}
\label{rem:VW-IID}
    For Bernoulli measures, in the case $\Q=\P$, $W_n$ and $V_n$ have the same law.  
\end{remark}

We provide three figures corresponding to Bernoulli measures, the first two of which were displayed in \secref{sec:mainresults}:
\begin{itemize}
    \item Figure~\ref{fig:BernoulliPQ}: $\cA=\{\textsf{0}, \textsf{1}, \textsf{2}\}$ and $\P_1 = 0.2\,\delta_{\textsf{0}}+0.3\,\delta_{\textsf{1}}+0.5\,\delta_{\textsf{2}}$ and $\Q_1 = 0.6\,\delta_{\textsf{0}}+0.3\,\delta_{\textsf{1}}+0.1\,\delta_{\textsf{2}}$. 
    \item Figure~\ref{fig:Bernoullisimple}:  $\cA=\{\textsf{0}, \textsf{1}\}$ and $\P_1 = \Q_1 = 0.3\,\delta_{\textsf{0}}+0.7\,\delta_{\textsf{1}}$.
    \item Figure~\ref{fig:Bernoullidegenerate}: $\cA=\{\textsf{0}, \textsf{1}, \textsf{2}\}$ and  $\P_1 = \Q_1 = 0.15\,\delta_{\textsf{0}}+0.15\,\delta_{\textsf{1}}+0.7\,\delta_{\textsf{2}}$. Notice that $\ln \P_1(\textsf{0}) = \ln \P_1(\textsf{1}) = -\gamma_-$. Such degeneracy implies that the second term in~\eqref{eq:energyentropy} takes the value $\ln 2$  for $s=-\gamma_-$. Indeed, when $\epsilon$ is very small, the cardinality in~\eqref{eq:energyentropy} is essentially $2^n$. As a consequence, $\ilpp(-\gamma_-) = -\gamma_--\ln 2$. A similar phenomenon occurs at~$-\gamma_+$ if several letters in $\cA$ have maximum probability.
\end{itemize}

\begin{figure}[h]
    \centering
    \begin{tabular}{p{0.45\textwidth} p{0.45\textwidth}}
        \vspace{0pt} \includegraphics[scale=\figscale]{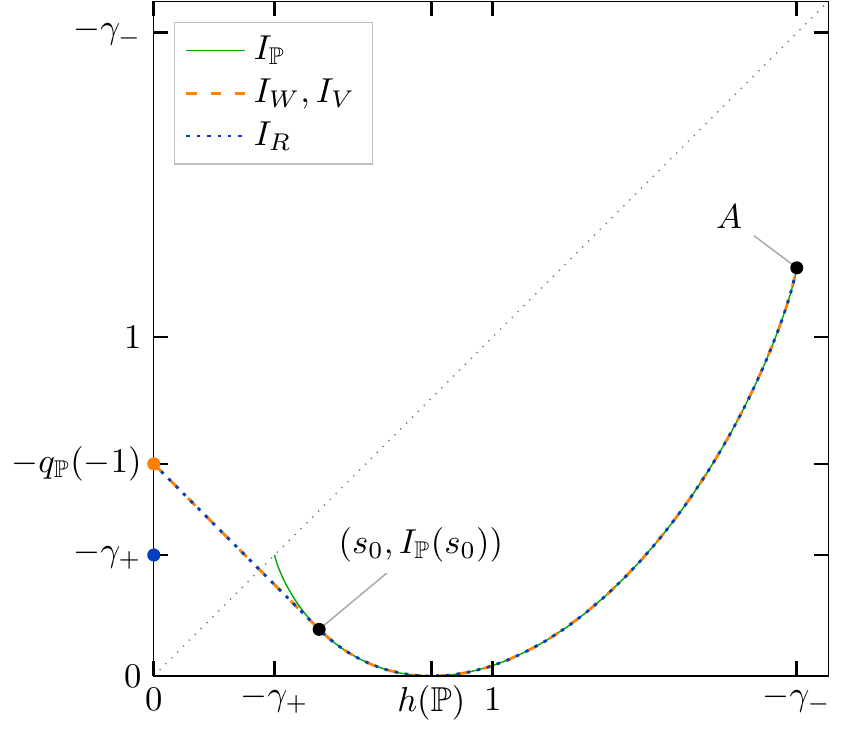} &
        \vspace{0pt} \includegraphics[scale=\figscale]{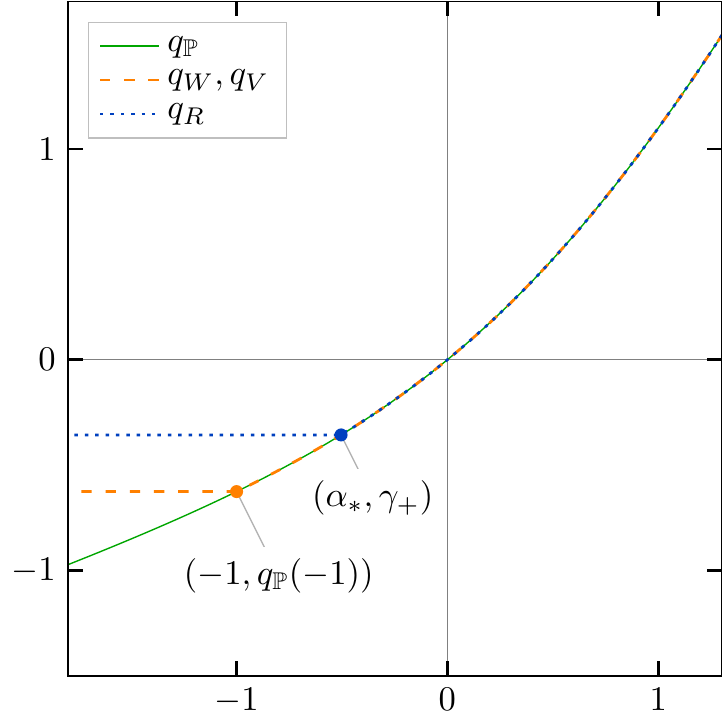}
    \end{tabular}
    \caption{Rate functions and pressures for a Bernoulli measure with degenerate least probable letter. The coordinates of the point $A$ are $(-\gamma_-, -\gamma_--\ln2)$. Rate functions are infinite wherever not drawn.}
    \label{fig:Bernoullidegenerate}
\end{figure}

\subsection{Irreducible Markov measures}
\label{sec:markov}

Let $\P$ and~$\Q$ be two stationary Markov measures on $\cA$, with transition matrices $P = [P_{i,j}]_{i,j\in \cA}$ and $Q = [Q_{i,j}]_{i,j\in \cA}$ respectively. We assume that the matrices $P$ and $Q$ are {\em irreducible} in the sense that there exists $M\in \nn$ such that all entries in the matrices $\sum_{i=1}^M P^i$ and $\sum_{i=1}^M Q^i$ are positive. Then,
\begin{equation}\label{eq:markovdefpn}
\P_n(x_1^n) = \P_1(x_1)\prod_{k=1}^{n-1}P_{x_k, x_{k+1}}
\qquad\text{and}\qquad
\Q_n(x_1^n) = \Q_1(x_1)\prod_{k=1}^{n-1}Q_{x_k, x_{k+1}},
\end{equation} 
where~$\P_1$ and~$\Q_1$ are the (unique and fully supported) invariant probability vectors for the matrices~$P$ and~$Q$ respectively. We assume, furthermore, that $Q_{i,j} = 0$ implies $P_{i,j} = 0$, which ensures that $\P_n \ll \Q_n$ for all~$n\in\nn$. It follows from Lemma~A.3 in~\cite{CJPS19} that~\assref{jsld} and~\assref{ud} hold with $C_n=O(1)$ and $\tau_n=M-1$, so the pair $(\P, \Q)$ is admissible. Similar arguments also yield~\assref{pa}. It is worth noting that if~$P$ is further assumed to be {\em aperiodic}, i.e.~if there exists $M\in \nn$ such that all entries in the matrix $P^M$ are positive, then we can fix $\ell  =\tau_n=M-1$ in~\eqref{eq:SLDeq}, whereas if~$P$ is merely irreducible, then~$\ell$ must be allowed depend on~$u$ and~$v$. This illustrates the importance of the condition $\ell \leq \tau_n$ in~\assref{sld} instead of $\ell = \tau_n$; see the discussion in~\cite[\S\S{2.5;A.1}]{CJPS19}. 

Our results thus apply to the setup described here. As in the previous example, the conclusions of Theorem~\ref{thm:CJPS} can be derived using classical methods, which yield natural expressions for $\ilpq$ and $\qlpq$, as we now show. 

\medskip

\noindent\textit{Method 1.} In view of~\eqref{eq:markovdefpn}, the pressure is easily seen to be given by
$$
    \qlpq (\alpha) = \ln \operatorname{spr}(M(\alpha)),
$$
where the spectral radius is computed for the deformation of the stochastic matrix~$P$ defined by $M_{i,j}(\alpha) := P_{i,j}Q_{i,j}^{-\alpha}$. By the Perron--Frobenius theorem and analytic perturbation theory, $\alpha \mapsto \qlpq (\alpha)$ is real-analytic, so the LDP for $(-\frac 1n\ln \Q_n)_{n\in\nn}$ with the rate function $\ilpq = \qlpq^*$ follows from the G\"artner--Ellis theorem~\cite[\S{2.3}]{DZ}. 

\medskip

\noindent\textit{Method 2.} From~\eqref{eq:markovdefpn}, we see that 
\begin{equation*}
    \ln \Q_n(x_1^n) =\sum_{k = 1}^{n-1} \ln Q_{x_k, x_{k+1}} + O(1)
\end{equation*}
for all $x\in \supp \P \subseteq \supp \Q$, so the LDP for $(-\frac 1n\ln \Q_n)_{n\in\nn}$ reduces to that of a sequence of Birkhoff sums. 
By Sanov's theorem applied to the pair empirical measures
${\cal E}_n(x) := \frac 1n \sum_{k=1}^n \delta_{(x_k, x_{k+1})}\in \cP(\cA^2)$ and the contraction principle, one derives the expression
\begin{equation}\label{eq:infilpqmarkov}
    \ilpq(s) = \inf_{\mu \in L_s^{(2)}}  H_{\mathrm{r}}(\mu|\mu_1 \otimes P),
\end{equation}
where $L_s^{(2)}$ is the set of probability measures $\mu$ on~$\cA\times\cA$ such that $\mu \ll \mu_1\otimes P$, $\mu_1 = \mu_2$, and
$$
    \sum_{i,j\in \cA}\mu (i,j)\ln Q_{i,j} = -s;
$$ 
see \cite[\S{3.1.3}]{DZ}\footnote{There it is assumed that $P_{i,j}>0$ for all $i,j\in \cA$.} or Lemma~4.49 in~\cite{DS1989}.
Here, $\mu_1$ (resp.~$\mu_2$) denotes the first (resp.~second) marginal of~$\mu$, and $\mu_1\otimes P$ denotes the measure on $\cA\times\cA$ defined by $(\mu_1\otimes P) (i,j) = \mu_1(i)P_{i,j}$. Note that $\ilpq$ vanishes at the single point $s = -\sum_{i,j\in \cA}\P_2(i,j)\ln Q_{i,j} =  \Sc(\P|\Q)$, where the infimum in~\eqref{eq:infilpqmarkov} is attained at $\mu = \P_2$.

\medskip

We now discuss $\ir$ and its convexity, with $\P=\Q$ a Markov measure as above. It is well known that 
\begin{equation}\label{eq:optimizecycle}
    \gamma_+ = \max_{p\leq |\cA|} \max_{u\in\cA^p} \frac 1p \ln \prod_{k=1}^{p} P_{u_k,u_{k+1}},
\end{equation}
with cyclic identification $u_{p+1} = u_1$; see e.g.\ Remark~1.ii in~\cite{Sz93} or~\cite[\S{3.4}]{AACG22}. 
The heuristic interpretation is the following: while in the generic IID case discussed in Section~\ref{sec:exiid} the probability $\P\{x: R_n(x) \leq \Exp{\epsilon n}\}$ was, at exponential scale, captured by the subset $\{x: x_1^{n+1} = {\hat{a} \hat{a} \hat{a}}\dotsb{\hat{a} \hat{a}}\}$ of the event~$\{x: R_n(x) = 1\}$ for $n$ large and $\epsilon$ small, now $\P\{x: R_n(x) \leq \Exp{\epsilon n}\}$ is essentially accounted for by the subset $\{x: x_1^{n+p} = u u u \dotsb u u_1^{r}\}$ of the event~$\{x: R_n(x) = p\}$, for any $p$ and $u$ that saturate~\eqref{eq:optimizecycle}, and $r = n - \lfloor \frac np \rfloor p$. In other words, the key scenario for small return times consists of a periodic orbit repeating some optimal cycle. 

By the last part of Theorem~\ref{thm:mainthmC}, the function $\ir$ is convex if and only if $\P$ is the measure of maximal entropy on $\supp \P$, that is if and only if $\P$ is the Parry measure of $\supp \P$, which is characterized by the Perron--Frobenius data of the adjacency matrix of the chain~\cite{Pa64}; see also~\cite[\S{2}]{CGS99}. In particular, if $P_{i,j}>0$ for all $i,j\in \cA$, then $\ir$ is convex if and only if $\P = \pi^{\otimes\nn}$ with $\pi$ uniform on~$\cA$.

\begin{remark}
    One can show that irreducible multi-step Markov measures also satisfy the assumptions in the present paper. In fact, Markov measures and multi-step Markov measures are merely special cases of the equilibrium measures discussed in \secref{sec:equilbowen}. For example, in the notation of \secref{sec:equilbowen}, one takes $\potone(x) = \ln P_{x_1, x_2}$ in the case of Markov measures. An explicit computation of $\qlpp$ for a specific Markov chain is provided in \cite{AACG22}.
\end{remark}

\subsection{Equilibrium measures for Bowen potentials}\label{sec:equilbowen}

As mentioned above, our assumptions cover the setups of~\cite{CGS99} and~\cite{AACG22}, where the large-deviation results are local in the sense that there is some strict subinterval $A \subset \rr$ such that the large-deviation lower bound and upper bound (see~\eqref{eq:LD-LB} and~\eqref{eq:LD-UB} below) are only shown to be valid respectively for open sets~$O$ contained in~$A$ and closed sets~$\Gamma$ contained in~$A$. 
The former work considers a single measure~$\P$ that is the equilibrium measure for a H\"older-continuous potential on a topologically mixing Markov subshift; the latter work, for a potential of summable variations on the full shift. 

In particular, Theorem~\ref{thm:mainthmA}.i and Theorem~\ref{thm:mainthmC}.i prove the following conjecture stated in \cite[\S{3.3}]{{AACG22}}:
\begin{quote}
    \textit{We believe that there exists a non-trivial rate function describing the large deviation asymptotic} [on the whole real line] \textit{for both return and waiting times, but this has to be proven using another method.}
\end{quote}
Moreover, the results of~\cite[\S{3.2}]{AACG22} are recovered as special cases of Theorem~\ref{thm:mainthmA}.ii and Theorem~\ref{thm:mainthmC}.ii. In order to facilitate the translation, we compare notations in Table~\ref{tab:aacg-dict}. Note also that what we refer to as pressure is called ``$L^q$-spectrum'' in \cite{AACG22}.
\begin{table}[h]
    \centering
    \begin{tabular}{l c c c c c c c}
        \toprule
        & \multicolumn{3}{c}{Pressures} & Variable & \multicolumn{3}{c}{Important points} \\
        \midrule
        Present paper & $\qlpp$ & $\qw$ & $\qr$ & $\alpha$ & $\alpha_*$ & $\gamma_+$ & $\gamma_-$ \\
        \cite{AACG22} & ${\cal M}_\varphi$ & ${\cal W}_\varphi$ & ${\cal R}_\varphi$ & $q$ & $q^*_\varphi$ & $\gamma_\varphi^+$ & $-v_\varphi^+$  \\
        \bottomrule
    \end{tabular}
    \caption{A summary of the notational changes between our results and those of~\cite{AACG22}. }
    \label{tab:aacg-dict}
\end{table}

We first consider the class of \emph{Bowen-regular} potentials on the full shift, and then discuss the extension to some subshifts (including those of~\cite{CGS99}) in Remark~\ref{rem:specification}. We recall that a potential $\potone$, i.e.\ a continuous function~$\potone:\Omega\to \rr$, is called Bowen regular if 
\begin{equation}\label{eq:bownreg}
    \sup \left\{ \left|\sum_{i=0}^{n-1}\potone(\shift^{i}x)- \sum_{i=0}^{n-1} \potone(\shift^{i}y)\right|: x \in \Omega, n\in\nn, y \in [x_1^n] \right\} < \infty,
\end{equation}
and that this class strictly contains the class of potentials with summable variations, which in turn strictly contains the class of H\"older-continuous potentials; we refer the reader to~\cite{Bo74} and~\cite[\S{4}]{Wa01} for a thorough discussion. 
We also recall that the topological pressure of any potential $\varphi$ is given by
\begin{equation}\label{eq:defptop}
    \ptop(\varphi) := \lim_{n\to\infty}\frac 1n \ln \sum_{u \in \cA^n}\Exp{\sup_{x \in [u]} \sum_{i=0}^{n-1}  \varphi(\shift^{i}x)}.
\end{equation}
Suppose that $\P$ and~$\Q$ are the (necessarily unique \cite{Bo74}) equilibrium measures for the Bowen-regular potentials $\potone$ and $\pottwo$ on~$\Omega$ in the sense that they belong to~$\cP_{\textnormal{inv}}(\Omega)$ and satisfy
\begin{equation}\label{eq:defequilibrium}
    h(\P) + \int \potone \dd\P = \ptop(\potone)
    \qquad\text{and}\qquad
    h(\Q) + \int \pottwo \dd\Q = \ptop(\pottwo),
\end{equation}
respectively.
The measure~$\P$ then satisfies the \emph{Bowen--Gibbs property} with respect to~$\potone$, i.e.\ there exists a constant $K\geq 1$ such that
\begin{equation}
\label{eq:Gibbsprop}
    K^{-1} \Exp{\sum_{i=0}^{n-1} \potone(\shift^{i}x)-n\ptop(\potone)}\leq \P([x_1^n]) \leq K \Exp{\sum_{i=0}^{n-1} \potone(\shift^{i}x)-n\ptop(\potone)},
\end{equation}
for every~$x \in \Omega$~\cite[\S{4}]{Wa01}, and one then deduces that it satisfies~\assref{ud} and~\assref{sld} with $\tau_n = 0$ and $C_n = K^3$. The same is true for~$\Q$ with~$\pottwo$, so~\assref{jsld} and~\assref{pa} follow from Remarks~\ref{rem:tau0sld} and~\ref{rem:tau0pa} respectively, and thus the pair $(\P, \Q)$ is admissible.  We note that \eqref{eq:Gibbsprop} is central in the analysis in~\cite[\S{2.1}]{AACG22}.
In order to simplify some formulae, we assume for the remainder of this subsection that
\begin{equation}\label{eq:ptopzero}
    \ptop(\potone) = \ptop(\pottwo) = 0,
\end{equation}
which results in no loss of generality since adding constants to~$\potone$ and~$\pottwo$ does not alter the set of equilibrium measures.

In the setup of the present subsection, the conclusions of Theorem~\ref{thm:CJPS} are well known and can be obtained more directly as follows.

\medskip

\noindent\textit{Method 1.} The bounds~\eqref{eq:Gibbsprop} imply that
\begin{equation}\label{eq:qlppbowen}
\qlpq(\alpha) = \ptop(\potone-\alpha \pottwo).
\end{equation}
In particular, we remark that $\qlpp(-1) = \ptop(2\potone)$; this quantity plays an important role in the formula for $\qw$, and is equal to $-\iw(0)$. In this setup, $\qlpq$ is differentiable and, for all $\alpha \in \rr$,
$$
    \qlpq'(\alpha) = -\int \pottwo \dd\mu_\alpha,
$$
where $\mu_\alpha$ is the equilibrium measure for the Bowen-regular potential $\potone-\alpha \pottwo$; see e.g.~Theorems~4.3.3 and 4.3.5 in \cite{Ke98}. The LDP of Theorem~\ref{thm:CJPS} then follows from the G\"artner--Ellis theorem.

\medskip

\noindent\textit{Method 2.} The same LDP can be obtained by noticing that, in view of \eqref{eq:Gibbsprop}, the large deviations of $(-\frac 1n \ln \Q_n)_{n\in\nn}$ are the same as those of the ergodic averages of~$-\pottwo$ with respect to $\P$, which is a well-studied problem; see e.g.~\cite{Yo90,Ki90,Co09,PS-2018}.
The rate function is then given, for all $s\in \rr$, by
\begin{equation}\label{eq:varilpqgibbs}
    \ilpq(s) =  -\sup\left\{ \int \potone \dd\eta +h(\eta) : \eta\in \cP_{\textnormal{inv}}, \int \pottwo \dd\eta = -s \right\}.
\end{equation}
The rate function $\ilpq$ vanishes at a single point $s = -\int \pottwo \dd \P =  \Sc(\P|\Q)$; the supremum in~\eqref{eq:varilpqgibbs} is then reached at $\eta = \P$.

\medskip 

For all $s\geq0$, the relations~\eqref{eq:defiwexplicit} and~\eqref{eq:defivexplicit} give
\begin{align*}
    \iw(s) & = -s-\sup\left\{\int( \potone+\pottwo )\dd\eta  + h(\eta):  \eta\in \cP_{\textnormal{inv}},\int \pottwo \dd\eta  \leq  -s \right\},\\
    \iv(s) 
    &=  -s-\sup\left\{2 \int \potone \dd\eta+h(\eta): \eta\in \cP_{\textnormal{inv}},  \int \potone \dd\eta \leq -s \right\}.
\end{align*}
When $\Q=\P$, we obtain $\iw=\iv$ and, for all $s\in \rr$,
$$\ilpp(s) =  s-\sup\left\{h(\eta):  \eta\in \cP_{\textnormal{inv}}, \int \potone \dd\eta = -s \right\}.$$

Next, by~\eqref{eq:defirexplicit}, $\ir$ differs from $\iv$ only at 0, with
\begin{equation*}
    \ir(0) = -\gamma_+ = - \sup_{\eta \in \cP_{\textnormal{inv}}}\int \potone \dd \eta;
\end{equation*}
the second identity is proved~\cite[\S{4.2}]{AACG22}. 

The very last assertion of Theorem~\ref{thm:mainthmC} applies, and $\ir$ is convex if and only if~$\P$ is the measure of maximal entropy on~$\Omega$ (note that $\supp \P = \Omega$ by \eqref{eq:Gibbsprop}), i.e.\ if $\P$ is the uniform measure. Equivalently, in terms of potentials, $\ir$ is convex if and only if~$\potone$ is cohomologous to a constant~\cite[\S{3.2}]{AACG22}.

The authors of \cite{AACG22} first derive the expression~\eqref{eq:formqW} for~$\qw$ and the expression~\eqref{eq:formqV} for~$\qr$. Then, since under the assumptions at hand, the pressure~$\qw$ is differentiable on the domain where it is equal to $\qlpp$, a version of the G\"artner--Ellis theorem implies the LDP for $(\frac 1n \ln W_n)_{n\in\nn}$ restricted to the interval $[\qlpp'(-1), -\gamma_-]$, i.e.\ where $\ilpp$ and $\iw$ coincide and are finite.  
In the same way, $(\frac 1n \ln R_n)_{n\in\nn}$ is shown to obey the LDP restricted to the interval
$[\qlpp'(\alpha_*), -\gamma_-]$, i.e.\ where $\ilpp$ and the convex envelope of~$\ir$ coincide and are finite; see e.g.\ Figure~\ref{fig:Bernoullisimple}. By construction, the approach of~\cite{AACG22} cannot describe the LDP on the interval $(0, \qlpp'(-1))$ (resp.\ $(0, \qlpp'(\alpha_*))$), and in particular cannot capture the nonconvexity of $\ir$, since $\ir$ is not the Legendre--Fenchel transform of $\qr$ in general.

As mentioned in the introduction, our method is very different in spirit. In addition to the explicit singularities in \eqref{ex:formiwqwm1} and \eqref{eq:formirqrm1}, the pressures suffer from the fact that even $\qlpq$ may fail to be differentiable under our decoupling assumptions. As a consequence, the  G\"artner--Ellis theorem cannot be used to obtain the LDP, even limited to the above-mentioned intervals.  We are able to circumvent these limitations by going in the opposite direction: we first establish the LDPs directly, using the Ruelle--Lanford method, and then we obtain the properties of the pressures as corollaries. 

While~\cite{CGS99} does not rely on any version of the G\"artner--Ellis theorem, the LDP there is still restricted to a nonexplicit interval, which is contained in that of~\cite{AACG22}. More precisely, the LDP is restricted to an interval of values of $s$ on which $\ilpp(s)$ is small enough so that some exponentially decaying error terms that arise in the proof decay faster than $\Exp{-n\ilpp(s)}$.

\begin{remark}\label{rem:specification}
    The above discussion and the results in~\cite{AACG22} are limited to the full shift~$\Omega$, while~\cite{CGS99} discusses Markov subshifts~$\Omega'$ that are topologically mixing. However, the above conclusions extend in a straightforward way to any subshift~$\Omega' \subseteq \Omega$ satisfying the \emph{flexible specification property} of Definition~\ref{def:specif} with $\tau_n=O(1)$, and in particular to any transitive subshift of finite type; see Remark~\ref{rem:specification-sld}. 
    Indeed, in this context, one can show that if~$\P$ and~$\Q$ are equilibrium measures for two potentials~$\potone$ and~$\pottwo$ on~$\Omega'$ satisfying Bowen's condition, i.e.~\eqref{eq:bownreg} with $x$ and $y$ restricted to $\Omega'$ in the supremum, then the Bowen--Gibbs property~\eqref{eq:Gibbsprop} holds for both measures and for all $x\in \Omega'$; see Remark 2.2 in \cite{climenhaga_equilibrium_2013} and \cite[\S{6.5}]{climenhaga_unique_2016} (with ${\cal G}=\cL$ in the notation therein). This in turn implies~\assref{jsld} and~\assref{pa} through Lemmas~\ref{lem:bded-spec-implies-per-spec}, \ref{lem:specud} and~\ref{lem:specsld}, once we have extended~$\P$ and~$\Q$ to measures on $\Omega$ by setting $\P(\Omega\setminus \Omega')=\Q(\Omega\setminus \Omega') = 0$.
    Thus, our results apply, and the above expressions for the rate functions and pressures remain valid. In particular, $\ir$ is convex if and only if $\P$ is the measure of maximal entropy on $\Omega'$.

    We emphasize that the restriction to $\tau_n = O(1)$ is important in two ways in the above argument: first, in order to apply Lemma~\ref{lem:bded-spec-implies-per-spec}, and second, in order to derive \eqref{eq:Gibbsprop}, on which the hypotheses~\eqref{eq:pnpmxud} and~\eqref{eq:pnpmx} of Lemmas~\ref{lem:specud} and~\ref{lem:specsld} rely in the present setup.
\end{remark}

\subsection{Beyond Bowen potentials: statistical mechanics and \textit{g}-measures}
\label{sec:new-g-sm}

The analysis in \secref{sec:equilbowen} relies heavily on the Bowen--Gibbs property~\eqref{eq:Gibbsprop}, which is  obtained as a consequence of the (quite restrictive) Bowen condition \eqref{eq:bownreg} imposed on the potentials~$\potone$ and~$\pottwo$. 

For the coming discussion, we introduce a weaker version of~\eqref{eq:Gibbsprop}. We say that $\P$ is {\em weak Gibbs} for the potential $\potone$ if there exists an $\Exp{o(n)}$-sequence $(K_n)_{n\in \nn}$ such that for all $x\in \Omega$,
\begin{equation}
\label{eq:weakGibbsprop}
    K^{-1}_n \Exp{\sum_{i=0}^{n-1} \potone(\shift^{i}x)-n\ptop(\potone)}\leq \P_n(x_1^n) \leq K_n \Exp{\sum_{i=0}^{n-1} \potone(\shift^{i}x)-n\ptop(\potone)}.
\end{equation}
The notion of weak Gibbs measure was introduced in \cite{Yu02}, and in many interesting situations, equilibrium measures for non-Bowen potentials can still be shown to be weak Gibbs, see for example \cite{PS20}. 

The conclusions of Theorem~\ref{thm:CJPS} remain valid if $\P$ and $\Q$ are assumed to be weak Gibbs for $\potone$ and $\pottwo$ respectively; indeed, in view of~\eqref{eq:weakGibbsprop}, the LDP of Theorem~\ref{thm:CJPS} again boils down to the LDP for the ergodic averages of $\pottwo$ with respect to $\P$, which is a well-studied problem; see e.g.~\cite[\S{5}]{Co09}, 
\cite{Va12}, \cite{PS-2018}, and \cite[\S{A.3}]{CJPS19}, as well as~\cite{EKW-1994} in the specific setup of \secref{sec:spinchains} and~\cite{CO00} in the specific setup of \secref{sec:gmeasures} below.

However, the weak Gibbs condition does not seem to imply~\assref{ud} and~\assref{sld}; see Remark~\ref{rem:weakGibbsdeco}.
To the best of the authors' knowledge, the LDP for $(\tfrac 1n \ln W_n)_{n\in\nn}$, $(\tfrac 1n \ln V_n)_{n\in\nn}$ and $(\tfrac 1n \ln R_n)_{n\in\nn}$ as well as the validity of~\eqref{eq:formqW},~\eqref{eq:formqV} and~\eqref{eq:formqR} are open problems for weak Gibbs measures.

We discuss in Sections~\ref{sec:spinchains} and~\ref{sec:gmeasures} below two important classes of measures enjoying the weak Gibbs property, which are shown to satisfy our decoupling assumptions, using specific arguments distinct from those in \secref{sec:equilbowen}. Although they have a large intersection, these two classes are distinct; see~\cite{FGM11,BEEL18}. 

Contrary to the regularity conditions in \secref{sec:equilbowen}, the setups of Sections~\ref{sec:spinchains} and~\ref{sec:gmeasures} allow for phase transitions: the equilibrium measures may fail to be unique and ergodic. Also, the pressure~$\qlpq$ may fail to be differentiable, so the conclusions of Theorem~\ref{thm:CJPS} cannot be obtained using the G\"artner--Ellis theorem anymore.

\subsubsection{Absolutely summable interactions in statistical mechanics}
\label{sec:spinchains}

An important situation where less regular potentials arise is the statistical mechanics of one-dimensional, translation-invariant systems; see e.g.~\cite[Ch.\,3--5]{Rue} or~\cite[Ch.\,II--III]{Sim}. Indeed, if $\Phi := \{\Phi_X\}_{X \Subset \zz}$ is a translation-invariant collection of functions~$\Phi_X : \cA^X \to \rr$ (interchangeably considered as functions on~$\cA^{\zz}$), called \emph{interactions}, satisfying the absolute summability condition
\begin{equation}\label{eq:summableint}
    \sum_{\substack{X \Subset \zz \\ 0 \in X}}\|\Phi_X\|_\infty < \infty,
\end{equation}
then there is a well-known correspondence between translation-invariant Gibbs states\,---\,either defined using the \emph{Dobrushin--Lanford--Ruelle equations}, or using convex combinations of weak limits of finite-volume Gibbs measures\,---\,and equilibrium measure (on $\cA^{\zz}$) for the ``energy per site'' potential
\[
  \potone = \sum_{\substack{X\Subset \zz \\ \min X =1}} \Phi_X,
\]
which we see both as a function on $\cA^\zz$ and on $\Omega = \cA^\nn$.\footnote{By our choice of summation along $X \Subset \zz$ such that $\min X = 1$, the function $\potone$ only depends on the positive coordinates of $x\in \cA^\zz$. This convention is equivalent to the more common choice of summing over $X \Subset \zz$ such that $0 \in X$, taking care to divide each term by the cardinality of~$X$; see e.g.~\cite[\S{3.3}]{Rue}.}

Let now $\P$ (resp.\ $\Q$) be the marginal on~$\Omega$ of some translation-invariant Gibbs state on $\cA^\zz$ for the absolutely summable interaction~$\Phi$ (resp.\ $\Psi$). Equivalently, $\P$ (resp.\ $\Q$) is an equilibrium measure for the potential $\potone$ (resp.\ $\pottwo$) on $\Omega$.

Since $\potone$ and $\pottwo$ may fail to satisfy Bowen's regularity condition, we cannot use the Bowen--Gibbs property~\eqref{eq:Gibbsprop}. However, the Dobrushin--Lanford--Ruelle equations allow to prove that~$\P$ and~$\Q$ satisfy~\assref{ud} and~\assref{sld} with $\tau_n = 0$, but with a possibly unbounded\footnote{A notable case where $C_n = O(1)$ is when the interactions have finite range; in this case also Bowen's regularity condition and~\eqref{eq:Gibbsprop} hold. We refer the reader to~\cite[Ch.\,5]{Rue}.}  sequence $(C_n)_{n\in\nn}$;
see Lemma~9.2 in \cite{LPS95}. \assref{jsld} and~\assref{pa} then follow from Remarks~\ref{rem:tau0sld} and \ref{rem:tau0pa}. Our main results thus apply, and we now identify some of the quantities at play in physical terms.

The associated topological pressure $\ptop(\potone)$ can be thought of as a free energy density: with 
\[
    U_n := \sum_{X \subseteq [1,n]} \Phi_X
\] 
the Hamiltonian (up to a factor of minus the inverse temperature) corresponding to $\Phi$ in the finite volume $[1,n]$ with free boundary conditions, we have
\begin{equation}
\label{eq:Ham-to-pot}
    U_n = \sum_{i=0}^{n-1} \potone \circ \shift^i + o(n), 
\end{equation}
so
\[
    F(\Phi) := \lim_{n\to\infty} \frac 1n \ln \sum_{u\in \cA^n}\Exp{U_n(u)} = \ptop(\potone).
\]
Using the Dobrushin--Lanford--Ruelle equations, one can show that 
\begin{equation}\label{eq:gibbsnZ}
    \P_n(u) = \Exp{U_n(u)-nF(\Phi) + o(n)},
\end{equation}
which, combined with~\eqref{eq:Ham-to-pot}, shows that~$\P$ is weak Gibbs in the sense of~\eqref{eq:weakGibbsprop}. The same is true of~$\Q$.
To simplify the discussion, we shall assume going forward that
\begin{equation}\label{eq:vanishfreenergy}
F(\Phi) = F(\Psi) = 0,
\end{equation}
which can be achieved by adding suitable constants to $\Phi_{\{i\}}$ and $\Psi_{\{i\}}$ for every $i \in \zz$.

The expressions for $\qlpq$ and the rate functions obtained in \secref{sec:equilbowen} remain valid, as replacing~\eqref{eq:Gibbsprop} with~\eqref{eq:weakGibbsprop} does not affect these computations (note that~\eqref{eq:vanishfreenergy} implies~\eqref{eq:ptopzero}). There, $\int \potone \dd\eta$ and $\int \pottwo \dd\eta$ are simply understood as the specific energies of the state $\eta$. 

We then obtain, either by~\eqref{eq:qlppbowen} or by direct computation using~\eqref{eq:gibbsnZ}, the relations
$$
    \qlpq(\alpha) = F(\Phi-\alpha \Psi)
    \qquad\text{and}\qquad 
    \qlpp(\alpha) = F((1-\alpha)\Phi),
$$
valid for all $\alpha \in \rr$. We remark that $\qlpp(\alpha)$ is related to the free energy density at (minus the) inverse temperature $1-\alpha$.
Moreover, in this setup, $\gamma_+$ takes the form of (minus) the asymptotic ground-state energy per unit volume:
\begin{equation}\label{eq:gammaplusground}
    \gamma_+ = \lim_{n\to\infty}\frac 1n \sup_{u\in \cA^n}U_n(u).
\end{equation}

The summability condition \eqref{eq:summableint} allows for phase transitions, i.e.\ the coexistence of several equilibrium measures, and thus the measures $\P$ and $\Q$ may fail to be ergodic, while still satisfying our decoupling assumptions. As a consequence, $\qlpq$ is not differentiable in general.

\begin{remark}\label{rem:statmech}
        Again, we have limited the above discussion to the full shift $\Omega$, but the conclusions remain true on many subshifts. In particular, they hold if the subshift satisfies the flexible specification property of Definition~\ref{def:specif} with $\tau_n=O(1)$. Indeed, although the technical details are slightly involved, one can then adapt the arguments of Lemma~9.2 in~\cite[\S{9}]{LPS95} in order to obtain \eqref{eq:pnpmxud} and \eqref{eq:pnpmx}, so that Lemmas~\ref{lem:specud} and~\ref{lem:specsld}, together with Lemma~\ref{lem:bded-spec-implies-per-spec}, establish \assref{ud}, \assref{sld}, \assref{jsld} and \assref{pa}. In particular, this applies to any transitive subshift of finite type.
\end{remark}

\subsubsection{$g$-measures}\label{sec:gmeasures}

A continuous function~$g: \Omega \to (0,1]$ is called a $g$-function on~$\Omega$ if 
\begin{equation*}
    \sum_{y\in \shift^{-1}\{x\}} g(y) = 1 
\end{equation*}
for all $x \in \Omega$. In this case, the potential $\potone = \ln g$ has vanishing topological pressure ($\ptop(\potone)=0$), and any equilibrium measure (recall~\eqref{eq:defequilibrium}) for the potential $\potone$ is a called a \emph{g-measure} on~$\Omega$; see e.g.~\cite{Wa75,PPW78,Wa05}. Let $\P$ be such an equilibrium measure. It is then well known that $\supp \P = \Omega$ and that
\begin{equation}\label{eq:defgnfullshift}
    g_n(x) := \frac{\P_n(x_1^n)}{\P_{n-1}(x_2^n)}
\end{equation}
defines a sequence of continuous functions that converges uniformly to~$g$ on~$\Omega$; see e.g.~\cite[\S{4}]{PPW78}. 
As shown in Lemma~\ref{lem:app-g-1}, this uniform convergence has two important consequences. First, it implies that~$\P$ satisfies the weak Gibbs condition.
Second, it implies that there is an $\Exp{o(n)}$-sequence $(D_n)_{n\in \nn}$ such that
\begin{equation*}
    D_n^{-1} \P_n(x_1^n)\P_m(x_{n+1}^{n+m})
    \leq \P_{n+m}(x_1^{n+m}) 
    \leq D_n \P_n(x_1^n)\P_m(x_{n+1}^{n+m})
\end{equation*}
for all $x\in \Omega$, and all $n,m\in \nn$,
which implies~\assref{ud} and~\assref{sld} with $\tau_n = 0$, but with a possibly unbounded sequence~$(C_n)_{n\in\nn}$. By Remark~\ref{rem:tau0pa},~\assref{pa} holds as well. If~$\Q$ is a second $g$-measure on~$\Omega$ (for a possibly different $g$-function $g'$), then~\assref{jsld} holds as well by Remark~\ref{rem:tau0sld}. Our results thus apply, and once again, the expressions obtained in \secref{sec:equilbowen} remain valid thanks to~\eqref{eq:Gibbsprop}, with $\potone = \ln g$ and $\pottwo = \ln g'$. 

\begin{remark}The above discussion can easily be generalized to transitive Markov subshifts $\Omega'\subseteq \Omega$. However, developing a theory of $g$-measures on more general subshifts seems to be more delicate, and goes beyond the scope of the present paper. Still, without any reference to the general theory of $g$-measures, one obtains by combining Lemmas~\ref{lem:app-g-1}, \ref{lem:specud} and \ref{lem:specsld} that our decoupling assumptions hold if $\Omega'$ satisfies the {\em flexible specification property} and the {\em abundant periodic orbit property} of Definition~\ref{def:specif}, and if one {\em assumes} that the sequence $(g_n)_{n\in \nn}$ defined by \eqref{eq:defgnfullshift} converges uniformly to some continuous function $g:\Omega'\to (0,1]$; see Lemma~\ref{lem:app-g-1} for a precise statement.
\end{remark}

We conclude this subsection by noting two interesting references. First, Hulse constructed in~\cite{H06} an example of a $g$-measure that is not ergodic, showing once again that~\assref{ud},~\assref{sld} and~\assref{pa} do not imply ergodicity.\footnote{Ergodicity follows from~\assref{sld} if we further assume that $\tau_n=O(1)$ and $C_n=O(1)$; see Lemma~A.2 in~\cite{CJPS19} for a slightly more general sufficient condition.} Second, in the framework of $g$-measures, the large deviations of empirical entropies (which are also entropy estimators) were studied in~\cite{CG2005}.

\subsection{Hidden Markov models and lack of exponential tightness}
\label{sec:lackexpotight}

\newcommand{\cAh}{\mathcal{A}^\textnormal{H}}
\newcommand{\Omh}{\Omega^\textnormal{H}}
\newcommand{\Ph}{\mathbb{P}^\textnormal{H}}
\newcommand{\Qh}{\mathbb{Q}^\textnormal{H}}
\newcommand{\f}{f}

A hidden Markov measure~$\P$ on~$\Omega$ with finite hidden alphabet~$\cAh$ is obtained from a shift-invariant Markov measure $\Ph$ on~$\Omh := (\cAh)^{\nn}$ and a surjective map~$\f : \cAh \to \cA$ by prescribing the marginals $\P_n := \Ph_n \circ (\f^{\otimes n})^{-1}$ for all~$n \in \nn$. The name ``hidden Markov model'' refers to the pair $(\Ph,\P)$.
There exist several different characterizations of hidden Markov measures; one which is particularly useful from the point of view of decoupling properties is the representation in terms of products of matrices discussed in Proposition~2.25 in~\cite{BCJPPExamples}. The reader is also encouraged to consult Example~2.25 in~\cite{CJPS19}. 
Using that $|\cAh|<\infty$, it is straightforward to show that, if~$\Ph$ and~$\Qh$ are irreducible, stationary Markov measures on~$\Omh$ that satisfy $\Ph_n \ll \Qh_n$ for all~$n\in\nn$, then $\P$ and~$\Q$ defined using the same function~$\f$ satisfy~\assref{ud},~\assref{jsld} and~\assref{pa}, and Theorems~\ref{thm:CJPS} and~\ref{thm:mainthmA}--\ref{thm:mainthmC} apply.

Obviously, this is a generalization of the setup of \secref{sec:markov}, but in a completely different spirit from that of \secref{sec:equilbowen}: hidden Markov measures can be far from Gibbsian; see e.g.\ Theorem~2.10 in~\cite{BCJPPExamples}.%
    \footnote{Historical details and references are given in the discussion of Blackwell--Furstenberg--Walters--van den Berg measures in~\cite[\S{2.2}]{BCJPPExamples}. Important references on the topic of Gibbsianity of hidden Markov measures (or lack thereof) include~\cite[\S{3}]{LMV98},~\cite{CU11} and~\cite{Ve11}.}
    The assumption $|\cAh|<\infty$ ensures the existence of a constant $c\geq 0$ such that 
\begin{equation}
\label{eq:pnexbdbl}
    \inf_{x\in \supp \P}\P_n(x_1^n)\geq \Exp{-cn} 
    \qquad\text{and}\qquad  
    \inf_{x\in \supp \Q}\Q_n(x_1^n)\geq \Exp{-cn},
\end{equation}
so $0\leq -\frac 1n \ln \Q_1(x_1^n)\leq c$ for $\P$-almost every $x\in \Omega$.
These bounds imply exponential tightness and thus guarantee goodness of the rate function~$\ilpq$. 
In fact, $\ilpq$ is infinite on $(c,\infty)$, and the same is then automatically true of $\iw$, $\iv$ and~$\ir$. By the same token, the bounds~\eqref{eq:pnexbdbl} imply that $\qlpq \leq c\alpha$ for all $\alpha \geq 0$, so $\qlpq$, $\qw$, $\qr$ and $\qv$ are finite everywhere. The bounds~\eqref{eq:pnexbdbl} and these consequences are a common feature of all the examples in Sections~\ref{sec:exiid}--\ref{sec:new-g-sm}.

Dropping the assumption that~$\cAh$ is finite, the class of hidden Markov measures allows for examples where~\eqref{eq:pnexbdbl} fails, but one then needs to verify on a case-by-case basis whether~\assref{ud},~\assref{jsld} and~\assref{pa} are satisfied in order for our results to apply. 
Using $\cAh = \{0\} \cup \nn$ as in \cite[\S{A.2}]{CJPS19}, one easily constructs fully supported, admissible pairs $(\P, \Q)$ on $\cA = \{\mathsf{a},\mathsf{b}\}$, satisfying also~\assref{pa}, but violating either or both inequalities in~\eqref{eq:pnexbdbl}. We limit ourselves to providing two sets of parameters following the notation in~\cite[\S{A.2}]{CJPS19}: Examples~\ref{ex:HMM1} and~\ref{ex:HMM2} are illustrated in Figures~\ref{fig:HMC1} and~\ref{fig:HMC2}, respectively.

\begin{example}\label{ex:HMM1}
    Define $\P$ using~$\gamma(n)=n\ln 2$, and $\Q$ using~$\hat\gamma(n)=2n+0.15n^2$. Then, the measure~$\P$ is uniform and $\Q(\mathsf{b}^n) \sim \Exp{-\hat \gamma(n)}$. Note that the quadratic term in the definition of~$\hat{\gamma}(n)$ makes the second bound in~\eqref{eq:pnexbdbl} fail. Here, 
    \begin{align*}
        \qlpq(\alpha) 
            &\geq  \liminf_{n\to\infty}\frac 1n \ln ( \P_{n}(\mathsf{b}^n)\Q_n(\mathsf{b}^n)^{-\alpha}) \\
            &= -\ln 2 + \alpha \liminf_{n\to\infty} \frac {\hat \gamma(n)}{n},
    \end{align*}
    so $\qlpq(\alpha)=\infty$ for all $\alpha>0$. One can show that $\qlpq'(0^-)<\infty$ and that $\ilpq(s)=\iw(s) = 0$ for all $s\geq \qlpq'(0^-)$.    
    Thus, the sequence~$(-\tfrac 1n \ln \Q_n)_{n\in\nn}$ is not exponentially tight with respect to~$\P$, and the rate functions~$\ilpq$ and~$\iw$ are not good; see Figure~\ref{fig:HMC1}. 

    \begin{figure}[htb]
        \centering
    \begin{tabular}{p{0.45\textwidth} p{0.45\textwidth}}
        \vspace{0pt} \includegraphics[scale=\figscale]{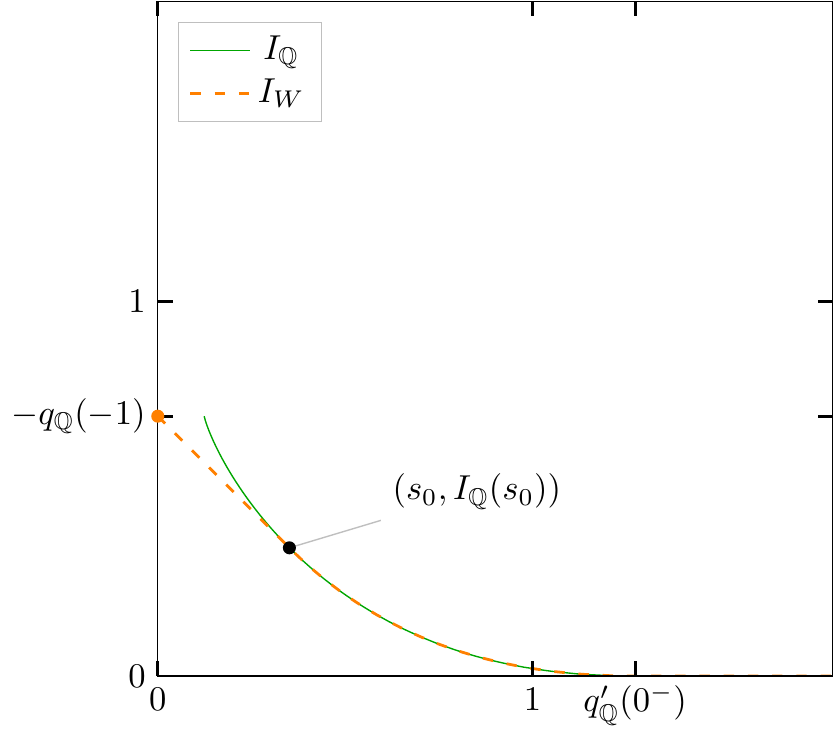} &
        \vspace{0pt} \includegraphics[scale=\figscale]{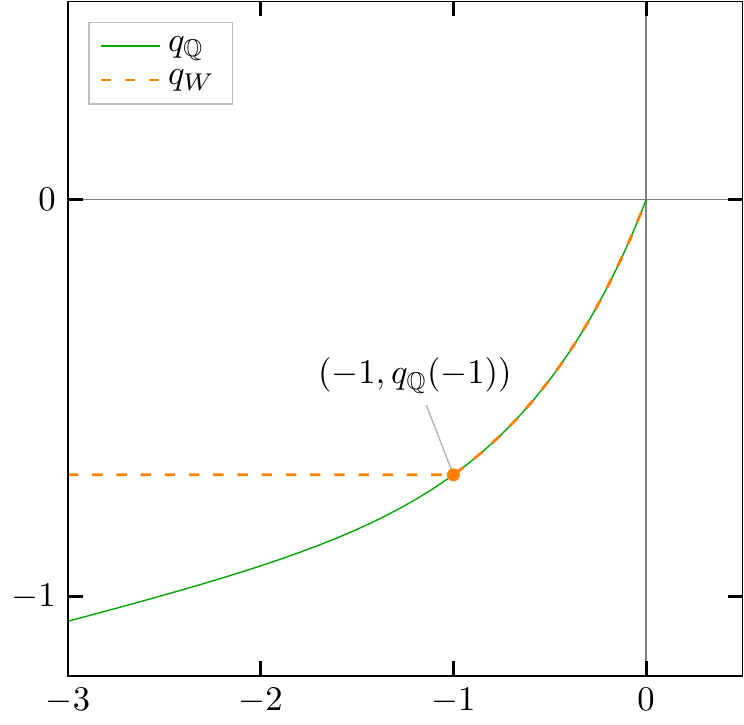}
      \end{tabular}
    \caption{Illustration of Example~\ref{ex:HMM1}. In the right picture, $\qlpq(\alpha)=\qw(\alpha)=\infty$ for all $\alpha>0$.}
    \label{fig:HMC1}
    \end{figure}
\end{example}

\begin{example}\label{ex:HMM2} Define~$\P=\Q$ using~$\gamma(0)=\hat \gamma(0)=0$, $\gamma(n) = \hat \gamma(n)= 0.6\,\Exp{n}$ for $n\in \nn$. Then, we have $\P(\mathsf{b}^n) = \Q(\mathsf{b}^n)\sim \Exp{-\gamma(n)}$ and both bounds in~\eqref{eq:pnexbdbl} fail. Here,~$\qlpp$ is infinite for all $\alpha > 1$, and $\qlpp'(1^-)=\infty$. 
    Each sequence of interest is exponentially tight and has a rate function that is good, but remains finite on an unbounded interval; in other words $\gamma_- = -\infty$. 

    \begin{figure}[htb]
        \centering
    \begin{tabular}{p{0.45\textwidth} p{0.45\textwidth}}
        \vspace{0pt} \includegraphics[scale=\figscale]{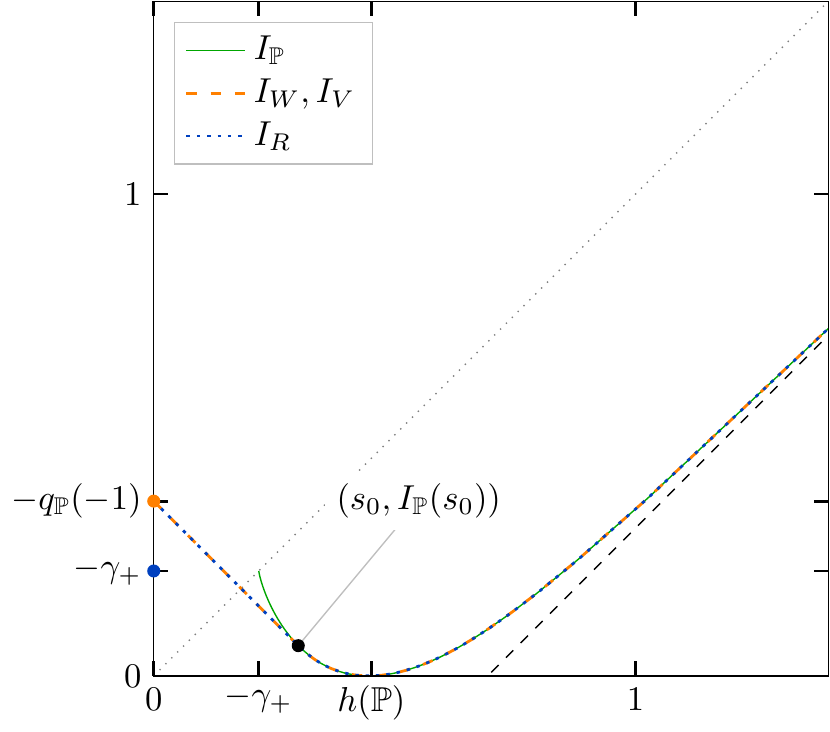} &
        \vspace{0pt} \includegraphics[scale=\figscale]{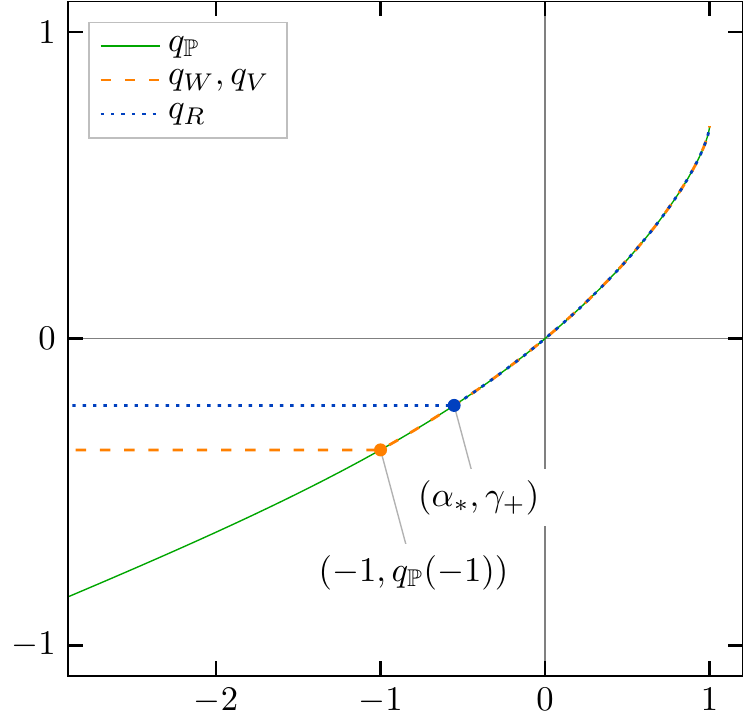}
      \end{tabular}
    \caption{Illustration of Example~\ref{ex:HMM2}. In the left picture, the dashed black asymptote is given by the equation $y = x-\qlpp(1)$, and in the right picture, $\qlpp(\alpha)=\qw(\alpha)=\qv(\alpha)=\qr(\alpha)=\infty$ for all $\alpha >1$.}
    \label{fig:HMC2}
    \end{figure}
\end{example}

\begin{remark}
\label{rem:expotight-1}
    The possible lack of exponential tightness is specific to the case where $\P\neq \Q$; when $\P=\Q$, our assumptions imply that all the random variables discussed in the paper are exponentially tight and that the corresponding rate functions are good. At the root of this fact is the exponential tightness of $(-\frac 1n\ln \P_n)_{n\in \nn}$ with respect to~$\P$ (Theorem~\ref{thm:CJPS}.iii). This exponential tightness will be derived as a consequence of~\eqref{eq:qppbdd}, but it can also be seen directly: for any $M\geq 0$,
    \begin{equation}
    \label{eq:expotightPn}
        \P\left\{x: -\frac 1n \ln \P_n(x_1^n) \geq M\right\} = \P\{x: \P_n(x_1^n) \leq \Exp{-nM}\}\leq |\cA|^{n}\Exp{-Mn}.
    \end{equation}
    We will come back to this point in Remark~\ref{rem:direxpotight}.
\end{remark}

\begin{remark}
    This class of examples also allows for cases where the pressure~$\qlpq$ is not differentiable\,---\,and hence where the G\"artner--Ellis route cannot provide the conclusions of Theorem~\ref{thm:CJPS}. Indeed, using the parameters $\gamma(n) = \hat \gamma(n)=\frac n{10}+4\ln(1+\frac n4)$ yields a situation where $\qlpq$ is not differentiable at $\alpha_0 \approx -0.1769$ and the rate functions are affine on an interval corresponding to the subdifferential of $\qlpq$ at $\alpha_0$.
\end{remark}

\section{Key estimates}
\label{sec:key-set} 

This section is devoted to technical estimates on the distribution of~$W_n$, $R_n$ and~$V_n$ at large but finite~$n \in \nn$. We start with estimates for~$W_n$ and then use those for our analysis of~$R_n$ and~$V_n$.
In order to do so, we first need a convenient reformulation of our decoupling assumptions.

We show in Lemma~\ref{lem:udevents} that, at the cost of replacing $C_n$ with $C_n|\cA|^{\tau_n}$ (which also satisfies~\eqref{eq:tauncn}), \assref{ud} implies that for every $n\in\nn$, $A\in \cF_n$ and $B\in \cF$,
\begin{equation}\label{eq:udeventseq}
    \P\left(A \cap \shift^{-n - \tau_n}B\right) \leq C_n \P(A)\P(B).
\end{equation}
In the same way, we show in Lemma~\ref{lem:sldevents} that at the cost of replacing $C_n$ with $(\tau_n+1)C_n$,~\assref{sld} implies that 
for every $n\in\nn$, $A\in \cF_n$ and $B\in \cF$, 
\begin{equation}\label{eq:sldeventseq}
    \max_{0\leq \ell \leq \tau_n}\P\left(A \cap \shift^{-n -\ell}B\right) \geq  C_n^{-1} \P(A)\P(B).
\end{equation}
We shall freely use the form~\eqref{eq:udeventseq} of~\assref{ud} and~\eqref{eq:sldeventseq} of~\assref{sld} throughout the paper.

\subsection{Waiting times}
\label{sec:Wn-estimates}

Because $W_n : \Omega \times \Omega \to \nn$ is $(\cF_n \otimes \cF)$-measurable, we will sometimes identify it to a function $\cA^n \times \Omega \to \nn$ denoted by the same symbol.

\begin{lemma}
\label{lem:UB-Q-Wn-a-y}
    Assume $\Q$ satisfies~\assref{sld}. Then, for all $m,n\in \nn$ and  $u\in\cA^n$,
    \begin{equation}\label{eq:UB2bounds}
           \Q\{ y : W_n(u,y) = m \} \leq \Q_n(u)  (1 - C_n^{-1}\Q_n(u))^{\left\lfloor\frac{m-1}{n + \tau_n}\right\rfloor}
    \end{equation}
    and
    \begin{equation}\label{eq:UB2only1bound}
        \Q\{ y :  W_n(u,y)\geq m\} \leq  (1 - C_n^{-1}\Q_n(u))^{\left\lfloor\frac{m-1}{n + \tau_n}\right\rfloor}.
    \end{equation}

\end{lemma}

\begin{proof}
    Let us fix $n, m$ and $u$ as in the statement, and let $k:=\lfloor\tfrac{m-1}{n + \tau_n}\rfloor$.
    First, if $k=0$, then \eqref{eq:UB2only1bound} is trivial and \eqref{eq:UB2bounds} holds since $\{ y : W_n(u,y) = m \}\subseteq \shift^{1-m}[u]$ and $\Q$ is shift invariant.
    Consider now the case $k\in\nn$.
    In view of~\assref{sld} (recall \eqref{eq:sldeventseq}) and shift invariance, we may inductively pick $k$ integers $0\leq \ell_1, \ell_2, \dots, \ell_k\leq \tau_n$ such that the intersections inductively defined by
    \[
        A_0 := [u] 
        \qquad \text{and} \qquad
        A_j := [u]^\complement \cap \shift^{-n - \ell_j}A_{j{-1}}
    \]
	for $j = 1, \dotsc, k$ satisfy $\Q\left([u] \cap \shift^{-n - \ell_j}A_{j-1}\right) \geq C_n^{-1} \Q_n(u)\Q(A_{j-1})$, and thus also
    \begin{equation}\label{eq:recverQaj}
        \Q(A_j) = \Q(A_{j-1})-\Q([u]\cap \shift^{-n-\ell_j} A_{j-1}) \leq \Q(A_{j-1})(1-C_n^{-1}\Q_n(u)).
    \end{equation}
    Iterating \eqref{eq:recverQaj} starting from $A_0$ yields
    \begin{equation}\label{eq:Qakub}
        \Q(A_k) \leq \Q_n(u)(1 - C_n^{-1} \Q_n(u))^{k}.
    \end{equation}
    Let  $M := \{m- jn - \sum_{i=1}^j \ell_i\}_{j=1}^{k}$. By construction, $kn + \sum_{i=1}^k\ell_j \leq k({n + \tau_n})<m$, so $M \subseteq  [1, m-n] \subseteq [1,m-1]$. 
    As a consequence,
    \begin{align*}
        \{ y :  W_n(u,y) = m \}  &= \shift^{1-m}[u] \cap \bigcap_{m'=1}^{m-1} \shift^{1-m'}[u]^\complement\\
        & \subseteq \shift^{1-m}[u] \cap \bigcap_{m'\in M} \shift^{1-m'}[u]^\complement\\
        & =  \shift^{1-m+nk + \sum_{j=1}^{k}\ell_j } A_{k}.
    \end{align*}
    Thus, by shift invariance and~\eqref{eq:Qakub}, we readily obtain~\eqref{eq:UB2bounds}.
    The proof of~\eqref{eq:UB2only1bound} is exactly the same with~$A_0$ replaced by~$\Omega$.     
\end{proof}

\begin{remark}\label{rem:telescop}
    The bounds~\eqref{eq:recverQaj} and \eqref{eq:Qakub} are inspired by~\cite[\S{2}]{Ko98}, and were already adapted to selective decoupling conditions in~\cite{CDEJR22}. It might be surprising that the  \emph{upper} bound~\eqref{eq:UB2bounds} relies on the \emph{lower} decoupling assumption \assref{sld}. In fact, even if $\tau_n = 0$ for all $n$, using~\assref{ud} would yield $\Q(A_k)\leq  C_n^k\Q_n(u)(1 -\Q_n(u))^{k}$ instead of \eqref{eq:Qakub}; the extra factor of~$C_n^k$ is too crude since we will be interested in the case where~$k \gg n$. The opposite will happen in Lemma~\ref{lem:LB-Q-Wn-a-y}, where a lower bound will be proved using~\assref{ud}; see \eqref{eq:Qfklb}.
\end{remark}

If $\Q_n(u)=0$, then $\Q\{y:W_n(u,y)=m\} \leq \Q(\shift^{1-m}[u])=\Q_n(u) = 0$ for all $m\in \nn$, and thus $W_n(u, \cdot\, )$ is almost surely infinite. Conversely, if $\Q$ satisfies~\assref{sld}, the bound~\eqref{eq:UB2only1bound} ensures that $W_n(u, \cdot\, )$ is almost surely finite whenever $\Q_n(u)>0$. This last observation is at the heart of the next lemma.

\begin{lemma}
    Let $n\in \nn$, let $\Q$ satisfy~\assref{sld}, and assume that $\P_n \ll \Q_n$. Then, the random variable~$W_n$ is $(\P\otimes\Q)$-almost surely finite.
\end{lemma}

\begin{proof}
    This is a consequence of the decomposition 
    $$
        W_n(x,y) = \sum_{u\in \cA^n}1_{[u]}(x)W_n(u,y),
    $$
    the bound~\eqref{eq:UB2only1bound}, and the absolute continuity assumption.
\end{proof}

We now turn to lower bounds on the distribution of $W_n(u,\cdot\,)$. The following lemma will be useful when $(1-C_n\Q_n(u))^{\left\lceil\frac{m-1}{n + \tau_n}\right\rceil}$ is close to~$1$.
\begin{lemma}
\label{lem:LB-Q-Wn-a-y}
    Assume $\Q$ satisfies~\assref{ud}. Then, for all $n,m\in\nn$ and $u\in\cA^n$ such that $C_n\Q_n(u)\leq 1$,
    \begin{equation}\label{eq:Qwnrupper}
    \begin{split}
        &\Q\{ y : W_n(u,y) = m \} \\
        &\splitspace \geq \frac{1}{n + \tau_n}\Q_n(u) \left(1 - ({n + \tau_n}) \left(1 - (1-C_n\Q_n(u))^{\left\lceil\frac{m-1}{n + \tau_n}\right\rceil}\right)\right).
    \end{split}
    \end{equation}
\end{lemma}

\begin{proof}
    Let us fix $n,m$ and $u$ as in the statement. We first prove that it suffices to establish the bound
    \begin{equation}\label{eq:toshowwninter}
        \begin{split}
            &\Q\{ y : k({n + \tau_n}) < W_n(u,y) \leq (k+1)({n + \tau_n}) \} \\
            &\splitspace \geq \Q_n(u) \left( 1 - ({n + \tau_n})\left(1 - (1-C_n\Q_n(u))^k\right)\right)
        \end{split}
    \end{equation}
    for all $k\in \nn\cup\{0\}$. 
    Since 
    \begin{equation}\label{eq:Qwnnondec}
        \{y:W_n(u,y) = t + 1 \} \subseteq \shift^{-1}\{y:W_n(u,y) = t\}
    \end{equation}
    for all $t\in\nn$, the probability $\Q\{y:W_n(u,y) = t\}$ is nonincreasing in~$t$.
    As a consequence, the left-hand side of~\eqref{eq:toshowwninter} is bounded above by $({n + \tau_n})\Q\{y : W_n(u,y) = k({n + \tau_n})+1\}$, so \eqref{eq:toshowwninter} implies that
    \begin{equation*}
    \begin{split}
        &\Q\{ y :  W_n(u,y) = 1+k({n + \tau_n}) \}\\
    &\splitspace\geq \frac{1}{n + \tau_n}\Q_n(u) \left(1 - ({n + \tau_n}) \left(1 - (1-C_n\Q_n(u))^{k}\right)\right).
    \end{split}
    \end{equation*}
    By nonincreasingness, the same lower bound applies to $\Q\{ y :  W_n(u,y) =m\}$ for all $k$ such that $1+k({n + \tau_n})\geq m$, and thus~\eqref{eq:toshowwninter} indeed implies~\eqref{eq:Qwnrupper}.

    We now establish \eqref{eq:toshowwninter}.
    For every $y\in\shift^{1-(k+1)({n + \tau_n})}[u] =:A$ we have $W_n(u,y)\leq(k+1)({n + \tau_n})$, and for every $y\in \bigcap_{r = 1}^{k({n + \tau_n})} \shift^{1-r}[u]^\complement$  we have $W_n(u,y)> k({n + \tau_n})$. As a consequence,
    \begin{equation}\label{eq:Qykntn}
        \Q\{y : k({n + \tau_n}) < W_n(u,y) \leq (k+1)({n + \tau_n})\} \geq
        \Q\left( A\cap \bigcap_{r = 1}^{k({n + \tau_n})} \shift^{1-r}[u]^\complement \right).
    \end{equation}
    We now write $
    \bigcap_{r = 1}^{k({n + \tau_n})} \shift^{1-r}[u]^\complement =  \bigcap_{j=1}^{n + \tau_n}B_j
    $, where for $1\leq j \leq {n + \tau_n}$,
    $$
        B_j := \bigcap_{d = 0}^{k-1} \shift^{1-d({n + \tau_n})-j}[u]^\complement.
    $$
    Notice that for each $j$, the events whose intersection defines $B_j$ are separated by ``gaps'' of size $\tau_n$, which will allow to use~\assref{ud} below. By a union bound, \eqref{eq:Qykntn} implies that
    \begin{align*}
        \Q\{y : k({n + \tau_n}) < W_n(u,y) \leq (k+1)({n + \tau_n})\}
        &\geq
        \Q\left( A\cap \bigcap_{j=1}^{n + \tau_n}B_j \right) 
        \\
        &=\Q(A)-\Q\left(\bigcup_{j=1}^{n + \tau_n}(A\cap B_j^\complement)  \right) 
        \\
        &\geq \Q(A)-\sum_{j=1}^{n + \tau_n}\Q(A\cap B_j^\complement) 
        \\
        &= \Q(A) -\sum_{j=1}^{n + \tau_n}(\Q(A)-\Q(A\cap B_j)).
    \end{align*}
    By shift invariance, $\Q(A)=\Q_n(u)$, and thus the proof of \eqref{eq:toshowwninter} will be complete once we have shown that
    \begin{equation}\label{eq:toshowBjj}
        \Q(A\cap B_j) \geq  \Q_n(u)(1-C_n\Q_n(u))^k.
    \end{equation}
    Fix $1\leq j \leq {n + \tau_n}$. We have 
    \begin{equation}\label{eq:AcapBj}
    A\cap B_j = \shift^{1-j}\left(\shift^{j-(k+1)({n + \tau_n})}[u]\cap  B_1\right) = \shift^{1-j} F_k,
    \end{equation}
    with $F_0, \dots, F_k$ inductively defined by
    $$
        F_0 := \shift^{j-n-\tau_n}[u] 
        \qquad \text{and} \qquad 
        F_i := [u]^\complement \cap \shift^{-n-\tau_n} F_{i-1}
    $$
    for $i=1,\dotsc,k$. By~\assref{ud} (recall \eqref{eq:udeventseq}),
    $$\Q(F_i) =  \Q(F_{i-1})-\Q([u]\cap \shift^{-n-\tau_n} F_{i-1}) \geq \Q(F_{i-1})(1-C_n\Q_n(u)).$$
    Iterating this bound starting with $F_0$ yields
    \begin{equation}\label{eq:Qfklb}
    \Q(F_k)\geq  \Q_n(u)(1-C_n\Q_n(u))^k.
    \end{equation}
    Combining this with \eqref{eq:AcapBj} and using shift invariance establishes \eqref{eq:toshowBjj}, as claimed.
\end{proof}

The next proposition makes precise the ``cut-off'' phenomenon sketched in \eqref{eq:cutoffexp} and uses the notation
\begin{equation}\label{eq:defUns}
    U_{n}(s) := \left\{ u \in \cA^n : \Q_n(u) \leq {\Exp{-ns}} \right\}
\end{equation}        
for $s>0$.

\begin{proposition}
    \label{prop:main-Wn-bound}Suppose that~$\Q$ satisfies~\assref{ud} and~\assref{sld}, and let $\P\in \cP_{\textnormal{inv}}(\Omega)$ satisfy $\P_n \ll \Q_n$ for all~$n\in\nn$. 
        Then, for all $s>0$ and all $0 < \delta \leq \epsilon < \tfrac s2$, we have, for all large enough $n$,
        \begin{equation}\label{eq:toshowUB}
            \begin{split}
                  &\P\otimes\Q\{(x,y) : \tfrac 1n \ln W_n(x,y) \in B(s,\epsilon)\}  \\
                  &\splitspace\leq   \Exp{n(s + \epsilon)} \sum_{u \in U_{n}(s-\epsilon-\delta)}\Q_n(u)\P_n(u) + \exp(-\Exp{\frac {n\delta}2})
            \end{split}
        \end{equation}
        and 
    \begin{equation}\label{eq:toshowLB}
        \P\otimes\Q\{(x,y) : \tfrac 1n \ln W_n(x,y) \in B(s,\epsilon)\}\geq \Exp{ n(s  - \delta)}\sum_{u \in U_{n}(s+\delta)}\Q_n(u)\P_n(u).
    \end{equation}
\end{proposition}

\begin{proof}
    The key idea in this proof is to show that, when $\frac 1n \ln m \in B(s,\epsilon)$, the $\lfloor\tfrac{m-1}{n + \tau_n}\rfloor$-th power of $(1 - C_n^{-1}\Q_n(u))$ in \eqref{eq:UB2bounds} vanishes superexponentially for all  $u \notin U_{n}(s-\epsilon-\delta)$ as $n\to\infty$, and that the $\lceil\tfrac{m-1}{n + \tau_n}\rceil$-th power of $(1-C_n\Q_n(u))$ in \eqref{eq:Qwnrupper} is very close to~$1$ for all $u\in  U_{n}(s+\delta)$ as $n\to\infty$.
    We fix $s$, $\epsilon$ and $\delta$ as in the statement and first note that
    \begin{equation}\label{eq:rewritecentral}
    \begin{split}
        & \P\otimes\Q\{(x,y) : \tfrac 1n \ln W_n(x,y) \in B(s,\epsilon)\} \\
        &\splitspace =  \sum_{u\in \cA^n}\P_n(u)\Q\{y: \tfrac 1n \ln W_n(u,y) \in B(s,\epsilon)\}.
    \end{split}
    \end{equation}

    \noindent{\it Proof of \eqref{eq:toshowUB}}. Notice that $|\{m\in \nn: \tfrac 1n  \ln m \in B(s,\epsilon)\}|\leq \Exp{n(s+\epsilon)}$. 
    Thus, using \eqref{eq:rewritecentral}, the nonincreasingness of $t\mapsto \Q\{y:W_n(u,y) = t\}$ (recall \eqref{eq:Qwnnondec}), and then Lemma~\ref{lem:UB-Q-Wn-a-y}, we find
    \begin{equation}\label{eq:rewritePotimesQ}
        \begin{split}
        &\P\otimes\Q\{(x,y) : \tfrac 1n \ln W_n(x,y) \in B(s,\epsilon)\} \\
                &\splitspace \leq \Exp{n(s+\epsilon)}\sum_{u\in \cA^n}\P_n(u)\Q\{y:  W_n(u,y) =  m_n\}\\
                &\splitspace \leq \Exp{n(s+\epsilon)}\sum_{u\in \cA^n}\P_n(u)\Q_n(u)(1 - C_n^{-1}\Q_n(u))^{k_n},
        \end{split}
        \end{equation}
        where  $m_n :=\lceil \Exp{n(s-\epsilon)}\rceil$ and $k_n := \big\lfloor\frac{m_n-1}{n + \tau_n}\big\rfloor$.
    We now split the above sum into a sum over $ U_{n}(s-\epsilon-\delta)$ and a sum over $ U_{n}^\complement(s-\epsilon-\delta)$. Clearly, 
    \begin{equation}
    \label{eq:toshowUB-i}
        \sum_{u\in  U_{n}(s-\epsilon-\delta)}\P_n(u)\Q_n(u)(1 - C_n^{-1}\Q_n(u))^{k_n}\leq \sum_{u \in U_{n}(s-\epsilon-\delta)}\P_n(u)\Q_n(u),
    \end{equation}
    while on the other hand, 
    \begin{equation}
    \label{eq:toshowUB-ii}
        \sum_{u\in  U_{n}^\complement(s-\epsilon-\delta)}\P_n(u)\Q_n(u)(1 - C_n^{-1}\Q_n(u))^{k_n}
        \leq (1 - C_n^{-1}\Exp{-n(s-\epsilon-\delta)})^{k_n}.
    \end{equation}
    But using the inequality $(1-x)^{k_n} \leq \Exp{-xk_n}$, we find
    \begin{align*}
        \Exp{n(s+\epsilon)}(1 - C_n^{-1}\Exp{-n(s-\epsilon-\delta)})^{k_n} & \leq \exp\left(n(s+\epsilon)-C_n^{-1}\Exp{-n(s-\epsilon-\delta)}k_n\right) \\
        &\leq \exp(-\Exp{\frac {n\delta}2})
    \end{align*}
    for $n$ large enough, thanks to the fact that $k_n \geq \Exp{n(s-\epsilon) 
    - o(n)}$ and $C_n = \Exp{o(n)}$ as $n\to\infty$. 
    Therefore, using the estimates~\eqref{eq:toshowUB-i} and \eqref{eq:toshowUB-ii} in~\eqref{eq:rewritePotimesQ} indeed yields \eqref{eq:toshowUB}.

    \medskip
    \noindent{\it Proof of \eqref{eq:toshowLB}}. This time, note that $|\{m\in \nn: \tfrac 1n  \ln m \in (s-\epsilon,s] \}|\geq \tfrac 12\Exp{ns}$ for $n$ large enough, and set $m_n :=\lfloor \Exp{ns}\rfloor$ and $k_n := \big\lceil\frac{m_n-1}{n + \tau_n}\big\rceil$. 
    Then, by  \eqref{eq:rewritecentral} and the nonincreasingness of $t\mapsto \Q\{y:W_n(u,y) = t\}$, we obtain
    \begin{align*}
        &\P\otimes\Q\{(x,y) : \tfrac 1n \ln W_n(x,y) \in B(s,\epsilon)\} \\
        &\splitspace \geq \P\otimes\Q\{(x,y) : \tfrac 1n \ln W_n(x,y) \in (s-\epsilon,s]\}\\
        &\splitspace \geq \frac{ \Exp{ns}}{2} \sum_{u\in \cA^n}\P_n(u)\Q\{y:  W_n(u,y) =  m_n\}\\
        &\splitspace \geq \frac{ \Exp{ns}}{2} \sum_{u\in U_{n}(s+\delta)}\P_n(u)\Q\{y:  W_n(u,y) =  m_n\}.
    \end{align*}
    Then, for $n$ large enough so that $C_n \Exp{-n(s+\delta)}\leq 1$, we can apply Lemma~\ref{lem:LB-Q-Wn-a-y} to every $u\in  U_{n}(s+\delta)$, and we obtain that
    \begin{equation}\label{eq:rewritePotimesQLB}
        \begin{split}
        &\P\otimes\Q\{(x,y) : \tfrac 1n \ln W_n(x,y) \in B(s,\epsilon)\} \\[2mm]
        &\splitspace \geq  \frac{  \Exp{ns}}{2}  \sum_{u\in U_{n}(s+\delta)}\P_n(u)\Q_n(u) \left(({n + \tau_n})^{-1} - \left(1 - (1-C_n\Q_n(u))^{k_n}\right)\right) \\
        &\splitspace \geq  \frac{  \Exp{ns}}{2}  \sum_{u\in U_{n}(s+\delta)}\P_n(u)\Q_n(u) \left(({n + \tau_n})^{-1} - \left(1 - (1-C_n\Exp{-n(s+\delta)})^{k_n}\right)\right). 
        \end{split}
        \end{equation}
    Notice that by Bernoulli's inequality, $(1-x)^{k_n}\geq 1-xk_n$ for all $x\leq 1$. As a consequence, for all $n$ large enough,
    \begin{equation*}
        (1-C_n\Exp{-n(s+\delta)})^{k_n}
            \geq 1-C_n\Exp{-n(s+\delta)}k_n
            \geq 1- \Exp{-\frac {n\delta}2},
    \end{equation*}
    where we have used that $k_n \leq \Exp{ns + o(n)}$ and $C_n = \Exp{o(n)}$.
    Substitution into \eqref{eq:rewritePotimesQLB} yields
    \begin{align*}
        &\P\otimes\Q\{(x,y) : \tfrac 1n \ln W_n(x,y) \in B(s,\epsilon)\} \\
        &\splitspace \geq \frac{\Exp{ns}}{2} \sum_{u\in U_{n}(s+\delta)}\P_n(u)\Q_n(u) \left(({n + \tau_n})^{-1} - \Exp{-\frac {n\delta}2}\right) ,
    \end{align*}
    from which the lower bound \eqref{eq:toshowLB} readily follows when $n$ is large enough.
\end{proof}

\subsection{From waiting times to return times}
\label{sec:Wn-to-Rn}

\begin{lemma}
\label{lem:UB-Wn-to-Rn}
    Assume $\P$ satisfies~\assref{ud}. Then, for all~$s > 0$ and all $0 < \epsilon < s$, we have, for all large enough $n$,
    \begin{equation*}
        \P\{x : \tfrac 1n \ln R_n(x) \in B(s,\epsilon)\} \leq C_n \sum_{u \in \cA^n} \P_n(u) \P\{ x : \tfrac 1n \ln W_n(u,x) \in B(s,2\epsilon)\}
    \end{equation*}
    and
    \begin{equation*}
        \P\{x : \tfrac 1n \ln V_n(x) \in B(s,\epsilon)\} \leq C_n \sum_{u \in \cA^n} \P_n(u) \P\{ x : \tfrac 1n \ln W_n(u,x) \in B(s,2\epsilon)\}.
    \end{equation*}
\end{lemma}

\begin{proof}
    Fix~$s > \epsilon > 0$, fix $n\in \nn$ large enough so that
    \begin{equation}\label{eq:conditionsnlarge}
        \Exp{n(s-\epsilon)} > {n + \tau_n} - 1   \qquad \text{and} \qquad \Exp{n(s-2\epsilon)} \leq \Exp{n(s-\epsilon)} +1-n-\tau_n,
    \end{equation}
    and fix~$u \in \cA^n$.
    Then, for all $x\in [u]$ such that $\frac 1n \ln R_n(x) \in B(s,\epsilon)$, we have $R_n(x) = W_n(u,\shift x) \in (\Exp{n(s-\epsilon)},\Exp{n(s+\epsilon)})$. The first condition in \eqref{eq:conditionsnlarge} implies that $W_n(u,\shift^{n + \tau_n}x) = W_n(u, T x)+1-n-\tau_n$ for all such $x$, so that, also using the second condition in \eqref{eq:conditionsnlarge},
    \begin{equation*}
    W_n(u,\shift^{n + \tau_n}x)\in (\Exp{n(s-\epsilon)}+1-n-\tau_n,\Exp{n(s+\epsilon)}+1-n-\tau_n)\subset(\Exp{n(s-2\epsilon)},\Exp{n(s+2\epsilon)}).
    \end{equation*}
    As a consequence, by~\assref{ud},
    \begin{align*}
        \P([u] \cap \{x : \tfrac 1n \ln {R}_n(x) \in B(s,\epsilon)\})
        &\leq \P([u] \cap \{x : \tfrac 1n \ln W_{n}(u,\shift^{n+\tau_n} x) \in B(s,2\epsilon)\}) \\
            &\leq C_n\P_n(u) \P\{x : \tfrac 1n \ln W_{n}(u,x) \in B(s,2\epsilon)\}.
    \end{align*}
     Taking the sum over $u\in \cA^n$ proves the desired bound for~$R_n$. The proof of the bound for~$V_n$ is almost identical: for all $x\in [u]$ such that $\frac 1n \ln V_n(x) \in B(s,\epsilon)$, we have this time $W_n(u,\shift^n x)\in (\Exp{n(s-\epsilon)},\Exp{n(s+\epsilon)})$, so that $W_n(u,\shift^{n + \tau_n}x)\in (\Exp{n(s-\epsilon)}-\tau_n,\Exp{n(s+\epsilon)}-\tau_n)\subset(\Exp{n(s-2\epsilon)},\Exp{n(s+2\epsilon)})$. The remainder of the argument is unchanged.
\end{proof}

\begin{lemma}
\label{lem:LB-Wn-to-Rn}
    Assume $\P$ satisfies~\assref{sld}. Then, for all~$s > 0$ and all $0 < \epsilon < s$, we have, for all large enough $n$,
    \begin{equation*}
        \P\{x : \tfrac 1n \ln R_n(x) \in B(s,\epsilon)\} \geq \frac{C_n^{-1}}{n + \tau_n} \sum_{u \in \cA^n} \P_n(u) \P\{ x : \tfrac 1n \ln W_n(u,x) \in B(s,\tfrac 12 \epsilon)\}
    \end{equation*}
    and
    \begin{equation*}
        \P\{x : \tfrac 1n \ln V_n(x) \in B(s,\epsilon)\} \geq \frac{C_n^{-1}}{1+\tau_n} \sum_{u \in \cA^n} \P_n(u) \P\{ x : \tfrac 1n \ln W_n(u,x) \in B(s,\tfrac 12 \epsilon)\}.
    \end{equation*}
\end{lemma}

\begin{proof}
    We first prove the statement concerning $R_n$. Fix~$s > \epsilon > 0$, fix $n$ large enough so that 
    \begin{equation}\label{eq:condlargeneps}
    \Exp{n(s + \frac 12 \epsilon)} + {n + \tau_n}-1 < \Exp{n(s+\epsilon)},
    \end{equation}
    and fix $u\in \cA^n$. By~\assref{sld}, there exists $0\leq \ell \leq \tau_n$ such that the set $$A_u := [u]\cap \shift^{-n-\ell} \{x:\tfrac 1n \ln W_n(u,x)\in B(s,\tfrac 12\epsilon)\}$$
    satisfies
    \begin{equation}\label{eq:ucaptlbl}
    \P(A_u)\geq  C_n^{-1}\P_n(u) \P\{x:\tfrac 1n \ln W_n(u,x)\in B(s,\tfrac 12\epsilon)\}.
    \end{equation}
    We now claim that
    \begin{equation}\label{eq:ucaptlnclaim}
    A_u\subseteq \bigcup_{j=1}^{n+\tau_n} \shift^{1-j}\{x:\tfrac 1n \ln R_n(x)\in B(s,\epsilon)\}.
    \end{equation}
    To prove \eqref{eq:ucaptlnclaim}, take an arbitrary $x\in A_u$, and let  $m:=W_n(u,\shift^{n+\ell}x)$, which, by the definition of $A_u$, satisfies  $\frac 1n \ln m \in B(s,\tfrac 12 \epsilon)$. Then, the set $I := \{i\in\nn : x_i^{i+n-1} = u\}$ contains $1$ and $n+\ell+m$, but excludes~$\{n+\ell+1, n+\ell+2, \dotsc, n+\ell+m-1\}$. Hence, with $j := \max\{ i\in I : i \leq n+\ell\}$ we find $R_n(\shift^{j-1} x) =  n + \ell +m -j \in [m, m+n+\tau_n-1]$, so by \eqref{eq:condlargeneps} we obtain \eqref{eq:ucaptlnclaim}.

    By taking the union over $u\in \cA^n$ in the left-hand side of \eqref{eq:ucaptlnclaim}, and using shift invariance to bound the probability of the right-hand side, we obtain
    \begin{equation}\label{eq:boundpunion}
    \sum_{u\in \cA^n}\P(A_u) = \P\left(\bigcup_{u\in \cA^n} A_u\right)  \leq (n+\tau_n)\P\{x:\tfrac 1n \ln R_n(x)\in B(s,\epsilon)\}.
    \end{equation}
    In view of \eqref{eq:ucaptlbl}, we have completed the proof of the statement about $R_n$.

    To adapt the proof for $V_n$, it suffices to replace the definition of $j$ with $j := \max\{ i\in I : i \leq 1+\ell \}$; then $V_n(\shift^{j-1}x) = 1+\ell+m-j\in [m,m+\tau_n]$ and the same arguments apply, with the factor $(n+\tau_n)$ replaced by $(1+\tau_n)$ in \eqref{eq:boundpunion}. 
\end{proof}

\begin{lemma}
    Assume $\P$ satisfies~\assref{sld}. Then, for every~$n\in\nn$, the random variables~$R_n$ and~$V_n$ are~$\P$-almost surely finite. 
\end{lemma}

\begin{proof}
    In view of~\eqref{eq:UB2only1bound}, the random variable $W_n(u,\cdot\,)$ is $\P$-almost surely finite for every fixed $u\in \cA^n$ with $\P_n(u)>0$. By shift invariance, so is $W_n(u,\cdot\,) \circ \shift^k$ for each $k$. In view of the expressions
    \begin{align*}
        R_n(x) &=  \sum_{u\in \cA^n}1_{[u]}(x)W_n(u, \shift x), \\ 
        V_n(x) &=  \sum_{u\in \cA^n}1_{[u]}(x)W_n(u, \shift^{n} x),
    \end{align*}
    the conclusion is immediate.
\end{proof}

\begin{remark}
An alternative way of showing that $R_n$ and $V_n$ are almost surely finite is to use the Poincar\'e recurrence theorem, see e.g.\ Theorem~1.4 in \cite{Wal}.
\end{remark}

\section{Weak LDPs and Ruelle--Lanford functions}\label{sec:RLfuncts}

We now briefly recall some terminology from the theory of large deviations, limiting ourselves to sequences of real-valued random variables. See for example \cite{DS1989,Ell,DZ} for proper introductions to the field.
Let $(Z_n)_{n\in\nn}$ be a sequence of (almost surely finite) real-valued random variables on a probability space $(\Omega_*,\P_*)$. The cases of interest will be
\begin{itemize}
    \item $\Omega_* = \Omega\times \Omega$, $\P_* = \P\otimes \Q$ and $Z_n = \frac 1n \ln W_n$ for Theorem~\ref{thm:mainthmA};
    \item $\Omega_* = \Omega$, $\P_* =\P$, with $Z_n = -\frac 1n \ln \Q_n$, $Z_n = \frac 1n \ln V_n$ and $Z_n = \frac 1n \ln R_n$ for, respectively, Theorems~\ref{thm:CJPS}, \ref{thm:mainthmB} and \ref{thm:mainthmC}.
\end{itemize}

The sequence $(Z_n)_{n\in\nn}$ is said to satisfy the \emph{large deviation principle} (LDP) if there exists a lower semicontinuous function~$I : \rr \to [0,\infty]$
such that 
\begin{equation}\label{eq:LD-LB}
    -\inf_{s\in O} I(s)\leq \liminf_{n\to\infty} \frac 1n\ln \P_*\{ x: Z_n(x) \in O \} 
\end{equation}
for every open set $O\subseteq \rr$ and 
\begin{equation}\label{eq:LD-UB}
    \limsup_{n\to\infty} \frac 1n\ln \P_*\{ x: Z_n(x) \in \Gamma \} 
        \leq -\inf_{{s} \in \Gamma} I({s})
\end{equation}
for every closed set $\Gamma \subseteq \rr$.
The bounds~\eqref{eq:LD-LB} and~\eqref{eq:LD-UB} are respectively called the \emph{large-deviation lower bound} and the \emph{large-deviation upper bound}, and the function~$I$\,---\,which is unique when it exists\,---\,is called the \emph{rate function}. Following standard terminology, we say that $I$ is a \emph{good} rate function if it properly diverges as $s\to\pm\infty$. We also recall that the large-deviation upper bound applied to the set $\rr$ implies that $\inf_{s\in \rr}I(s) = 0$. 
Ubiquitous in the theory of large deviations is the question of whether the rate function is convex and can be expressed as the Legendre--Fenchel transform of the corresponding pressure. A detailed analysis of these considerations for the random variables of interest is postponed to~\secref{sec:full}. 

Our analysis will require additional vocabulary which is discussed e.g.\ in~\cite[\S{1.2}]{DZ}.  The sequence  $(Z_n)_{n\in\nn}$ is said to satisfy the \emph{weak large deviation principle} if~\eqref{eq:LD-LB} holds for all open sets $O\subseteq \rr$, and~\eqref{eq:LD-UB} holds for all \emph{compact} sets $\Gamma \Subset \rr$. We shall sometimes refer to the standard LDP as the {\em full} LDP when we need to emphasize the contrast to the weak LDP used as a stepping stone towards the full LDP. The following notion will play a role in doing so: the sequence $(Z_n)_{n\in\nn}$ is said to be \emph{exponentially tight} if, for every~$\beta \in \rr$, there exists~$M > 0$ such that $\P_*\{ x : Z_n(x) \notin [-M,M]\} \leq \Exp{-\beta n}$ for all~$n$ large enough. 
To be more precise, we will appeal to the two following facts for real-valued sequences. First, if the weak LDP holds and the sequence is exponentially tight, then the full LDP holds with a good rate function. Second, if the full LDP holds with a good rate function, then the sequence is exponentially tight. While all our LDPs are full, and while exponential tightness does play a role in our analysis, we emphasize that the sequence $(\tfrac 1n \ln W_n)_{n\in\nn}$ need not be exponentially tight, as illustrated in \secref{sec:lackexpotight}.

As mentioned, we will first prove the weak LDP. We will do so using \emph{Ruelle--Lanford} (RL) \emph{functions}. We introduce the \emph{lower RL function} $ \underline{I} : \rr \to [0,\infty]$ defined by\footnote{The name ``{RL function}'' was first used in \cite{LPS94,LPS95}. The method of RL functions is often used in conjunction with subadditive arguments; this is in particular the case of the derivation of Theorem~\ref{thm:CJPS} in \cite{CJPS19}. In the present paper, once Theorem~\ref{thm:CJPS} is taken for granted, the proof of our results is not of the subadditive kind.} 
\begin{align*}
    \underline{I}(s) := -\lim_{\epsilon\to 0} \liminf_{n\to\infty} \frac{1}{n}\ln\P_*\{x : Z_n(x) \in B(s,\epsilon)\},
\end{align*}
and the \emph{upper RL function} $\overline{I}$ : $\rr \to [0,\infty]$ defined by 
\begin{align*}
    \overline{I}(s) := - \lim_{\epsilon\to 0} \limsup_{n\to\infty} \frac{1}{n}\ln\P_*\{x : Z_n(x) \in B(s,\epsilon)\}.
\end{align*}
It follows from their definition that $\overline I$ and $\underline I$ are lower semicontinuous, and that $\overline I \leq \underline I$. Moreover, the weak LDP holds if and only if we have the equality
\begin{equation}
\label{eq:RL-eq}
    \underline{I}(s) = \overline{I}(s) 
\end{equation}
for every~$s\in\rr$; see {e.g.}~\cite[\S{4.1.2}]{DZ} or~\cite[\S{3.2}]{CJPS19}. The common value in~\eqref{eq:RL-eq} must then coincide with~$I(s)$. 

The core of this section is devoted to proving the weak LDP for the sequences of interest via the validity of~\eqref{eq:RL-eq}. To be more precise, for each sequence, both RL functions are shown to be equal to the proposed rate function, as detailed in Table~\ref{tab:funct-and-rates}. We consider separately the case $s>0$ in \secref{sec:posvalues} and the case $s=0$ in \secref{sec:atorigin}, noting that the case $s < 0$ is trivial by mere nonnegativity of the random variables under study. We denote by $\uiw$, $\uiv$ and~$\uir$ (resp.\ $\oiw$, $\oiv$ and~$\oir$) the lower (resp.\ upper) RL functions associated with our sequences. 

\begin{table}[h]
    \begin{center}
        \small
        \begin{tabular}{c c c c c c}
            \toprule
            Random & Probability & Assumptions & Proposed & Eq.~\eqref{eq:RL-eq} & Eq.~\eqref{eq:RL-eq}  \\
            variable & space && rate function & for $s>0$ & for $s=0$ \\
            \midrule 
            $\tfrac 1n \ln W_n$ &
                $(\Omega^2,\P \otimes \Q)$  & admissibility & $\iw$  \eqref{eq:defiwexplicit} & Prop.~\ref{prop:RL-Wn} & Prop.~\ref{prop:LDPwns0} \\
            $\tfrac 1n \ln V_n$ &
                $(\Omega,\P)$ & \assref{sld}, \assref{ud}&  $\iv$  \eqref{eq:defivexplicit} & Prop.~\ref{prop:irandivspositive} & Prop.~\ref{prop:LDPtildeRns0} \\
            $\tfrac 1n \ln R_n$ &
            $(\Omega,\P)$ & \assref{sld}, \assref{ud}, \assref{pa}&  $\ir$ \eqref{eq:defirexplicit} & Prop.~\ref{prop:irandivspositive} & Prop.~\ref{prop:RNat0} \\
            \bottomrule
        \end{tabular}
    \end{center}
   \caption{Assumptions, proposed rate function and references to where the key equality~\eqref{eq:RL-eq} is established for the three weak LDPs to be proved.} 
    \label{tab:funct-and-rates}
\end{table}

\subsection{At positive values}\label{sec:posvalues}

Our first goal is to prove equality of the upper and lower Ruelle--Lanford functions at positive values of~$s$. 
We start with the RL functions $\uiw$ and $\oiw$ of the sequence $(\tfrac 1n \ln W_n)_{n\in\nn}$. Most of the work to show that $\uiw(s)=\oiw(s)$ when $s>0$ was done in \secref{sec:Wn-estimates}: in view of Proposition~\ref{prop:main-Wn-bound}, it only remains to estimate the quantity
\begin{equation}\label{eq:defjn}
    J_n(s) :=  \sum_{u\in U_n(s)}  \Q_n(u)\P_n(u) = \int_{[s,\infty)}\Exp{-rn} \dd\mu_n(r),
\end{equation}
where $\mu_n$ is the distribution of $-\frac 1n \ln \Q_n$ with respect to $\P$, and where $U_n(s)$ was defined in \eqref{eq:defUns}.
The integral in \eqref{eq:defjn} allows to express the limiting behavior of~$J_n$ in terms of the rate function~$\ilpq$ of Theorem~\ref{thm:CJPS}, as shown by the following straightforward variation of Varadhan's lemma.

\begin{lemma}
\label{lem:modVaradhan}
   Assume $(\P, \Q)$ is admissible. Then, for all $s>0$,
       \begin{align*}
           \sup_{r>s}(-r-\ilpq(r)) \leq   \liminf_{n\to\infty} \frac 1n \ln  J_n(s) \leq  \limsup_{n\to\infty}\frac 1n \ln  J_n(s) \leq \sup_{r\geq s}(-r-\ilpq(r)).
       \end{align*}
\end{lemma}

\begin{proof}
       Let us fix $s > 0$. For the lower bound, note that for every choice of $r>s$ and $0 < \epsilon < r-s$,
       \begin{align*}
          \liminf_{n\to\infty}\frac 1n \ln J_n(s) &\geq \liminf_{n\to\infty}\frac 1n\ln \int_{B(r,\epsilon)}\Exp{-r'n}\dd\mu_n(r') \\
          &\geq  -r-\epsilon + \liminf_{n\to\infty}\frac 1n \ln  \mu_n(B(r,\epsilon)) \\[2mm]
          &\geq -r -\epsilon - \ilpq(r),
       \end{align*}
       where we have used the large-deviation lower bound of Theorem~\ref{thm:CJPS}.
       For the upper bound, let $\beta := \inf_{r\geq s}(r+\ilpq(r)) \in [s, \infty)$. Then,
       \begin{align*}
         J_n(s)  &=  \int_{[s, \beta]}\Exp{-rn}\dd\mu_n(r) + \int_{(\beta, \infty)}\Exp{-rn}\dd\mu_n(r)\leq    \int_{[s, \beta]}\Exp{-rn}\dd\mu_n(r) + \Exp{-\beta n}.
       \end{align*}
       As a consequence, it suffices to show that 
       \begin{equation}\label{eq:toshowminusbeta}
           \limsup_{n\to\infty} \frac 1n\int_{[s, \beta]}\Exp{-rn}\dd\mu_n(r) \leq -\beta,
       \end{equation}
       which, since $[s,\beta]$ is compact, follows from a standard covering argument, see e.g.\ Lemma~4.3.6 in~\cite{DZ} or the proof of Proposition~\ref{prop:abstractZn}.iv below.
\end{proof}

\begin{proposition}
\label{prop:RL-Wn}
    Assume the pair $(\P, \Q)$ is admissible. Then, for all $s>0$, the RL functions for $(\frac 1n \ln W_n)_{n\in \nn}$ with respect to $\P\otimes \Q$ satisfy
    \begin{equation}\label{sq:toshowRLwn}
        \uiw(s) = \oiw(s) =  -s+\inf_{r\geq s}(r+\ilpq(r)).
    \end{equation}
\end{proposition}

\begin{proof}
    Let $s>0$. 
    By Proposition~\ref{prop:main-Wn-bound} and Lemma~\ref{lem:modVaradhan}, we find that for all $0 < \delta \leq \epsilon < \tfrac 12 s$,
    \begin{equation}\label{eq:ubepsidelta}
        \limsup_{n\to\infty} \frac 1n \ln\P\otimes\Q\{(x,y) : \tfrac 1n \ln W_n(x,y) \in B(s,\epsilon)\} 
        \leq s+\epsilon -  \inf_{r\geq s-\epsilon-\delta}(r+\ilpq(r)) 
    \end{equation}
    and
    \begin{equation}\label{eq:lbepsidelta}
        \liminf_{n\to\infty}  \frac 1n \ln\P\otimes\Q\{(x,y) : \tfrac 1n \ln W_n(x,y) \in B(s,\epsilon)\}
        \geq s - \delta -  \inf_{r> s+\delta}(r+\ilpq(r)).
    \end{equation}
    By taking the limit as $\delta \to 0$ first and then $\epsilon \to 0$ in \eqref{eq:ubepsidelta}, and since $\ilpq$ is lower semicontinuous, we obtain that $-\oiw(s)\leq s-  \inf_{r\geq s}(r+\ilpq(r))$. Taking the same limits in  \eqref{eq:lbepsidelta} yields $-\uiw(s)\geq s - \inf_{r> s}(r+\ilpq(r))$. We then have
    \begin{equation}\label{eq:fourquantsr}
        -s + \inf_{r\geq s}(r+\ilpq(r))   \leq \oiw(s) \leq \uiw(s) \leq  -s + \inf_{r> s}(r+\ilpq(r)).
    \end{equation}
    For all $s'\in (0,s)$, the last inequality in~\eqref{eq:fourquantsr} applied to $s'$ yields
    $$
    \uiw(s')\leq -s'+ \inf_{r> s'}(r+\ilpq(r)) \leq -s'+\inf_{r\geq s}(r+\ilpq(r)).
    $$
    Since $\uiw$ is lower semicontinuous (as a RL function), this in turn implies that
    $$
    \uiw(s)\leq \liminf_{s'\uparrow s} \uiw(s') \leq -s +\inf_{r\geq s}(r+\ilpq(r)),
    $$
    so the first three quantities in~\eqref{eq:fourquantsr} are actually equal, as desired.
\end{proof}

\begin{proposition}
\label{prop:irandivspositive}
    Assume~$\P$ satisfies~\assref{ud} and~\assref{sld}. Then, for all $s>0$, the RL functions for $(\frac 1n \ln R_n)_{n\in \nn}$ and  $(\frac 1n \ln V_n)_{n\in \nn}$ with respect to $\P$ satisfy
    \begin{equation*}
        \uir(s) = \oir(s)  = -s+\inf_{r\geq s}(r+\ilpp(r))
    \end{equation*}
    and
    \begin{equation*}
        \uiv(s) = \oiv(s)   = -s+\inf_{r\geq s}(r+\ilpp(r)).
    \end{equation*}
\end{proposition}

\begin{proof}
    Let $s>0$. 
    First, we remark that Proposition~\ref{prop:RL-Wn} applies to the admissible pair $(\P,\P)$, so that $\oiw(s) = \uiw(s) = -s+\inf_{r\geq s}(r+\ilpp(r))$.
    Then, Lemma~\ref{lem:UB-Wn-to-Rn} implies that $\oir(s) \geq \oiw(s)$ and $\oiv(s) \geq \oiw(s)$; recall also \eqref{eq:rewritecentral}. In the same way, Lemma~\ref{lem:LB-Wn-to-Rn} implies that  $\uir(s) \leq \uiw(s)$ and $\uiv(s) \leq \uiw(s)$. Since also $\oir\leq \uir$ and $\oiv\leq \uiv$, the proof is complete.
\end{proof}

\subsection{At the origin}\label{sec:atorigin}

In this subsection, we prove that, for each of the sequences $(\frac 1n \ln W_n)_{n\in \nn}$, $(\frac 1n \ln V_n)_{n\in \nn}$ and $(\frac 1n \ln R_n)_{n\in \nn}$, the upper and lower RL functions match at $s = 0$. 
We recall that the limit
\begin{equation}\label{eq:qlpqm1exists}
    \qlpq(-1) = \lim_{n\to\infty}\frac 1n \ln \sum_{u\in\cA^n} \P_n(u)\Q_n(u)
\end{equation}
exists for any admissible pair $(\P,\Q)$ by Theorem~\ref{thm:CJPS}.ii. 

\begin{proposition}\label{prop:LDPwns0}
    If the pair $(\P, \Q)$ is admissible, then
    \begin{equation}\label{eq:toshowLDPwns0}
        \uiw(0)
        =\oiw(0) 
        = -\qlpq(-1) = \inf_{r\geq 0}(r+\ilpq(r)).
    \end{equation}
\end{proposition}

\begin{proof}
    For all $\epsilon>0$ and $n\in \nn$,
    \begin{align*}
        \P\otimes \Q  \{(x,y): \tfrac 1n  \ln W_n(x,y) \in B(0,\epsilon)\} 
            &\geq  \P\otimes \Q  \{(x,y):W_n(x,y) =1\} \\ 
            &=\sum_{u\in \cA^n} \P_n(u)\Q_n(u).
    \end{align*}
    In view of \eqref{eq:qlpqm1exists} and the definition of $\uiw$, we have $\uiw(0) \leq -\qlpq(-1)$. 
    To obtain the opposite inequality for $\oiw(0)$, observe that  
    \begin{align*}
        \P\otimes \Q  \{(x,y): W_n(x,y) = k\} \leq \P\otimes \Q  \{(x,y): x_1^n = y_{1+k}^{n+k}\} =\sum_{u\in\cA^n} \P_n(u)\Q_n(u)
    \end{align*}
    for every $u\in \cA^n$, $n\in\nn$ and $k\in\nn$. Therefore, a union bound gives,  for every $\epsilon > 0$,
    \begin{equation*}
        \P\otimes \Q  \{(x,y):W_n(x,y) \leq \Exp{\epsilon n}\}\leq \Exp{\epsilon n}\sum_{u\in \cA^n} \P_n(u)\Q_n(u).
    \end{equation*}
    By \eqref{eq:qlpqm1exists} and the definition of $\oiw$, we conclude that $\oiw(0)\geq -\qlpq(-1)$. Since also $\uiw(0) \geq \oiw(0)$ we have thus established the first two equalities in~\eqref{eq:toshowLDPwns0}.

    To complete the proof, it remains to observe that, since $\ilpq(r)=\infty$ for all $r<0$, and since $\qlpq = \ilpq^*$ by Theorem~\ref{thm:CJPS}.ii, 
    \begin{equation}\label{eq:ltr0}
        \inf_{r\geq 0}(r+\ilpq(r))= \inf_{r\in \rr}(r+\ilpq(r)) = -\sup_{r\in \rr}(-r-\ilpq(r)) = -\qlpq(-1),
    \end{equation}
    which establishes the last identity in \eqref{eq:toshowLDPwns0}.
\end{proof}

\begin{proposition}
    \label{prop:LDPtildeRns0}
    If $\P$ satisfies~\assref{ud} and~\assref{sld},
    then
    \begin{equation}\label{eq:toshowLDPvns0}
        \uiv(0) = \oiv(0) 
        = - \qlpp(-1) =  \inf_{r\geq 0}(r+\ilpp(r)).
    \end{equation}        
\end{proposition}

\begin{proof}
    The stated assumptions allow to apply Theorem~\ref{thm:CJPS} to the pair $(\P,\P)$, and in particular its consequences~\eqref{eq:qlpqm1exists} and~\eqref{eq:ltr0} with $\Q=\P$.
    Since the third equality in~\eqref{eq:toshowLDPvns0} is a special case of~\eqref{eq:ltr0}, and since $\oiv(0)\leq \uiv(0)$ by definition, it suffices to establish the inequalities $\uiv(0) \leq -\qlpp(-1)$ and $\oiv(0) \geq -\qlpp(-1)$ in order to complete the proof. To this end, we let $\epsilon > 0$ be arbitrary and restrict our attention to~$n$ large enough so that $\min\{n, \Exp{\epsilon n}\}>\tau_n+1$.

    For each $u\in \cA^n$,~\assref{sld} implies that there is $0\leq \ell\leq \tau_n$ such that $\P([u]\cap \shift^{-n-\ell}[u]) \geq C_n^{-1} \P_n(u)^2$. Since $[u]\cap \{x:V_n(x) \leq \tau_n+1\} \supseteq [u]\cap \shift^{-n-\ell}[u]$, this in turn implies that
    \begin{align*}
       \P\{x:V_n(x) < \Exp{\epsilon n}\} \geq \P\{x:V_n(x) \leq \tau_n+1\} \geq   C_n^{-1} \sum_{u\in \cA^n} \P_n(u)^2.
    \end{align*}
    Combining this with \eqref{eq:qlpqm1exists} for $\Q=\P$ establishes the inequality $\uiv(0) \leq -\qlpp(-1)$.

    On the other hand, for each $k\in\nn$,
    $$
    \{x:V_n(x) = k\} \subseteq \bigcup_{u\in \cA^{n}}[u]\cap \shift^{1-n-k}[u] \subseteq \bigcup_{v\in \cA^{n-\tau_n}}[v]\cap \shift^{1-n-k}[v].
    $$
    Assuming without loss of generality that the sequence $(\tau_n)_{n\in\nn}$ is nondecreasing (so in particular $\tau_{n-\tau_n} \leq \tau_n$), 
    we obtain from~\assref{ud} that
    $$
        \P\{x:V_n(x) = k\} \leq C_{n-\tau_n} \sum_{v\in \cA^{n-\tau_n}} \P_{n-\tau_n}(v)^2.
    $$
    Considering the union over $k = 1,2,\dotsc, \lceil\Exp{\epsilon n}\rceil-1$, we further obtain
    \begin{equation}\label{eq:limlem37}
        \begin{split}
            \limsup_{n\to\infty}\frac 1n\ln \P\{x:V_n(x) < \Exp{\epsilon n}\}&\leq\limsup_{n\to\infty}\frac 1n\ln\left( \Exp{\epsilon n} C_{n-\tau_n}  \sum_{v\in \cA^{n-\tau_n}} \P_{n-\tau_n}(v)^2 \right)\\
            & = \epsilon + \limsup_{n\to\infty}\frac 1n\ln \sum_{u\in \cA^{n}} \P_n(u)^2,
        \end{split}
    \end{equation}
    where we have used that $\lim_{n\to\infty}\frac {n-\tau_n}n =1$.
    Combining this with \eqref{eq:qlpqm1exists} for $\Q=\P$ establishes the inequality $\oiv(0) \geq -\qlpp(-1)$, so the proof is complete.
\end{proof}

Let us now turn to $\uir(0)$ and $\oir(0)$, whose comparison is significantly more involved. We start with a technical lemma.
\begin{lemma}\label{lem:propzgammaplus}
    If $\P$ satisfies~\assref{sld} and~\assref{ud}, then
    \begin{equation}\label{eq:redefgammap}
        \gamma_+ = \lim_{n\to\infty}\frac 1n \sup_{u\in \cA^n}\ln \P_n(u) \geq \qlpp(-1).
    \end{equation}
\end{lemma}

\begin{proof}
    From Theorem~\ref{thm:CJPS}.iii.b, we know that the limit superior in the definition \eqref{eq:defgammaplus} of $\gamma_+$ is actually a limit, which is the first equality in \eqref{eq:redefgammap}. We now prove the inequality, noting that, by Theorem~\ref{thm:CJPS}.ii, the limit defining $\qlpp(-1)$ exists. For each $n\in\nn$, we find $\sup_{v\in \cA^n}\P(v)  = \sum_{u\in \cA^n}\P_n(u)\sup_{v\in \cA^n}\P_n(v)  \geq \sum_{u\in \cA^n}\P_n(u)^2$, and the claim follows from~\eqref{eq:qlpqm1exists} with $\Q=\P$.
\end{proof}

The following proposition shows that $\uir(0)  =  \oir(0) = -\gamma_+$ under the assumptions of Theorem~\ref{thm:mainthmC}. We give several additional inequalities in order to underline the role of~\assref{pa}. The proposition also shows that, while $\oir(0)$ and $\uir(0)$ are defined in terms of $\{x : R_n(x) < \Exp{\epsilon n}\}$, the subset $\{x : R_n(x) < n\}$ accounts for the full behavior of the probability at the exponential scale. 
\begin{proposition}
\label{prop:RNat0}
    The following hold:
    \begin{enumerate}[i.]
        \item If $\P$ satisfies~\assref{ud}, then 
        \begin{equation}\label{eq:toshowlimsupgammap}  
            \limsup_{n\to\infty} \frac 1n \ln \P\{x : R_n(x) < n\} \leq -\oir(0) \leq \gamma_+.
        \end{equation}
        \item  If $\P$ satisfies~\assref{pa}, then
        \begin{equation}\label{eq:toshowuir0r} 
        -\uir(0) 
        \geq  \liminf_{n\to\infty} \frac 1n \ln \P\{x : R_n(x) < n\} 
        \geq \gamma_+.
    \end{equation}
        \item If $\P$ satisfies~\assref{ud} and~\assref{pa}, then
        \begin{equation*}
            -\uir(0)  =  -\oir(0) =   \lim_{n\to\infty} \frac 1n \ln \P\{x : R_n(x) < n\} = \gamma_+.
        \end{equation*}
    \end{enumerate}
\end{proposition}

\begin{proof}
    Given $u\in \cA^n$, we denote by $\operatorname{per}(u)$ the period of~$u$, {i.e.}\ the smallest $p \in \nn$ such that $u_i = u_{i+p}$ for all $1\leq i \leq n-p$. Since the condition is vacuously true for $p=n$, we have the bound $\operatorname{per}(u)\leq n$.
    The key observation is that, for every $1\leq k< n$ and $x\in \Omega$,
    \begin{equation}
    \label{eq:obsperiod}
        R_n(x)=k 
        \qquad\iff\qquad  
        \operatorname{per}(x_1^{k+n})=k.
    \end{equation}

    \begin{enumerate}[i.]
        \item Since $\{x:R_n(x) < n\} \subseteq \{x:R_n(x) < \Exp{\epsilon n}\}$ for all large enough $n$, the first inequality in~\eqref{eq:toshowlimsupgammap} readily follows from the definition of~$\oir$. We now prove the second inequality. Let $n$ be large enough so that $n>\max\{1,\tau_n\}$ and let $1\leq k< n$.
        The map $\{u\in \cA^n : \operatorname{per}(u)=k\}\to \cA^k$ given by $u\mapsto u_{n-k+1}^n$ is injective.
        By~\eqref{eq:obsperiod}, we thus have
        \begin{align*}
            \P\{x:R_n(x) = k\} 
            = \sum_{u\in \cA^n:\operatorname{per}(u)=k}\P_{n+k}(u u_{n-k+1}^n) \leq \sum_{v\in \cA^k}\sup_{u\in \cA^{n}}\P_{n+k}(uv).
        \end{align*}
        As in the proof of Proposition~\ref{prop:LDPtildeRns0}, we assume without loss of generality that $(\tau_n)_{n\in\nn}$ is nondecreasing. Then,~\assref{ud} yields
        $$\P_{n+k}(uv)\leq \P([u_1^{n-\tau_n}]\cap \shift^{-n}[v])\leq   C_{n - \tau_n}\P_{n-\tau_n}(u_1^{n-\tau_n})\P_k(v),$$
        and so 
        \begin{align*}
            \P\{x:R_n(x) = k\} 
            &\leq C_{n - \tau_n} \sum_{v\in \cA^k}\sup_{u\in \cA^{n}}\P_{n-\tau_n}(u_1^{n-\tau_n})\P_k(v)\\
            & =  C_{n - \tau_n} \sup_{u\in \cA^{n}}\P_{n-\tau_n}(u_1^{n-\tau_n}) \\
            &=   C_{n - \tau_n} \sup_{u\in \cA^{n-\tau_n}}\P_{n-\tau_n}(u).
        \end{align*}
        Taking a union over $1\leq k <n$ gives 
        $$
            \P\{x:R_n(x) <n\} \leq (n-1)C_{n - \tau_n} \sup_{u\in \cA^{n-\tau_n}}\P_{n-\tau_n}(u).
        $$
        Comparing with the definition~\eqref{eq:defgammaplus} of~$\gamma_+$, and using that~$\lim_{n\to\infty}\frac {n-\tau_n}n =1$, we deduce that
        \begin{equation}\label{eq:limsupprnmn}
            \limsup_{n\to\infty}\frac 1n \ln \P\{x:R_n(x) < n\} 
            \leq  
            \gamma_+.    
        \end{equation}
        Now, observe that
        $$\P\{x:R_n(x) \in [n, \Exp{n\epsilon})\} \leq \P\{x:V_n(x) \in [1, \Exp{n\epsilon}+1-n)\}\leq  \P\{x:V_n(x) < \Exp{\epsilon n}\}.$$ By~\eqref{eq:limlem37}, which only relies on~\assref{ud}, this implies
        \begin{equation}
        \label{eq:limsupprnmn-2}
            \lim_{\epsilon \to 0}\limsup_{n\to\infty}\frac 1n \ln \P\{x:R_n(x) \in [n, \Exp{n\epsilon})\} \leq \gamma_+.
        \end{equation}
        By definition of~$\oir$, the inequalities~\eqref{eq:limsupprnmn} and~\eqref{eq:limsupprnmn-2} imply $-\oir(0) \leq \gamma_+$, as desired.

        \item As above, the first inequality in~\eqref{eq:toshowuir0r} is immediate by the definition of $\uir$. We now establish the second one.
        Let $p\in\nn$ and $u \in\cA^p$ be arbitrary. 
        By \eqref{eq:obsperiod}, for $n > p$ and $r_n = \lceil \frac n{p}\rceil+1$, we have $[u^{r_n}] \subseteq \{x:R_n(x) \leq p \}$. Therefore,
        \begin{equation}\label{eq:limperur}
            \begin{split}
                \liminf_{n\to\infty} \frac 1n \ln \P\{x:R_n(x)<n\} &\geq \liminf_{n\to\infty}\frac 1n \ln \P\{x:R_n(x) \leq p \} \\
                &\geq  \liminf_{n\to\infty}\frac 1n \ln \P([u^{r_n}]) \\
                &=  \liminf_{r\to\infty}\frac 1{p r} \ln \P([u^{r}]).
            \end{split} 
        \end{equation}
        Since the expression in the last line of~\eqref{eq:limperur} can be made arbitrarily close to $\gamma_+$ by~\assref{pa}, the proof of Part~ii is complete.

        \item The conclusion is just the combination of Parts~i and ii. \qedhere
    \end{enumerate}
\end{proof}

Proposition~\ref{prop:RNat0} is the only place where~\assref{pa} is ever used in our proofs. We do not know if~\assref{pa} can be lifted, nor are we aware of any example of a measure satisfying~\assref{ud} and~\assref{sld} but not~\assref{pa}. We take the remainder of this subsection to briefly discuss what remains true if~\assref{pa} is dropped or weakened. In this discussion, we always assume that $\P$ satisfies~\assref{ud} and~\assref{sld}.

First, if~\assref{pa} is dropped, we remark that Propositions~\ref{prop:RNat0}.i and \ref{prop:irandivspositive} still ensure that $\oir(s) \geq \ir(s)$ for all $s\in \rr$, with $\ir$ defined in \eqref{eq:defirexplicit}. Thus, the weak large-deviation upper bound for $(\frac 1n \ln R_n)_{n\in\nn}$ holds. By retracing the proofs, one easily concludes that also the full large-deviation upper bound holds with the rate function $\ir$, and that $\oqr = \oir^* \leq \ir^*$, where $\oqr$ is defined as in~\eqref{eq:defqr} with a limit superior.

In Proposition~\ref{prop:RNat0},~\assref{pa} is only used to obtain the large-deviation lower bound at $s=0$. More specifically, all we actually derive from~\assref{pa} is that
\begin{equation}\label{eq:paneeded}
    \liminf_{n\to\infty} \frac 1n \ln \P\{x : R_n(x) < n\}\geq \gamma_+.
\end{equation}
So instead of~\assref{pa}, one could have taken~\eqref{eq:paneeded} as an assumption, or any other condition implying it.

In fact, one could even obtain a full LDP for $(\frac 1n \ln R_n)_{n\in\nn}$ without~\eqref{eq:paneeded}. Indeed, if one can show by some means that the limit $D:=  \lim_{n\to\infty} \frac 1n \ln \P\{x : R_n(x) < n\}$ exists, then necessarily $D\leq \gamma_+$ by Proposition~\ref{prop:RNat0}.i, and the proofs can easily be adapted to obtain the full LDP, by merely replacing $-\gamma_+$ with $\min\{-D,\iv(0)\}$ in the definition~\eqref{eq:defirexplicit} of $\ir$. We do not know under what conditions the limit defining $D$ exists, and we have been unable to produce any counter example.

We now return to means of establishing~\eqref{eq:paneeded}.
Under~\assref{pa}, the proof of Proposition~\ref{prop:RNat0}.ii actually shows that the probability of $\{x : R_n(x) < n\}$ is asymptotically captured by periodic orbits of period much smaller than $n$. \assref{pa} can be slightly relaxed so as to take into account words $u$ that can be repeated many times, but not necessarily infinitely many times as \assref{pa} requires:

\begin{definition}[WPA]
    \label{def:wpa}
        \namedlabel{ass:wpa}{WPA}  
            A measure $\P\in \cP_{\textnormal{inv}}(\Omega)$ satisfies the assumption of {\em weak periodic approximation} ({WPA})
            if for every $\epsilon>0$ there exists $M\in \nn$ such that for all $m\geq M$ there is a word $u\in \Omega_{\textnormal{fin}}$ with $|u|\leq \epsilon m$, and such that
            $$\frac 1m \ln \P_{|u| \left\lfloor\frac{m}{|u|}\right\rfloor }\left(u^{\left\lfloor  \frac{m}{|u|} \right\rfloor}\right) \geq \gamma_+ - \epsilon.$$ 
\end{definition}

Then, an easy modification of the proof of Proposition~\ref{prop:RNat0} shows that~\eqref{eq:paneeded} still holds assuming~\assref{wpa} instead of \assref{pa}, and thus so do the conclusions of Theorem~\ref{thm:mainthmC}. 

In order to conclude the discussion, we briefly comment on results from the literature about Poincar\'e recurrence times (see \cite{AV08,AC15,AAG21} and references therein) which can be used to establish~\eqref{eq:paneeded}. 
The Poincar\'e recurrence times are defined by $T_n(x) = \inf\{k \in \nn : \P([x_1^n] \cap \shift^{-k}[x_1^n]) > 0\}$.\footnote{We use the condition $\P([x_1^n] \cap \shift^{-k}[x_1^n]) > 0$ instead of $[x_1^n] \cap \shift^{-k}[x_1^n]\neq \emptyset$ because the discussion is not limited to the subshift $\supp \P$.} 
The asymptotic behavior of $T_n$ is overall very different from that of~$R_n$; in particular $T_n \leq n+\tau_n$ almost surely by~\assref{sld}. However, the following two relations hold $\P$-almost surely: first, $T_n \leq R_n$, and second, $T_n <n$ implies that $R_{n-T_n}\leq T_n$. It follows that if one can show that $\lim_{\epsilon\to 0}\liminf_{n\to\infty}\ln \tfrac 1n \P\{T_n < \epsilon n\} \geq \gamma_+$, then also $\lim_{\epsilon\to 0}\liminf_{n\to\infty}\ln \tfrac 1n \P\{x:R_n(x) < \epsilon n\} \geq \gamma_+$, and in particular~\eqref{eq:paneeded} holds. Such a result is proved in~\cite{AV08} under an assumption called ``Hypothesis~1'', which is very similar to \assref{wpa}. The same bound is obtained in~\cite{AC15} under an assumption called ``Assumption~1'' which, in spirit, also plays role similar to that of~\assref{wpa}.

\section{Proof of the main results}
\label{sec:full}

At this stage, we have proved that if the pair $(\Q, \P)$ is admissible, then $\oiw = \uiw = \iw$, with $\iw$ defined by~\eqref{eq:defiwexplicit}; see Propositions~\ref{prop:RL-Wn} and~\ref{prop:LDPwns0}, and notice that all three functions are infinite on the negative real axis. This implies that the sequence $(\tfrac 1n \ln W_n)_{n\in\nn}$ satisfies the weak LDP with respect to~$\P\otimes\Q$, with the rate function~$\iw$; see the beginning of Section~\ref{sec:RLfuncts}.
In the same way, by Propositions~\ref{prop:irandivspositive} and~\ref{prop:LDPtildeRns0}, we have proved that, if $\P$ satisfies~\assref{ud} and~\assref{sld}, then the sequence $(\tfrac 1n \ln V_n)_{n\in\nn}$ satisfies the weak LDP with respect to~$\P$, with the rate function $\iv$ given in~\eqref{eq:defivexplicit}. 
Finally, combining Propositions~\ref{prop:irandivspositive} and \ref{prop:RNat0}, we have shown that for $\P$ satisfying~\assref{ud}, \assref{sld} and~\assref{pa}, the sequence $(\tfrac 1n \ln R_n) _{n\in\nn}$ satisfies the weak LDP with respect to~$\P$, with the rate function $\ir$ given in~\eqref{eq:defirexplicit}. See Table~\ref{tab:funct-and-rates} for a summary.

Since upper and lower Ruelle--Lanford functions are always lower semicontinuous, we conclude that $\iw$, $\iv$ and~$\ir$ are lower semicontinuous. Alternatively, lower semicontinuity can be checked explicitly using the expressions \eqref{eq:defiwexplicit}, \eqref{eq:defivexplicit} and \eqref{eq:defirexplicit}, together with the fact that the rate function~$\ilpq$ in Theorem~\ref{thm:CJPS} is lower semicontinuous. 

In this section, we promote the weak LDPs to full ones and establish the claimed relations about the rate functions and accompanying pressures. This will conclude the proofs of Theorems~\ref{thm:mainthmA},~\ref{thm:mainthmB} and~\ref{thm:mainthmC}.
The proofs of these theorems, and actually also that of Theorem~\ref{thm:CJPS}, have many (rather standard) arguments in common, which we have extracted as Proposition~\ref{prop:abstractZn}. 

\begin{lemma}
    \label{lem:iwivconvex}
    Under the assumptions of Theorem~\ref{thm:mainthmA}, the rate function $\iw$ is convex. In particular, if $\P$ satisfies~\assref{ud} and~\assref{sld}, then~$\iv$ is convex.
\end{lemma}

\begin{proof}
    We claim that for every $\lambda \in [0,1]$ and $s_1, s_2\in \rr$, 
    $$
        \lambda \iw(s_1)+(1-\lambda)\iw(s_2)\geq \iw(s)
    $$ 
    where $s := \lambda s_1 + (1-\lambda) s_2$. The stated inequality is obvious if $s_1<0$ or $s_2<0$. For $s_1, s_2\geq 0$,
    \begin{align*}
        &\lambda \iw(s_1) + (1-\lambda)\iw(s_2) \\
        &\splitspace = -s + \inf_{r_1 \geq s_1, r_2 \geq s_2} (\lambda r_1 + (1-\lambda)r_2 + \lambda \ilpq(r_1)+(1-\lambda) \ilpq(r_2))\\
        &\splitspace \geq  -s + \inf_{r_1 \geq s_1, r_2 \geq s_2} (\lambda r_1 + (1-\lambda)r_2 +  \ilpq(\lambda r_1 + (1-\lambda)r_2))\\
        &\splitspace  = -s  +\inf_{r \geq s} (r+\ilpq(r)) = \iw(s),
    \end{align*}
    where the inequality in the second line relies on the convexity of $\ilpq$, by Theorem~\ref{thm:CJPS}.i. The second part of the lemma follows from the identity $\iv =\iw$ when $\P=\Q$; recall Remark~\ref{eq:iveqiw}.
\end{proof}

\begin{lemma}\label{lem:ubexpWna}
    Suppose that $\Q$ satisfies~\assref{sld}. Then, for every $\alpha > 0$, there exists an $\Exp{o(n)}$-sequence $(\kappa_{\alpha,n})_{n\in\nn}$ such that 
    $$
    \int W_n(u,y)^\alpha\dd\Q(y) \leq \kappa_{\alpha,n} \Q_n(u)^{-\alpha}
    $$
    for each $u\in \supp \Q_n$.
\end{lemma}

\begin{proof}
    Fix $n\in \nn$, $u\in \supp \Q_n$, and let $q := 1 - C_n^{-1}\Q_n(u)$. Without loss of generality, we assume that $C_n\geq C_1 > 1$, so that $0<1-C_1^{-1}\leq q < 1$.
    For all $t\geq 0$, the bound~\eqref{eq:UB2only1bound} of Lemma~\ref{lem:UB-Q-Wn-a-y} yields
    \begin{align*}
        \Q\{y : W_n(u,y)> t\} 
            &= \Q\{y : W_n(u,y)\geq \lfloor t\rfloor+1\} \\
            &\leq q^{\left\lfloor \frac{\lfloor t\rfloor}{n + \tau_n}\right\rfloor} \\
            &\leq q^{ \frac{t}{n + \tau_n}-2} \\
            &\leq  (1-C_1^{-1})^{-2}q^{ \frac{t}{n + \tau_n}}.
    \end{align*}
    Then, by a standard consequence of Fubini's theorem (see e.g.\ Theorem~8.16 in~\cite{Ru}),
    \begin{align*}
        \int W_n(u,y)^\alpha \dd \Q(y) &= \alpha \int_0^\infty t^{\alpha -1} \Q\{y : W_n(u,y)> t\}\dd t\\
        &\leq   \frac  \alpha {(1-C_1^{-1})^2} \int_0^\infty t^{\alpha -1} q^{\frac {t}{n + \tau_n}}\dd t\\
        & = \frac  \alpha {(1-C_1^{-1})^2} \Gamma(\alpha)\left(\frac {1}{n + \tau_n}\right)^{-\alpha}\left(-\ln q\right)^{-\alpha}.
    \end{align*}
    Further using that $-\ln(q) = -\ln(1-C_n^{-1}\Q_n(u))\geq C_n^{-1}\Q_n(u)$ completes the proof.
\end{proof}

We are now in a position to prove the main results. With Proposition~\ref{prop:abstractZn} from \appref{app:weak-to-full} and the weak LDPs at hand, the proofs of Theorems~\ref{thm:mainthmA}--\ref{thm:mainthmC} only consist in providing a few remaining estimates specific to each case.

\subsection{Proof of Theorem~\ref{thm:mainthmA}}

    Since $(\frac 1n \ln W_n)_{n\in\nn}$ satisfies the weak LDP, the assumptions of Proposition~\ref{prop:abstractZn} are satisfied with $Z_n := \frac 1n \ln W_n$ on $(\Omega_*, \P_*) := (\Omega\times \Omega, \P\otimes \Q)$, with the convex (recall Lemma~\ref{lem:iwivconvex}) rate function $I:=\iw$. We now show that the conclusions of Theorem~\ref{thm:mainthmA} follow from Proposition~\ref{prop:abstractZn}.

    We first prove that $\qw$ exists and that $\qw = \iw^*$. For this we define $\uqw$ and $\oqw$ as the limit inferior and limit superior corresponding to the definition \eqref{eq:defqw} of $\qw$. 
    The bound $\uqw\geq \iw^*$ is provided by Proposition~\ref{prop:abstractZn}.\ref*{part:qgeqistar}, and by Proposition~\ref{prop:abstractZn}.\ref*{part:qleqistarneg} we have $\qw(\alpha) = \iw^*(\alpha)$ for all $\alpha <0$. Since $\iw\leq \ilpq$ by definition, we have $\iw^*(0)  \geq  \ilpq^*(0) = \qlpq(0)= 0 = \qw(0)$, where we have used Theorem~\ref{thm:CJPS}.ii. 
    Finally, for $\alpha > 0$, Lemma~\ref{lem:ubexpWna} gives 
    \begin{equation}\label{eq:iintwn}
        \begin{split}
            \iint W_n(x,y)^\alpha \dd\P(x)\dd\Q(y)
            &= \sum_{u\in\supp \P_n}\P_n(u)\int W_n(u,y)^\alpha \dd\Q(y) \\
            &\leq \kappa_{n,\alpha} \sum_{u\in\supp \P_n}\P_n(u) \Q_n(u)^{-\alpha}.
        \end{split}
    \end{equation}
    Note that we have used the absolute continuity granted by admissibility of the pair~$(\P,\Q)$. We conclude that 
    $$
    \oqw(\alpha)\leq \qlpq(\alpha) = \ilpq^*(\alpha) \leq \iw^*(\alpha),
    $$
    where from left to right, we have used \eqref{eq:iintwn}, Theorem~\ref{thm:CJPS}.ii, and the fact that $\iw \leq \ilpq$ by definition.
    We have thus proved that $\qw = \iw^*$. By Proposition~\ref{prop:abstractZn}.\ref*{part:qIfullcases}, the weak LDP then extends to a full one. This completes the proof of Theorem~\ref{thm:mainthmA}.i.

    For Part ii, we have already established that $\qw = \iw^*$ (in particular $\qw$ exists as a limit), and since $\iw$ is convex and lower semicontinuous, this implies that also $\iw = \qw^*$.
    To establish~\eqref{eq:formqW}, note that
     \begin{align*}
        \qw(\alpha)=\iw^*(\alpha) = \sup_{s\geq 0} (\alpha s - \inf_{r\geq s}(r-s+\ilpq(r))) = \sup_{r\geq s\geq 0} ((\alpha +1) s -r - \ilpq(r)).
    \end{align*}
    Since the quantity to optimize is linear in $s$, it suffices to consider the extremal points $s\in \{0,r\}$, which yields, using again Theorem~\ref{thm:CJPS}.ii,
    \begin{align*}
        \qw(\alpha)&= \sup_{r\geq 0} (\max\{0, (\alpha+1)r\} -r - \ilpq(r)) \\
        &=  \max\left\{\sup_{r\geq 0}(-r - \ilpq(r)), \sup_{r\geq 0}(\alpha r - \ilpq(r))\right\} \\
        &=  \max\{\qlpq(\alpha), \qlpq(-1)\},
    \end{align*}
    so the proof of Theorem~\ref{thm:mainthmA}.ii is complete.

    In order to prove Theorem~\ref{thm:mainthmA}.iii, we simply note that when $\Q=\P$, then by \eqref{eq:formqW} and \eqref{eq:qppbdd}, we have $\qw(1) = \qlpp(1) = \htop(\supp \P) <\infty$, so Proposition~\ref{prop:abstractZn}.\ref*{part:qIfullcases2} applies.
\hfill\qed

\subsection{Proof of Theorem~\ref{thm:mainthmB}}

    The proof is almost the same as that of Theorem~\ref{thm:mainthmA}, applying this time Proposition~\ref{prop:abstractZn} to $ Z_n := \frac 1n \ln V_n$ on $(\Omega_*, \P_*) := (\Omega, \P)$, with the convex rate function $I:=\iv$. 
    We will only require a slightly more involved argument to derive the inequality $\oqv(\alpha)\leq \iv^*(\alpha)$ when $\alpha > 0$, where $\oqv$ is defined by taking the limit superior in the definition \eqref{eq:defqv} of $\qv$.
    For $\alpha >0$, recalling that $W_n(u, x)\leq W_n(u, \shift^k x)+k$, we obtain
    \begin{align*}
        \int V_n^\alpha \dd\P 
            &=  \sum_{u\in \cA^n}  \int 1_{[u]}(x) W_n(u, \shift^{n}x)^\alpha \dd\P(x) \\
            &\leq \sum_{u\in \cA^n}\int  1_{[u]}(x) (W_n(u, \shift^{n+\tau_n}x)+\tau_n)^\alpha\dd\P(x).
    \end{align*}
    Then, by \assref{ud},
    \begin{align*}
       &\int  1_{[u]}(x) (W_n(u, \shift^{n+\tau_n}x)+\tau_n)^\alpha\dd\P(x) \\ 
       &\splitspace =  \sum_{j=1}^\infty \P([u]\cap \shift^{-n-\tau_n}\{x: W_n(u,x)=j\})(j+\tau_n)^\alpha\\
       &\splitspace \leq  C_n  \sum_{j=1}^\infty \P_n(u) \P\{x: W_n(u,x)=j\}(j+\tau_n)^\alpha\\
       &\splitspace = C_n\P_n(u)\int (W_n(u,x)+\tau_n)^\alpha\dd\P(x),
    \end{align*}
    so
    \begin{align*}
        \int V_n^\alpha \dd\P & \leq    C_n  \sum_{u\in \cA^n}\P_n(u)\int (W_n(u,x)+\tau_n)^\alpha\dd\P(x) \\
    &  \leq   C_n  \sum_{u\in \cA^n} \P_n(u)\int (W_n(u,x)(1+\tau_n))^\alpha\dd\P(x) \\
    &  \leq   C_n  (1+\tau_n)^\alpha \kappa_{\alpha,n} \sum_{u\in \cA^n}  \P_n(u)^{1-\alpha},
    \end{align*}
    where the last inequality was obtained using
    Lemma~\ref{lem:ubexpWna} with $\Q = \P$. 
    By this and Theorem~\ref{thm:CJPS}.ii, and since $\iv \leq \ilpp$ by definition, we conclude that $\oqv(\alpha) \leq \qlpp(\alpha) = \ilpp^*(\alpha)\leq  \iv^*(\alpha)$, as claimed.

    The same arguments as in the proof of Theorem~\ref{thm:mainthmA} then provide the full LDP together with the Legendre--Fenchel duality relations and  \eqref{eq:formqV}, which is merely a special case of \eqref{eq:formqW}.

    Next, the exponential tightness and goodness assertions in Theorem~\ref{thm:mainthmB}.i follow, as in the proof of Theorem~\ref{thm:mainthmA}.iii, from Proposition~\ref{prop:abstractZn}.\ref*{part:qIfullcases2} and the bound $\qv(1) = \qlpp(1)=\htop(\supp \P)<\infty$.

    Finally, \eqref{eq:reliv0qlppm1} is proved as follows: by  Lemma~\ref{lem:propzgammaplus} we have $\qlpp(-1)\leq \gamma_+$ and by \eqref{eq:formqV} and Legendre--Fenchel duality we have $\iv(0) = \sup_{\alpha \in \rr}(-\qv(\alpha)) = -\qlpp(-1)$; this can also be obtained using \eqref{eq:toshowLDPvns0}.
\hfill\qed

\subsection{Proof of Theorem~\ref{thm:mainthmC}}
    Once more, we plan to apply Proposition~\ref{prop:abstractZn}, this time with $Z_n := \frac 1n \ln R_n$ on $(\Omega_*, \P_*) := (\Omega, \P)$ and the rate function $I:=\ir$, keeping in mind that the latter is not convex in general.

    The proof that $\qr = \ir^*$ follows the same argument as in Theorems~\ref{thm:mainthmA} and \ref{thm:mainthmB}, and once again, only the proof of the inequality $\oqr(\alpha)\leq \ir^*(\alpha)$ when $\alpha > 0$ needs to be adapted, where $\oqr$ is again defined by taking the limit superior in the definition \eqref{eq:defqr} of $\qr$. We have, for $\alpha >0$,
    \begin{align*}
        \int R_n^\alpha \dd\P
        &=  \sum_{u\in\cA^n} \int 1_{[u]}(x) W_n(u, \shift x)^\alpha \dd\P(x)\\
        &\leq   \sum_{u\in\cA^n} \int 1_{[u]}(x)(W_n(u, \shift^{n+\tau_n}x)+n+\tau_n-1)^\alpha \dd\P(x)\\ 
        & \leq C_n  (n+\tau_n)^\alpha \kappa_{\alpha,n} \sum_{u\in \cA^n}  \P_n(u)^{1-\alpha},
    \end{align*}
    where the steps giving the last inequality are exactly as in the proof of Theorem~\ref{thm:mainthmB} with $\tau_n$ replaced by $\tau_n+n-1$.
    As previously, we conclude that $\oqr(\alpha) \leq \qlpp(\alpha) = \ilpp^*(\alpha)\leq  \ir^*(\alpha)$, since $\ir \leq  \ilpp$. We thus have $\qr = \ir^*$.

    Unlike in the previous theorems, the rate function is not convex in general and we cannot invert the Legendre--Fenchel transform. For the same reason, we cannot apply Proposition~\ref{prop:abstractZn}.\ref*{part:qIfullcases} to obtain the full LDP. Nevertheless, since  $\qr(1) = \qlpp(1)=\htop(\supp \P)<\infty$, the full LDP, exponential tightness and goodness of $\ir$ follow at once from Proposition~\ref{prop:abstractZn}.\ref*{part:qIfullcases2}. This completes the proof of Part~i.
    
    \medskip 

    We now turn to the proof of Part~ii. We have already seen that $\qr = \ir^*$, so it remains to prove~\eqref{eq:formqR}, which we do by singling out the pathological case where $\gamma_-=0$. 
    
    \medskip

    \noindent\emph{Case 1: $\gamma_- = 0$.} In this case, $\gamma_+ = 0$ as well by \eqref{eq:ineqgammapm}, so Theorem~\ref{thm:CJPS}.iii yields that $\ilpp(0) = 0$ and $\ilpp(s) = \infty$ for $s\neq 0$. It follows from its definition that $\ir=\ilpp$, and taking the Legendre--Fenchel transform gives $\qr=\qlpp = 0$. In particular, both sides of \eqref{eq:formqR} vanish identically.

    \medskip

    \noindent\emph{Case 2: $\gamma_- < 0$.} Since $\ir(s) = \iv(s)$ for all $s>0$ by definition, and since $\lim_{s\downarrow 0}\iv(s) = \iv(0)$ when $\gamma_- < 0$, we find
    \begin{align*}
        \qr(\alpha) &= \sup_{s\in \rr}(\alpha s - \ir(s)) \\
        &= \max\left\{-\ir(0), \sup_{s>0}(\alpha s - \iv(s))\right\} \\
        &=  \max\{-\ir(0),\iv^*(\alpha)\}. 
    \end{align*}
    Further using that $-\ir(0)=\gamma_+$ by definition and that $$\iv^*(\alpha)=\qv(\alpha) =  \max\{\qlpp(\alpha), \qlpp(-1)\}$$ by~\eqref{eq:formqV}, we conclude that
    $\qr(\alpha) =  \max\{\gamma_+,\qlpp(\alpha), \qlpp(-1)\}$. In view of \eqref{eq:reliv0qlppm1}, the quantity $\qlpp(-1)$ can be omitted from the maximum, and \eqref{eq:formqR} follows.

    \medskip

    Part~iii of Theorem~\ref{thm:mainthmC}, about (non)convexity of $\ir$, is proved in \secref{sec:nonconvex}.
\hfill\qed

\begin{remark}\label{rem:direxpotight}
    In order to follow a common route for the proof of our three main theorems, our arguments to promote the weak LDP to a full one did not rely on exponential tightness, since the latter does not hold in general in the setup Theorem~\ref{thm:mainthmA}.i--ii. We have instead made extensive use of the properties of the pressures.
    We now briefly argue that in the context of Theorems~\ref{thm:mainthmA}.iii, \ref{thm:mainthmB} and \ref{thm:mainthmC}, exponential tightness could have been proved more directly, and the full LDP could have thus been obtained without any reference to pressures.
    
    In the setup of Theorem~\ref{thm:mainthmA}.iii, that is in the case of $(\frac 1n \ln W_n)_{n\in\nn}$ with $\Q=\P$, the idea is that multiplying~\eqref{eq:UB2only1bound} in the case $m \sim \Exp{nM}$ by~$\P_n(u)$ and summing over~$u\in\cA^n$ yields terms that fall in either of two categories: those with~$\P_n(u) \gg \Exp{-nM}$ decay superexponentially, while those with $\P_n(u) \lesssim \Exp{-nM}$ are controlled in view of the exponential tightness estimate~\eqref{eq:expotightPn}. The argument relies on~\assref{sld} only, and the details are easily written by following the proof of the estimate~\eqref{eq:toshowUB}. Further assuming~\assref{ud}, exponential tightness can be extended to $(\frac 1n \ln R_n)_{n\in\nn}$ and~$(\frac 1n \ln V_n)_{n\in\nn}$ by introducing a minor variation of Lemma~\ref{lem:UB-Wn-to-Rn} where balls are replaced by intervals of the form $(s, \infty)$ and $(s-\epsilon, \infty)$.
\end{remark}

\subsection{Nonconvexity in Theorem~\ref{thm:mainthmC}}
\label{sec:nonconvex}

In this subsection, we prove the numerous relations stated in Theorem~\ref{thm:mainthmC}.iii. The topological entropy $\htop(\supp \P)$ as well as the notion of measure of maximal entropy (MME) on $\supp \P$ are recalled in \appref{app:subshiftsupp}.
\begin{proposition}
    Let $\P$ satisfy the assumptions of Theorem~\ref{thm:mainthmC}, and consider the following properties:
    \begin{multicols}{2}
        \begin{enumerate}[(a)]
            \item \edef\@currentlabel{(a)} \label{part:olda} $\ir$ is convex; 
            \item \edef\@currentlabel{(b)}\label{part:oldb} $\ir =\iv$;
            \item \edef\@currentlabel{(c)}\label{part:bbp} $\qr =\qv$;
            \item \edef\@currentlabel{(d)}\label{part:oldc} $\gamma_+ = \qlpp(-1)$; 
            \item \edef\@currentlabel{(e)}\label{part:olde} $\ilpp(-\gamma_+) = 0$; 
            \item \edef\@currentlabel{(f)}\label{part:oldg} $\qlpp(\alpha) = -\gamma_+\alpha$ for all $\alpha \leq 0$;
            \item \edef\@currentlabel{(g)} \label{part:oldh} $\gamma_+ = \gamma_- $;
            \item \edef\@currentlabel{(h)} \label{part:oldi} $\ilpp(s) = \infty$ for all $s \in \rr\setminus\{-\gamma_+\}$;
            \item \edef\@currentlabel{(i)} \label{part:oldj} $h(\P) = \htop(\supp \P)$, i.e.\ $\P$ is a MME on $\supp \P$.
        \end{enumerate}
    \end{multicols}
    Then, Properties~\ref{part:olda}--\ref{part:oldg} are equivalent, Properties~\ref{part:oldh}--\ref{part:oldj} are equivalent, and Properties~\ref{part:oldh}--\ref{part:oldj} imply Properties~\ref{part:olda}--\ref{part:oldg}. If, in addition, $\tau_n=O(1)$ and $C_n=O(1)$, then Properties~\ref{part:olda}--\ref{part:oldj} are equivalent.
\end{proposition}

\begin{proof}
    Let us summarize the information we have gathered on the rate functions. First, $\iv(s)$ and $\ir(s)$ coincide for $s \neq 0$ by comparison of~\eqref{eq:defivexplicit} and~\eqref{eq:defirexplicit}. By the same two equations, and since we have already established that $ -\gamma_+ \leq - \qlpp(-1)$ in Lemma~\ref{lem:propzgammaplus} and that $-\qlpp(-1)=\iv(0)$ in Theorem~\ref{thm:mainthmB}.ii, we obtain
    \begin{equation}\label{eq:chaineI0}
        0\leq \ir(0) = -\gamma_+ \leq - \qlpp(-1) = \iv(0) \leq \ilpp(0).
    \end{equation}
    We now prove that \ref{part:olda}--\ref{part:oldg} are equivalent.

    \begin{description}
        \item[\textnormal{\ref{part:olda}$\Rightarrow$\ref{part:oldb}}] 
        
        If $\gamma_+ = 0$, Theorem~\ref{thm:CJPS}.iii.a implies that $\ilpp(0) = 0$, so by~\eqref{eq:chaineI0} we obtain $\iv(0)=\ir(0)$ and thus $\iv = \ir$.\footnote{Note that the argument when $\gamma_+ = 0$ does not actually use \ref{part:olda}; in fact properties \ref{part:olda}--\ref{part:oldg} always hold true if $\gamma_+ = 0$.}

        Assume now $\gamma_+<0$. By Theorem~\ref{thm:CJPS}.iii.a and the formulae~\eqref{eq:defivexplicit} and~\eqref{eq:defirexplicit}, the functions~$\iv$ and~$\ir$ are finite on $[0,-\gamma_+]$, and equal on $(0,-\gamma_+]$. 
        Now, if \ref{part:olda} holds, then both $\iv$ and $\ir$ are convex, and since they are also lower semicontinuous, their restriction to $[0,\gamma_+]$ is continuous, so  $\ir(0) = \iv(0)$, and thus $\ir = \iv$.

        \item[\textnormal{\ref{part:olda}$\Leftarrow$\ref{part:oldb}}]  The function~$\iv$ is convex by Lemma~\ref{lem:iwivconvex}.

        \item[\textnormal{\ref{part:oldb}$\Rightarrow$\ref{part:bbp}}] By Theorems~\ref{thm:mainthmB}.ii and \ref{thm:mainthmC}.ii, assuming \ref{part:oldb} yields $\qv = \iv^* = \ir^* = \qr$.

        \item[\textnormal{\ref{part:bbp}$\Rightarrow$\ref{part:oldc}}] As mentioned in Remarks~\ref{rem:sstar} and~\ref{rem:alphastar}, for all $\alpha \leq -1$ we have $\qv(\alpha) = \qlpp(-1)$ and $\qr(\alpha) = \gamma_+$, so by \ref{part:bbp} we find $\qlpp(-1) = \gamma_+$.

        \item[\textnormal{\ref{part:oldc}$\Rightarrow$\ref{part:oldb}}]  Combining \ref{part:oldc} and \eqref{eq:chaineI0} gives again $\iv(0)=\ir(0)$.
        
        \item[\textnormal{\ref{part:oldb}$\Rightarrow$\ref{part:olde}}] 
        Assuming \ref{part:oldb}, the identity $\ir(0)=\iv(0)$ reads $-\gamma_+ = \inf_{s\geq 0}(s+\ilpp(s))$ according to~\eqref{eq:defivexplicit} and~\eqref{eq:defirexplicit}, so that  for all~$\epsilon > 0$, there exists $s_\epsilon \geq 0$ such that $\ilpp(s_\epsilon) < -\gamma_+ - s_{\epsilon} + \epsilon$. Since also $s_\epsilon\geq -\gamma_+$ by Theorem~\ref{thm:CJPS}.iii.a, we have 
        $$
        0\leq \ilpp(s_\epsilon) < -\gamma_+ - s_{\epsilon} + \epsilon \leq \epsilon.
        $$
        It follows that $s_\epsilon \to -\gamma_+$ and $\ilpp(s_\epsilon) \to 0$ as $\epsilon \to 0$. By nonnegativity and lower semicontinuity of~$\ilpp$, this allows to conclude that $\ilpp(-\gamma_+) = 0$.

        \item[\textnormal{\ref{part:olde}$\Rightarrow$\ref{part:oldg}}] 
        By Theorem~\ref{thm:CJPS}.iii.a and \ref{part:olde}, we have $\ilpp(-\gamma_+) = 0$ and $\ilpp(s) = \infty$ for $s < -\gamma_+$. Since $\ilpp$ is nonnegative, for all $\alpha \leq 0$, the duality relations~\eqref{eq:LdualIpq} imply that $\qlpp(\alpha) = \sup_{s\geq -\gamma_+}(\alpha s - \ilpp(s)) = -\alpha\gamma_+$, where the supremum is reached at $s=-\gamma_+$.

        \item[\textnormal{\ref{part:oldg}$\Rightarrow$\ref{part:oldc}}] The latter is a special case of the former with $\alpha = -1$.
        
    \end{description}

    We now prove the equivalence of \ref{part:oldh}--\ref{part:oldj}.

    \begin{description}
        \item[\textnormal{\ref{part:oldh}$\Leftrightarrow$\ref{part:oldi}}] This is a consequence of Theorem~\ref{thm:CJPS}.iii.a.
                
        \item[\textnormal{\ref{part:oldh}$\Leftrightarrow$\ref{part:oldj}}] This is a consequence of Proposition~\ref{prop:MMEprops}.i.
    \end{description}

    To show that  \ref{part:oldh}--\ref{part:oldj} imply \ref{part:olda}--\ref{part:oldg} we remark that if \ref{part:oldi} holds, then we have \ref{part:olde} because $\inf_{s\in \rr}\ilpp(s) = 0$.
    Finally, to obtain the last statement, we prove in Proposition~\ref{prop:sconvtau} below that if $\tau_n=O(1)$ and $C_n=O(1)$, then \ref{part:oldc} implies \ref{part:oldi}.
\end{proof}

\begin{proposition}\label{prop:sconvtau}
    Suppose that $\P$ satisfies~\assref{ud} and~\assref{sld} with $\tau_n=O(1)$ and $C_n=O(1)$. If $\qlpp(-1) = \gamma_+$, then $\ilpp(s) = \infty$ for all $s \in \rr\setminus\{-\gamma_+\}$.
\end{proposition}

\begin{proof}
    Seeking a contradiction, suppose that there is $s>-\gamma_+$ such that $\ilpp(s) <\infty$  (recall that $\ilpp(s)=\infty$ for all $s<-\gamma_+$ by Theorem~\ref{thm:CJPS}.iii.a).
    The idea of the proof is that, if such $s$ exists, then there are words whose probability decreases like $\Exp{-sn}\ll \Exp{\gamma_+ n}$, and that these words appear often enough to make $\qlpp(-1)$ strictly smaller than $\gamma_+$, which contradicts the assumption.

    Let $\tau:= \sup_{n\in \nn}\tau_n < \infty$, $C:=\sup_{n\in \nn}C_n < \infty$, and $\epsilon := s+\gamma_+>0$. For $\ell, m\in \nn$ and $w\in \cA^m$, the computations below will involve
    \begin{align*}
        R_1 &:= \frac {\ln  (1 - C\P_m(w))}{2(m+\tau)} + \frac {\ln 2} \ell + \frac {\ln C} {2\ell}  - \frac {\ln  (1 - C\P_m(w))}\ell
    \intertext{and}
        R_2 &:= \frac {m}{2\ell}\left(-\gamma_+ + \frac {\ln \P_m(w)} {m} -\frac{2\tau \gamma_+}{m}+\frac{3 \ln C}{m}\right).
    \end{align*}

    Let us show that $m$, $\ell$, and $w$ can be chosen so that $R_1$ and $R_2$ are negative, which will later lead to the contradiction we are seeking. Since $\ilpp(s) <\infty$, the LDP implies that for all $m$ large enough we can choose a word $w\in \cA^m$ such that $\P_m(w)\leq \Exp{-m(s-\frac 12 \epsilon)}$, and thus $\gamma_+- \frac {\ln \P_m(w)} {m} \geq \frac \epsilon 2$. We fix $m$ large enough and $w\in \cA^m$ so that not only the above holds, but also $-\frac{2\tau \gamma_+}{m}+\frac{3 \ln C}{m} \leq \frac \epsilon 4$ and $\Exp{-m(s-\frac 12 \epsilon)} < C^{-1}$ (so in particular $1-C\P_m(w)>0$). Thus, we have 
    $$
        R_2 \leq -\frac{m\epsilon}{8\ell}<0,
    $$ 
    and, since the first term in $R_1$ is negative, taking $\ell$ large enough yields
    $$
        R_1 \leq  \frac {\ln  (1 - C\P_m(w))}{4(m+\tau)}<0.
    $$ 
    We remark that the above condition on $\ell$ implies in particular that $\ell>2(m+\tau)$, which will be useful to have in mind below.

    For the remainder of the proof, $m$, $w$ and~$\ell$ are fixed as above, so that $R_1$ and $R_2$ are negative. Let $k\in \nn$ and consider the intervals $L_j := [1+j\ell, (j+1)\ell -\tau]$ for $j=0, 1, \dots, 2k-1$. These $2k$ intervals are separated by gaps of length $\tau$. Let $A_k$ be the set of points $x\in \Omega$ for which at least $k$ out of the $2k$ intervals contain {\em no} occurrence of $w$ (an interval $[i_1, i_2]$ contains no occurrence of $w$ if there is no $i_1 \leq i \leq i_2-m+1$ such that $x_i^{i+m-1}=w$). Clearly,
    \begin{equation}\label{eq:sump2kl}
        \begin{split}
            \sum_{u\in \cA^{2k\ell}} \P^2_{2k\ell}(u)  &= \int_{A_k} \P_{2k\ell}([x_1^{2k\ell}]) \dd \P(x) +\int_{A_k^\complement} \P_{2k\ell}([x_1^{2k\ell}]) \dd \P(x) \\
            & \leq \P(A_k)\sup_{x\in \Omega} \P_{2k\ell}(x_1^{2k\ell}) +  \sup_{x\in A_{k}^\complement} \P_{2k\ell}(x_1^{2k\ell}).     
        \end{split}
    \end{equation}
    We will show below that  
    \begin{align}
        \label{eq:limsupr2}
        \limsup_{k\to\infty} \frac 1{2k\ell} \ln (\P(A_k)\sup_{x\in \Omega} \P_{2k\ell}(x_1^{2k\ell})) &\leq  \gamma_+ + R_1,\\
        \label{eq:limsupr1}
        \limsup_{k\to\infty} \frac 1{2k\ell} \ln \sup_{x\in \cA_k^{\complement}}\P_{2k\ell}(x_1^{2k\ell})&\leq \gamma_+ + R_2.
    \end{align}
    Combining this with~\eqref{eq:sump2kl} yields, since the limit defining $\qlpp(-1)$ exists,
    \begin{align*}
        \qlpp(-1) 
            &= \lim_{n\to\infty}\frac 1n \ln \sum_{u\in \cA^{n}}\P^2_{n}(u) \\ 
            &= \limsup_{k\to\infty} \frac 1{2k\ell}\sum_{u\in \cA^{2k\ell}} \P^2_{2k\ell}(u) \\
            &\leq \max\left\{\gamma_+ + R_2, \gamma_+ + R_1\right\} \\
            &< \gamma_+,
    \end{align*}
    which contradicts the assumption, and thus proves the claim.

    \smallskip\noindent{\em Proof of~\eqref{eq:limsupr2}}. 
     By Remark~\ref{rem:boundednesscompared}, since $\tau_n = O(1)$, we can assume without loss of generality that the sequence $(C_n)_{n\in \nn}$ required for~\eqref{eq:udeventseq} and~\eqref{eq:sldeventseq} to hold, is also bounded by $C$.

    The event ``there is no occurrence of $w$ in the interval $L_j$'' can be written as $\shift^{-j\ell}U$, where $U := \{x:W_m(w, x)>\ell+1-m-\tau\} \in \cF_{\ell-\tau}$. By the definition of $A_k$,
    \begin{equation}\label{eq:AbigcupM}
        A_k\subseteq \bigcup_{\substack{M\subseteq \{0, 1, \dots, 2k-1\}\\|M|=k}}\bigcap_{j\in M} \shift^{-j\ell}U.
    \end{equation}
    Using the bound~\eqref{eq:UB2only1bound} of Lemma~\ref{lem:UB-Q-Wn-a-y}, we obtain
    \begin{align*}
        \P(U)&\leq  (1 - C\P_m(w))^{\lfloor \frac{\ell+1-m-\tau}{m+\tau}\rfloor} \leq (1 - C\P_m(w))^{ \frac{\ell}{m+\tau}-2}.
    \end{align*}
    Since $U\in \cF_{\ell-\tau}$, we can apply~\assref{ud} recursively $k-1$ times to the intersection in~\eqref{eq:AbigcupM}. This and a simple union bound (noting that the union in~\eqref{eq:AbigcupM} contains $\binom{2k}{k}$ terms) yields
    $$
    \P(A_k)\leq \binom{2k}{k} C^{k-1}(1 - C\P_m(w))^{ \frac{k \ell}{m+\tau}-2k}.
    $$
    Since $\lim_{n\to\infty}\frac 1{2k} \ln  \binom{2k}{k} = \ln 2$, we have proved that $\limsup_{k\to\infty} \frac 1{2k\ell} \ln \P(A_k) \leq  R_1$. Combining this with the definition of $\gamma_+$ completes the proof of~\eqref{eq:limsupr2}.

    \smallskip\noindent{\em Proof of~\eqref{eq:limsupr1}}. 
    First, let $a_n := \sup_{u\in \cA^n}\ln  \P_n(u)$. Clearly, the sequence $(a_n)_{n\in \nn}$ is nonincreasing, and by~\assref{sld}, for all $n,m\in \nn$, there is $0\leq \ell\leq \tau$ such that $a_{n+m} \geq a_{n+m+\ell} \geq a_n + a_m -\ln C$.
    By applying Fekete's lemma to the sequence $(\ln C-a_n)_{n\in\nn}$, we obtain 
    $\gamma_+ = \sup_{n\in \nn} \frac{a_n-\ln C}{n}$, so for every $x\in \Omega$ and $n\in\nn$,
        \begin{equation}\label{eq:lnpnc}
        \ln \P_n(x_1^n) \leq n\gamma_++\ln C.
    \end{equation}

    Let $x\in A_{k}^{\complement}$. By construction, we can choose $k$ occurrences of $w$ in $x_1^{2k\ell}$, which are pairwise separated by gaps of length at least $\tau$ (possible additional occurrences of $w$ in $x_1^{2k\ell}$ do not play any role in the argument). In other words, we can write $x_1^{2k\ell} = u_0 w u_1 w u_2 \dots u_{k-1} w u_k$ where $|u_0|\geq 0$ and $|u_i|\geq \tau$ for all $1\leq i\leq k$. We then apply~\assref{ud} before and after each occurrence of $w$ of $x_1^{2k\ell}$.\footnote{Remark that decoupling is not necessarily used near the boundaries of the intervals $L_j$.} More precisely, in order to apply~\assref{ud}, we discard a prefix of length $\tau$ in each $u_i$ with $1\leq i \leq k$, and we discard a suffix of length $\tau$ in each $u_i$ with $0\leq i \leq k-1$; the $u_i$ whose length is insufficient for this to be done are discarded entirely. By using \eqref{eq:lnpnc} to bound the probability of what remains of the $u_i$ at the end of this process, we easily obtain the estimate
    \begin{align*}
        \P_{2k\ell}(x_1^{2k\ell})&\leq \P_m(w)^k C^{3k+1}\Exp{\gamma_+(2k\ell - km - 2k\tau)},
    \end{align*}
    where the $C^{3k+1}$ comes from the fact that we have used~\assref{ud} at most $2k$ times and~\eqref{eq:lnpnc} at most $k+1$ times, and where $2k\ell - km - 2k\tau$ is the minimum cumulative length of the words to which we have applied~\eqref{eq:lnpnc}. By the definition of $R_2$, this establishes~\eqref{eq:limsupr1}, and thus the proof is complete.
\end{proof}

\section{Relation to almost sure convergence results}
\label{sec:LLN-v2}

In this section we discuss the complementarity of our LDPs and existing almost sure convergence results in the literature.

\subsection{A brief review of results on almost sure convergence}
\label{ssec:as-remix}

Throughout this subsection, we consider two measures $\P,\Q\in \cP_{\textnormal{inv}}$, and we only assume that~$\Q$ satisfies~\assref{ud} and~\assref{sld}. 
Under these conditions, the limits 
\begin{equation}
    \label{eq:hnP-to-hP}
        h_\P(x) := \lim_{n\to\infty} -\frac 1n \ln \P_n(x_1^n)
\end{equation}
and
\begin{equation}
    \label{eq:hnQ-to-hQ}
        h_\Q(x) := \lim_{n\to\infty} -\frac 1n \ln \Q_n(x_1^n)
\end{equation}
exist for $\P$-almost all~$x$. Moreover,
\begin{equation}\label{eq:intsmb}
    \int h_\P(x) \dd\P(x)
    = \lim_{n\to\infty} -\frac{1}{n} \sum_{u\in\cA^n} \P_n(u) \ln \P_n(u)=: h(\P) 
\end{equation}
and 
\begin{equation}\label{eq:intcrosssmb}
    \int h_\Q(x) \dd\P(x)
    = \lim_{n\to\infty} -\frac{1}{n} \sum_{u\in\cA^n} \P_n(u) \ln \Q_n(u) =: \Sc(\P|\Q) .
\end{equation}

The relations \eqref{eq:hnP-to-hP} and~\eqref{eq:intsmb} are the contents of the well-known SMB theorem. While often stated assuming ergodicity, the general form of the SMB theorem, which only requires shift invariance, is obtained as a special case of the Brin--Katok formula~\cite{BK83}; see also Theorem~9.3.1 of~\cite{OlVi} for a direct proof.

On the other hand, \eqref{eq:hnQ-to-hQ} and~\eqref{eq:intcrosssmb} are much less general, and the literature on them is comparatively sparse. Even the limit defining~$\Sc(\P|\Q)$ in~\eqref{eq:intcrosssmb} fails to exist in some cases; see e.g.~\cite[\S{A.5.2}]{vEFS93}.
In our setup, the almost sure existence of the limit \eqref{eq:hnQ-to-hQ} and the validity of~\eqref{eq:intcrosssmb} are guaranteed by the assumption~\assref{ud} imposed on $\Q$. Indeed, as discussed in \cite[\S{5}]{CDEJR22}, the assumption implies that the functions $(f_n)_{n\in \nn}$ defined by $f_n(x) = \ln \Q_n(x_1^n)$ satisfy the gapped almost subadditivity condition of the adaptation of Kingman's theorem given in~\cite[\S{2}]{Ra22}.

Still assuming only that $\Q$ satisfies~\assref{ud} and~\assref{sld} (and that $\P$ is merely shift invariant), the random variables studied in our main theorems are known to satisfy the following almost sure identities, which, as mentioned in the introduction, justify their role as entropy estimators: for $\P$-almost every $x$, we have
\begin{equation}\label{eq:Rn-to-hP}
    \lim_{n\to\infty} \frac 1n \ln R_n(x)= h_{\P}(x) 
\end{equation}
and
\begin{equation}\label{eq:really-Vn-to-hP}
    \lim_{n\to\infty} \frac 1n \ln V_n(x) = h_{\P}(x),
\end{equation}
and, for $(\P\otimes\Q)$-almost every $(x,y)$, we have 
\begin{equation}
\label{eq:Wn-to-hQ}
    \lim_{n\to\infty} \frac 1n \ln W_n(x,y) = h_{\Q}(x).
\end{equation}

The convergences expressed in~\eqref{eq:Rn-to-hP} and \eqref{eq:really-Vn-to-hP} were first proved under the assumption that $\P$ is ergodic, in which case $h_{\P}(x) = h(\P)$ for $\P$-almost every~$x$; see~\cite{WZ89,OW93,Ko98}. 
However, combining Remark~1 of~\cite[\S{2}]{Ko98} with a generalization of Kac's lemma, as found e.g.\ in~\cite[\S{1.2.2}]{OlVi}, shows that shift invariance suffices for~\eqref{eq:Rn-to-hP}. 
Since $R_n(x) \leq V_n(x) + n-1$ and $V_n(x) = R_n(\shift^{m_n(x)-1}x) - (n-m_n(x))$ with $m_n(x) := \max\{1\leq i\leq n: x_i^{i+n-1}=x_1^n\}$, it is straightforward to show that the bounds used in~\cite[\S{2}]{Ko98} to prove~\eqref{eq:Rn-to-hP} via the Borel--Cantelli lemma also imply~\eqref{eq:really-Vn-to-hP}.
The convergence expressed in~\eqref{eq:Wn-to-hQ} was first proved for Markov measures, and then under various strong mixing assumptions; see~\cite{WZ89,Sh93,MS95,Ko98}. However, it is known that~\eqref{eq:Wn-to-hQ} may fail for some mixing measures, even when~$\P=\Q$;~see~\cite{Sh93}. In the present setup,~the \assref{sld} assumption satisfied by~$\Q$ and~\eqref{eq:hnQ-to-hQ} have been shown to imply~\eqref{eq:Wn-to-hQ} in~\cite[\S{2}]{CDEJR22}.

\subsection{Zeroes of the rate functions}
\label{ssec:zero-set-J}

In this subsection, we discuss how the LDPs of Theorems~\ref{thm:CJPS} and \ref{thm:mainthmA}--\ref{thm:mainthmC} complement the above almost sure results.
We assume throughout that the pair~$(\P, \Q)$ is admissible, but not necessarily that $\P$ satisfies~\assref{pa}. 
The central role is held by the set
\begin{equation}\label{eq:defJ}
    J := \{s : \ilpq(s)=0\} \subseteq [0,\infty),
\end{equation}
which, by convexity of $\ilpq$ (see Theorem~\ref{thm:CJPS}.i), is an interval and coincides with the subdifferential of $\qlpq$ at 0. Since $\ilpq$ is lower semicontinuous, and since $\inf_{s\in \rr} \ilpq(s) = 0$, only the three following cases are possible: $J=[a,b]$ for some $0\leq a\leq b<\infty$, $J=[a,\infty)$ for some $a\geq 0$, and $J = \emptyset$, which happens when $\ilpq(s)>0$ for all $s\in \rr$ and $\lim_{s\to\infty}\ilpq(s)=0$. When $\Q = \P$, we necessarily have $J=[a,b]$ since $\ilpp$ is a good rate function (recall Theorem~\ref{thm:CJPS}.iii). 

\begin{lemma}\label{lem:IvIrvanishJ}
    Let $(\P, \Q)$ be admissible. Then, the rate function~$\iw$ of \eqref{eq:defiwexplicit} vanishes precisely on $J$. Moreover, if $\Q=\P$, the same is true of the rate functions~$\iv$ of~\eqref{eq:defivexplicit} and~$\ir$ of~\eqref{eq:defirexplicit}. 
\end{lemma}

\begin{proof}
    The statement for $\iw$ is an immediate consequence of~\eqref{eq:defiwexplicit}. Appealing to Remark~\ref{eq:iveqiw}, the statement for~$\iv$ is immediate. In order to extend the result to $\ir$, it remains to note that $\ir(0) = 0$ if and only if $\gamma_+ = 0$, which, by \eqref{eq:Isinterval}, happens if and only if $\ilpp(0) = 0$, i.e.\ if and only if $0\in J$.
\end{proof}

We briefly comment on the simple case where~$J$ is a singleton, which occurs in particular in Examples~\ref{sec:exiid}--\ref{sec:equilbowen}, and which corresponds to the property of {\em exponential rates for entropy} in the terminology of \cite[Chap. III]{SH96}, \cite{JB13} and \cite{CRS18}. In this case, we obtain from Lemma~\ref{lem:IvIrvanishJ} that the rate functions of Theorems~\ref{thm:CJPS} and \ref{thm:mainthmA}--\ref{thm:mainthmC} all vanish at a single point, say $a$. In this situation, a standard argument shows that, almost surely, all limits in \eqref{eq:hnQ-to-hQ} and \eqref{eq:Rn-to-hP}--\eqref{eq:Wn-to-hQ} exist and are equal to $a$; see e.g.~\cite[\S{II.6}]{Ell} and the argument in the proof of Proposition~\ref{prop:lln-Q}.i below. In this simple case, our LDPs thus imply the almost sure convergence results reviewed in \secref{ssec:as-remix}.

When $J$ is an interval of positive length, our LDPs do not imply these almost sure convergence results anymore. Still, all displayed equations in \secref{ssec:as-remix} remain valid under our assumptions (see the references given there), and we show in the following two propositions that the limiting values are constrained by $J$.

\begin{proposition}\label{prop:lln-Q}
    Let $(\P, \Q)$ be an admissible pair. Then, the following hold:
      \begin{enumerate}[i.]
        \item If $J$ is nonempty, then $h_\Q(x) \in J$ for $\P$-almost all~$x$ and $\Sc(\P|\Q)\in J$.
        \item If $J$ is empty, then $h_\Q(x) =\infty$ for $\P$-almost all~$x$ and  $\Sc(\P|\Q) = \infty$.
        \item If $J$ is a singleton, then $J=\{\Sc(\P|\Q)\}$, $h_\Q(x) = \Sc(\P|\Q)$ for $\P$-almost all $x$, and the convergence in~\eqref{eq:hnQ-to-hQ} (resp.\eqref{eq:Wn-to-hQ}) occurs in~$L^p(\P)$ (resp.\ $L^p(\P\otimes \Q)$) for all $p\geq 1$.
      \end{enumerate}
 \end{proposition}
 
 \begin{proof}     
    We proceed with the proof, taking~\eqref{eq:intcrosssmb}, \eqref{eq:Wn-to-hQ} as well as the almost sure existence of the limit in \eqref{eq:hnQ-to-hQ} for granted.
    \begin{enumerate}[i.]
         \item Suppose that the interval $J$ is nonempty. Then, by convexity, $$\inf_{s: \operatorname{dist}(s, J)\geq \delta}\ilpq(s) > 0$$ for every $\delta>0$. Therefore, by the large-deviation upper bound, the probability $\P\{x : \operatorname{dist}(-\frac 1n \ln\Q_n(x_1^n), J)\geq \delta\}$ decreases exponentially fast and a standard application of the Borel--Cantelli lemma, followed by taking the limit $\delta \to 0$, implies that $\inf J \leq h_\P(x) \leq \sup J$ for $\P$-almost every~$x$. By~\eqref{eq:intcrosssmb}, we obtain $\Sc(\P|\Q)\in J$.
         
         \item
         If $J=\emptyset$, a similar argument applies to $\P\{x : -\frac 1n \ln\Q_n(x_1^n) \leq K\}$ with $K$ arbitrarily large, forcing $h_\Q(x)$ to properly diverge to~$\infty$ for $\P$-almost every~$x$. By~\eqref{eq:intcrosssmb}, we conclude that also $\Sc(\P|\Q) = \infty$.
         
         \item If~$J$ is a singleton, by Part~i we conclude that $J =\{\Sc(\P|\Q)\}$ and that $h_\Q(x) = \Sc(\P|\Q)$ for $\P$-almost all $x$. We now claim that there exists $\alpha > 0$ such that $\qlpq(\alpha)<\infty$. Indeed, if this were not the case, we would have $\ilpq(s) = \qlpq^*(s)=\sup_{\alpha \leq 0}(\alpha s -\qlpq(\alpha))$ for all $s\in \rr$, so $\ilpq$ would be nonincreasing, which contradicts the assumption that $J$ is a singleton. 
         
         For such $\alpha$, note that the function $\phi: s \mapsto \exp(\alpha s^{1/p})$ grows superlinearly and that
         \begin{equation}\label{eq:preVP}
             \limsup_{n\to\infty}  
             \int \phi\left(|\tfrac 1n \ln\Q_n|^p\right) \dd\P
                 \leq  
                 \exp \left( \qlpq(\alpha)\right)
         \end{equation} 
         by Jensen's inequality applied to $\sqrt[n]{\,\cdot\,}$. Hence, de la Vall\'ee Poussin's criterion guarantees the sufficient uniform integrability requirement for the convergence in~\eqref{eq:hnQ-to-hQ} to hold in~$L^p(\P)$. The exact same argument applies to~\eqref{eq:Wn-to-hQ}: we just need to replace \eqref{eq:preVP} with
         \begin{equation*}
            \limsup_{n\to\infty}  
            \int \phi\left((\tfrac 1n \ln W_n)^p\right) \dd (\P\otimes \Q)
                \leq  
                \exp \left(\qw(\alpha)\right)
        \end{equation*} 
        and to recall that, by \eqref{eq:formqW} and since $\alpha>0$, we have $\qw(\alpha)=\qlpq(\alpha)<\infty$.
         \qedhere
    \end{enumerate}
\end{proof}

In the next proposition, we consider the special case where $\Q=\P$, which allows to include the limits \eqref{eq:Rn-to-hP} and \eqref{eq:really-Vn-to-hP} into the discussion. Notice also that even when discussing \eqref{eq:Rn-to-hP}, the measure~$\P$ is not assumed to satisfy~\assref{pa}.

 \begin{proposition}\label{prop:lln-Qv2} 
    Assume $\P$ satisfies~\assref{ud} and~\assref{sld}, and let  $J$ be defined by \eqref{eq:defJ} with $\Q=\P$. Then, the following hold:
    \begin{enumerate}[i.]
        \item  $J\subseteq [-\gamma_+,-\gamma_-]$. 
        \item  $h_\P(x) \in J$ for $\P$-almost every~$x$, and $h(\P) \in J$.
        \item If $J$ is a singleton, then $J=\{h(\P)\}$ and $h_\P(x) = h(\P)$ for $\P$-almost every $x$. Moreover, for all $p\geq 1$, the convergence in~\eqref{eq:hnP-to-hP}, \eqref{eq:Rn-to-hP} and \eqref{eq:really-Vn-to-hP} holds in $L^p(\P)$, and that of~\eqref{eq:Wn-to-hQ} holds in $L^p(\P\otimes \P)$.
     \end{enumerate}
\end{proposition}

\begin{proof}
    Part~i is immediate by~\eqref{eq:Isinterval} and Part~ii is a special case of Proposition~\ref{prop:lln-Q}.i. 
    All statements in Part~iii except the one about \eqref{eq:Rn-to-hP} are either special cases of Proposition~\ref{prop:lln-Q}.iii or proved in exactly the same way (possibly using the simplification offered by \eqref{eq:qppbdd}).
    
    The statement about \eqref{eq:Rn-to-hP} requires additional considerations to lift~\assref{pa}. As discussed at the end of \secref{sec:atorigin}, the large-deviation upper bound of Theorem~\ref{thm:mainthmC} still holds without assuming~\assref{pa}, and we have
    $\oqr  \leq \ir^*$, with $\oqr$ defined by taking the limit superior in \eqref{eq:defqr}. Then, \eqref{eq:preVP} is simply replaced by
    \begin{equation*}
        \limsup_{n\to\infty}  
        \int \phi\left((\tfrac 1n \ln R_n)^p\right) \dd\P
            \leq  
            \exp \left( \oqr(\alpha)\right).
    \end{equation*} 
    Since $\oqr(\alpha) \leq \ir^*(\alpha) = \qlpp(\alpha)$ for all $\alpha \geq 0$ by \eqref{eq:formqR}, the same argument applies once more.
\end{proof}

\begin{remark}
    The $\P$-essential image of $h_\P(x)$ can be a strict subset of $J$. 
\end{remark}

\appendix

\section{Supports and subshifts}\label{app:subshiftsupp}

In this short appendix, we collect some basic definitions and technical results from the theory of dynamical systems, which are used at several places in the paper. 

We recall that a nonempty set $\Omega' \subseteq \Omega$ is called a {\em subshift} of $\Omega$ if it is closed (with respect to the product topology) and if it satisfies $\shift(\Omega ') \subseteq \Omega'$. The {\em language} of $\Omega'$ is defined as $\cL := \{x_n^m: 1\leq n \leq m,  x\in \Omega'\}$. The subshift $\Omega'$ is called transitive if, whenever $u\in\cL$ and $v\in\cL$, there exists $\xi\in\cL$ such that $u\xi v\in\cL$.

A nonempty set $\Omega'\subseteq \Omega$ is a subshift if and only if there exists a collection ${\cal W}\subset \Omega_{\textnormal{fin}}$ of {\em forbidden} words such that $\Omega'$ consists of those $x\in \Omega$ such that no word $x_n^m$, $1\leq n \leq m$, belongs to $\cal W$. For a given subshift $\Omega'$, the set $\cal W$ is not uniquely determined, but if it can be chosen as a finite set, then $\Omega'$ is called a {\em subshift of finite type}. In the further specialized case where $\cal W$ can be chosen as a subset of $\cA^2$, then $\Omega'$ is called a {\em Markov subshift}.

We now fix $\P\in \cP_{\textnormal{inv}}(\Omega)$ and consider $\Omega':=\supp \P =  \{x\in \Omega: \P_n(x)>0\text{ for all }n\in\nn\}$. It is clear that $\Omega'$ is a subshift whose language is given by $\cL=\{u\in \Omega_{\textnormal{fin}}: \P([u])>0\}$, and that actually $\shift(\Omega') = \Omega'$.

The topological entropy of $\Omega'$ is defined by
\begin{equation}\label{eq:htopdef}
    \htop(\Omega') := \lim_{n\to\infty}\frac 1n \ln |\{x_1^n: x\in \Omega'\}|  =  \lim_{n\to\infty}\frac 1n \ln |\supp \P_n| \leq \ln |\cA|,
\end{equation}
where we recall that $\supp \P_n := \{u\in \cA^n: \P_n(u)>0\}$. The limits in~\eqref{eq:htopdef} exist by a standard subadditivity argument (which does not require any decoupling assumption). 

\begin{lemma}
    \label{lem:proofchainhtop}
    Let $\P\in \cP_{\textnormal{inv}}(\Omega)$. Then, the bounds~\eqref{eq:ineqgammapm} hold.
\end{lemma}

\begin{proof}
     Obviously, $0\leq -\gamma_+$. Let now $H(\P_n) := -\sum_{u\in \cA^n}\P_n(u)\ln \P_n(u)$. By the definition of $\gamma_+$ we find $\ln\P_n(u)\leq  n\gamma_+ +o(n)$ for every $u\in \cA^n$ and so $-n\gamma_+ \leq H(\P_n)+o(n)$. On the other hand, $H(\P_n)\leq \ln |\supp\P_n|$ (see e.g.\ Corollary~4.2.1 in~\cite{Wal}). Finally, clearly $\min_{u\in \supp \P_n}\P_n(u) \leq 1/|\supp \P_n|$, so $\ln|\supp \P_n|\leq -\min_{u\in \supp \P_n}\ln\P_n(u)$. Dividing the above inequalities by~$n$ and taking $n\to\infty$ completes the proof of~\eqref{eq:ineqgammapm}. Remark that the inequality $ h(\P) \leq  \htop(\supp\P)$ also follows from the much more general result~\eqref{eq:varprincp} below.
\end{proof}

\subsection{Specification properties}

The following definitions will be especially helpful when combined with Lemmas~\ref{lem:specud} and~\ref{lem:specsld} below. As in the statement of our decoupling conditions, we shall assume that $\tau_n = o(n)$, and there will be no loss of generality assuming that the various assumptions hold with the same sequence $(\tau_n)_{n\in \nn}$. Again, the case where the sequence is bounded (written $\tau_n = O(1)$) will sometimes be singled out. 
\begin{definition}
\label{def:specif}
    For a subshift $\Omega'\subset \Omega$ with language $\cL$, we define the following properties:
    \begin{itemize}
        \item the subshift $\Omega'$ satisfies the \emph{flexible specification property} if for every $n\in \nn, u\in \cL\cap \cA^n$, and every $v\in \cL$, there exists $0\leq \ell \leq \tau_n$ and $\xi \in \cA^\ell$ such that $u\xi v\in \cL$;
        \item the subshift $\Omega'$ satisfies the \emph{abundant periodic orbit property} if for each $n\in \nn$ and each $u\in \cL\cap  \cA^{n}$, the cylinder $[u]$ contains at least one periodic point of period no larger than $n + \tau_n$, i.e.\ if there exists a word $\xi$ with $0\leq |\xi|\leq \tau_n$ such that $u\xi u\xi u\xi \dots \in \Omega'$.
    \end{itemize}
\end{definition}

\begin{remark}\label{rem:specification-sld}
        Obviously, if $\P$ satisfies~\assref{sld}, then~$\supp \P$ satisfies the flexible specification property (with the same sequence $(\tau_n)_{n\in \nn}$). In the same way, if a subshift $\Omega'$ satisfies Bowen's specification property \cite{Bo74}, then the two properties in \eqref{def:specif} are satisfied with $\tau_n = O(1)$. The same is true when $\Omega'$ is a transitive subshift of finite type: the flexible specification property (with $\tau_n = O(1)$) is quite immediate, and the abundant periodic orbit property is then obtained using Lemma~\ref{lem:bded-spec-implies-per-spec} below. The interested reader can gather other sufficient conditions by combining ideas from \cite{Ju11,KLO-2016,CLTH21}.
\end{remark}

    \begin{lemma}
    \label{lem:bded-spec-implies-per-spec}
        If the flexible specification property holds with $\tau_n = O(1)$, then the abundant periodic orbit property holds with some (possibly larger) $\tau_n' = O(1)$.
    \end{lemma}

    \begin{proof}
        Assume that the flexible specification property holds with $\tau:= \sup_{n\in \nn}\tau_n < \infty$. Following~\cite{Be88} and~\cite[\S{3}]{Ju11}, this implies the existence of a \emph{synchronizing word}~$w$, i.e.\ a word $w\in \cL$ such that whenever $uw\in \cL$ and $wv\in \cL$, then $uwv\in \cL$. 
        Using the assumption twice, for every~$u\in\cL$, we can find $\xi_1$ and $\xi_2$, each of length at most $\tau$, such that $w\xi_1 u\xi_2 w \in \cL$. Note that $|\xi_1 w \xi_2| \leq 2 \tau + |w|$. But by the synchronizing property of~$w$, we must have $u\xi_2 w\xi_1u\xi_2 w\xi_1u\xi_2 w\xi_1\dotsb \in \Omega'$, so the abundant periodic orbit property holds with $\tau_n':=2 \tau + |w| = O(1)$.
    \end{proof}

\subsection{Measures of maximal entropy}

Let $\Omega'$ be a subshift. We denote by $\cP_{\textnormal{inv}}(\Omega')$ the set of shift-invariant probability measures on $\Omega'$, and we identify the elements of $\cP_{\textnormal{inv}}(\Omega')$ with those of $\{\mu\in \cP_{\textnormal{inv}}(\Omega): \mu(\Omega')=1\}$ in the obvious way.

By the variational principle (see e.g.\ Theorem~8.6 in~\cite{Wal}) applied to the dynamical system $(\Omega', \shift)$, we have
\begin{equation}\label{eq:varprincp}
    \htop(\Omega') = \max_{\mu\in \cP_{\textnormal{inv}}(\Omega')}h(\mu),
\end{equation}
where the right-hand side is indeed a maximum, since the map $h$ is upper semicontinuous (see e.g.\ Theorem~8.2 in~\cite{Wal}). In particular $h(\mu)\leq \htop(\Omega')$ for all $\mu\in \cP_{\textnormal{inv}}(\Omega')$, and $\mu$ is called a {\em measure of maximal entropy} (MME) on $\Omega'$ if $h(\mu) =\htop(\Omega')$.

Given a continuous potential $\varphi:\Omega'\to \rr$, we denote by $\ptop(\varphi)$ the topological pressure of $\varphi$ with respect to the dynamical system $(\Omega',T)$, that is, 
$$
    \ptop(\varphi) := \lim_{n\to\infty}\frac 1n \ln \sum_{u \in \cA^n \cap \cL}\Exp{\sup_{x \in [u] \cap \Omega'} \sum_{i=0}^{n-1}  \varphi(\shift^{i}x)},
$$
which reduces to \eqref{eq:defptop} when $\Omega'=\Omega$. With this in mind, noting that $\htop(\Omega')$ is nothing but $\ptop(\varphi_0)$ for the potential $\varphi_0:\Omega'\to \rr$ that vanishes identically, we conclude that $\mu\in \cP_{\textnormal{inv}}(\Omega')$ is a MME on $\Omega'$ if and only if $\ptop(\varphi_0) = h(\mu)$, that is, if and only if $\mu$ is an equilibrium measure for~$\varphi_0$.

\begin{proposition}
\label{prop:MMEprops}
    Let $\P\in \cP_{\textnormal{inv}}(\Omega)$, and let $\Omega' := \supp \P$.
    \begin{enumerate}[i.]
        \item Assume $\P$ satisfies~\assref{sld}. Then, the measure~$\P$ is a MME on $\Omega'$ if and only if $\gamma_- = \gamma_+$.
        \item Assume $\P$ satisfies~\assref{sld} with $\tau_n=O(1)$. Then, the dynamical system $(\Omega', \shift)$ is {\em intrinsically ergodic}, i.e.\ there is exactly one MME on $\Omega'$. 
    \end{enumerate}
\end{proposition}

\begin{proof}
    Recall that~$\gamma_{\pm}$ are defined in~\eqref{eq:defgammaplus}.
    \begin{enumerate}[i.]
        \item Assume $\P$ is a MME on $\Omega'$. As discussed above, $\P$ is then an equilibrium measure for the identically vanishing potential $\varphi_0$ defined on $\Omega'$.
        By~\assref{sld}, $\Omega'$ enjoys the flexible specification property (Remark~\ref{rem:specification-sld}) and thus we can apply Theorem~2.1 in~\cite{PS20}, which guarantees that $\P$ is weak Gibbs with respect to $\varphi_0$, i.e.\ that for all $x\in \Omega'$,
        $$
        \P_n(x_1^n) = \Exp{\sum_{i=0}^{n-1} \varphi_0(\shift^{i}x)-n\ptop(\varphi_0)+o(n)}  = \Exp{-n\htop(\Omega')+o(n)}.        
        $$
        From this we obtain $\gamma_+ = \gamma_-=\htop(\Omega')$.
        Conversely, if $\gamma_- = \gamma_+$, then~\eqref{eq:ineqgammapm} implies that $h(\P) = \htop(\Omega')$, so $\P$ is a MME.

        \item Under the stated assumption, $\Omega'$ satisfies the flexible specification property with $\tau_n=O(1)$ (Remark~\ref{rem:specification-sld}) and the conclusion follows from e.g.\ Theorem~3.2 in~\cite{CLTH21}. 
        \qedhere
    \end{enumerate}
\end{proof}

\section{Technical results about decoupling assumptions}
\label{app:decoupling}

 This appendix complements the discussion of our assumptions in~\secref{ssec:setup} and uses the notations therein.

\subsection{Alternative definitions}

The following lemma is used in Remark~\ref{rem:changetaun}.
\begin{lemma}\label{lem:increasecntaun}
    If any of the assumptions~\assref{ud}, \assref{sld} or \assref{jsld} holds with some sequences $(C_n)_{n\in \nn}$ and $(\tau_n)_{n\in \nn}$, then it holds with any choice of sequences $(C'_n)_{n\in \nn}$ and $(\tau'_n)_{n\in \nn}$ satisfying $C_n'\geq C_n$ and $\tau_n'\geq \tau_n$ for all $n\in \nn$.
\end{lemma}

\begin{proof}
    The only nontrivial part of the statement is that we may replace $\tau_n$ with $\tau_n'$ in~\assref{ud}, which we now prove. Let $n,m\in\nn$, $u\in \cA^n$, $v\in \cA^m$ and $\xi\in \Omega^{\tau'_n}$. \assref{ud} then implies, with $\xi' = \xi_{\tau_n+1}^{\tau_n'}$ (possibly the empty word if $\tau_n'=\tau_n$) that $\P_{n+\tau_n'+m}(u\xi v) \leq C_n\P_{n}(u)\P_{m+\tau_n'-\tau_n}(\xi'v)\leq C_n\P_{n}(u)\P_{m}(v)$.
\end{proof}

The next two lemmas justify \eqref{eq:udeventseq} and \eqref{eq:sldeventseq} respectively.

\begin{lemma}\label{lem:udevents} 
    If $\P\in \cP_{\textnormal{inv}}(\Omega)$ satisfies~\assref{ud}, then
    \begin{equation*}
        \P\left(A \cap \shift^{-n - \tau_n}B\right) \leq \widetilde C_n \P(A)\P(B)
    \end{equation*}
    for  every $n\in \nn$, $A\in \cF_n$ and $B\in \cF$, with $\widetilde C_n := C_n|\cA|^{\tau_n}= \Exp{o(n)}$.
\end{lemma}

\begin{proof}
    First, since the semi-algebra of cylinder sets generates $\cF$, and since $B\mapsto \P(A\cap \shift^{-n-\tau_n}B)$ defines a finite measure, it suffices to consider the case where $B\in \cF_m$ for some $m\in \nn$ (see e.g.\ Lemma~1.9.4 in~\cite{Bog}). But then, by~\assref{ud},
    \begin{align*}
        \P\left(A \cap \shift^{-n - \tau_n}B\right) &= \sum_{\substack{u\in \cA^n: [u]\subseteq A\\v\in \cA^m: [v]\subseteq B\\ \xi \in \cA^{\tau_n}}}\P_{n+\tau_n+m}(u\xi v)  \\
        &\leq C_n  \sum_{\substack{u\in \cA^n: [u]\subseteq A\\v\in \cA^m: [v]\subseteq B\\ \xi \in \cA^{\tau_n}}}\P_n(u)\P_m(v) \\
        & =  C_n  \sum_{\xi \in \cA^{\tau_n}}\P(A)\P(B),
    \end{align*}
    from which the result follows, since the sum over $\xi$ contains $|\cA|^{\tau_n}$ terms.
\end{proof}

\begin{lemma}
\label{lem:sldevents} 
    If $\P\in \cP_{\textnormal{inv}}(\Omega)$ satisfies~\assref{sld}, then 
    \begin{equation*}
        \max_{0\leq \ell \leq \tau_n}\P\left(A \cap \shift^{-n -\ell}B\right) \geq \widehat C_n^{-1} \P(A)\P(B)
    \end{equation*}
    for every $n\in \nn$, $A\in \cF_n$ and $B\in \cF$, with $\widehat C_n :=  (\tau_n+1)C_n = \Exp{o(n)}$.
\end{lemma}

\begin{proof}
    Consider the finite measure $\mu$ defined by $\mu(B)=\sum_{\ell=0}^{\tau_n} \P(A\cap \shift^{-n-\ell}B)$. If $B\in \cF_m$ for some $m\in \nn$, we obtain by~\assref{sld},
    \begin{align*}
        \mu(B)&= \sum_{\substack{u\in \cA^n: [u]\subseteq A\\v\in \cA^m: [v]\subseteq B\\ \xi \in \bigcup_{\ell=0}^n \cA^\ell}}\P_{n+|\xi|+m}(u\xi v) \\
        &\geq C_n^{-1} \sum_{\substack{u\in \cA^n: [u]\subseteq A\\v\in \cA^m: [v]\subseteq B}}\P_n(u)\P_m(v) \\
        &= C_n^{-1}\P(A)\P(B).
    \end{align*}
    As in Lemma~\ref{lem:udevents}, this extends to any $B\in \cF$. 
    Since $$\mu(B) \leq (\tau_n+1)\max_{0\leq \ell\leq \tau_n}\P\left(A \cap \shift^{-n - \ell}B\right)$$ by definition, the proof is complete.
\end{proof}

\begin{remark}\label{rem:eventsequivalent}
    It is obvious that the condition stated in Lemma~\ref{lem:udevents} is actually equivalent to~\assref{ud}, by considering $A=[u]$ and $B=[v]$. In the same way, the condition stated in Lemma~\ref{lem:udevents} is equivalent to~\assref{sld}, although going back to the original statement of~\assref{sld} costs an additional (harmless) factor of $|\cA|^{\tau_n}$, since the inequality $\P\left([u] \cap \shift^{-n -\ell}[v]\right) \geq \widehat C_n^{-1} \P([u])\P([v])$ only implies that there exists some $\xi\in \cA^{\ell}$ such that $\P\left([u] \cap \shift^{-n}[\xi]\cap  \shift^{-n -\ell}[v]\right) \geq \widehat C_n^{-1}|\cA|^{-\ell} \P([u])\P([v])$.
\end{remark}

\begin{remark}\label{rem:boundednesscompared}
    In general, boundedness of the sequence $(C_n)_{n\in\nn}$ does not imply boundedness of the sequences $(\widetilde C_n)_{n\in\nn}$ and $(\widehat C_n)_{n\in\nn}$. However, this implication holds when $\tau_n = O(1)$.
\end{remark}

\subsection{Sufficient conditions}
\label{ssec:subshifts}

We start with a sufficient condition for~\assref{pa} to hold. In a nutshell, the sufficient condition is that for every $p\in \nn$, $v\in \cA^p$, there is some $\xi\in \Omega_{\textnormal{fin}}$ such that the probability of $(v\xi)^{n-1}v$ satisfies a lower bound in terms of $\P_{p}(v)^n$.
We stress that~\assref{sld} does not imply such a lower bound in general. Indeed, applying \assref{sld} a total of $n-1$ times recursively  yields a word $v\xi^{(n-1)}v \cdots v \xi^{(2)} v\xi^{(1)}v$ whose probability is bounded below by $C_p^{n-1}\P_p(v)^n$, but there is no reason for the $\xi^{(i)}$ to be equal in general, unless of course $\tau_n = 0$ for all $n$, in which case~\assref{sld} does imply the assumptions of Lemma~\ref{lem:pafromarray} and thus~\assref{pa}.

\begin{lemma}
\label{lem:pafromarray} 
    Let $\P\in \cP_{\textnormal{inv}}(\Omega)$ and assume that there exists a family $(c_{p,n})_{p\in\nn,n\in\nn}$ satisfying
    \begin{equation}
        \label{eq:cond-c-array}
        \lim_{p\to\infty}\frac 1p \limsup_{n\to\infty}\frac{\ln c_{p,n}}{n}=0,
    \end{equation}
    and such that for every $p\in\nn$ and $v\in \cA^p$, there exists $\xi \in \Omega_{\textnormal{fin}}$ such that 
    \begin{equation}\label{eq:hyplemarray}
       \P_{np+(n-1)|\xi|}\left((v\xi)^{n-1} v\right) \geq c_{p,n} (\P_p(v))^n
    \end{equation}
    for all $n\in\nn$. 
    Then, the measure~$\P$ satisfies~\assref{pa}.
\end{lemma}

\begin{proof}
    Let $\epsilon>0$ and $P \in \nn$ be arbitrary. By definition of~$\gamma_+$, there exists $p \geq P$ and $v\in \cA^p$ such that $\P_p(v) > \Exp{p(\gamma_+ -\frac 12 \epsilon)}$. With~$\xi \in\Omega_{\textnormal{fin}}$ as in \eqref{eq:hyplemarray}, we have 
    \begin{align*}
        \liminf_{n\to\infty} \frac 1 {n (p+|\xi|)}\ln \P_{n(p+|\xi|)}((v\xi)^{n}) 
        &\geq \liminf_{n\to\infty} \frac 1 {n p}\ln \P_{n(p+|\xi|)}((v\xi)^{n}) \\
        &\geq \liminf_{n\to\infty} \frac 1 {n p}\ln \P_{(n+1)p+n|\xi|}((v\xi)^{n}v) \\
        &\geq \liminf_{n\to\infty} \frac 1 {np}((n+1) \ln \P_p(v) + \ln c_{p, n+1}) \\
        &> \gamma_+ - \frac \epsilon 2 + \frac 1{p} \liminf_{n\to\infty}\frac{\ln c_{p,n}}{n}.
    \end{align*}
    By \eqref{eq:cond-c-array}, we can choose $P$\,---\,and thus~$p$\,---\,large enough so that the right-hand side is bounded below by $\gamma_+- \epsilon$, so the choice $u = v\xi$ satisfies the requirement in the definition of~\assref{pa}.
\end{proof}

We now give a sufficient condition for \assref{ud} which is used in the setup of Sections~\ref{sec:equilbowen} and \ref{sec:new-g-sm}. 

\begin{lemma}\label{lem:specud}
	Assume that there exists an $\Exp{o(n)}$-sequence $(C_n)_{n\in\nn}$ such that
	\begin{equation}\label{eq:pnpmxud}
			\P_{n+m}(x_1^{n+m})\leq C_n \P_n(x_1^n)\P_m(x_{n+1}^{n+m})
	\end{equation}
	for all $x\in \supp \P$ and all $n, m\in\nn$. Then, the measure~$\P$ satisfies~\assref{ud} with $\tau_n = 0$.
\end{lemma}

\begin{proof}
    Let $u,v$ be as in the definition of~\assref{ud}. First, if $\P_{n+m}(uv) > 0$, there exists $x\in [uv]\cap \supp \P$, and in this case~\eqref{eq:pnpmxud} implies that $\P_{n+m}(uv)\leq C_n \P_n(u)\P_m(v)$, as required. The same inequality trivially holds if $\P_{n+m}(uv) = 0$.
\end{proof}

The next lemma gives sufficient conditions for \assref{sld}, \assref{jsld} and \assref{pa} to hold in terms of the properties introduced in Definition~\ref{def:specif}. This is again particularly useful in the setup of Sections~\ref{sec:equilbowen} and \ref{sec:new-g-sm}. 

\begin{lemma}\label{lem:specsld}
    Assume that there exists an $\Exp{o(n)}$-sequence $(C_n)_{n\in\nn}$ such that
	\begin{equation}\label{eq:pnpmx}
			\P_{n+m}(x_1^{n+m})\geq C_n^{-1} \P_n(x_1^n)\P_m(x_{n+1}^{n+m})
	\end{equation}
	for all $x\in \supp \P$ and all $n, m\in\nn$. Then, the following hold:
	\begin{enumerate}[i.]
		\item If $\supp \P$ satisfies the flexible specification property, then $\P$ satisfies~\assref{sld}.
		\item If $\supp \P$ satisfies the abundant periodic orbit property, then $\P$ satisfies~\assref{pa}.
		\item If $\supp \P$ satisfies the flexible specification property, if $\Q\in \cP_{\textnormal{inv}}$ is another measure satisfying~\eqref{eq:pnpmx}, and if $\supp \P = \supp \Q$, then the pair $(\P, \Q)$ satisfies~\assref{jsld}.
	\end{enumerate}
\end{lemma}

\begin{proof}
	First note the following consequence of repeated applications of~\eqref{eq:pnpmx}: for all $u,v\in \cL$ such that $uv\in \cL$, we have
	\begin{equation}\label{eq:Pbddbelow}
			\P([uv])\geq \delta^{|u|}\P([v]),
	\end{equation}
	with $\delta = C_1^{-1} \inf_{a\in \supp \P_1(a)}\P_1(a)\in (0,1]$.
	\begin{enumerate}[i.]
        \item Let $u,v$ be as in the definition of~\assref{sld}. We assume that $u,v\in \cL$, for otherwise there is nothing to prove. By the flexible specification property, there exists $\xi$ with $|\xi| \leq \tau_n$ such that $u\xi v\in \cL$.  Then, applying~\eqref{eq:pnpmx} to any $x\in [u\xi v]$ implies that $\P_{n+|\xi|+m}(u\xi v)\geq C_n^{-1}\P_n(u)\P_{|\xi|+m}(\xi v) \geq C_n^{-1}\delta^{\tau_n}\P_n(u)\P_m(v)$, where we have used \eqref{eq:Pbddbelow}. We have thus proved \assref{sld} with $C_n' := C_n \delta^{-\tau_n} = \Exp{o(n)}$. 
        
        \item Let $p\in \nn$ and $v\in \cA^p$. By the abundant periodic orbit property, there exists $\xi$ with $|\xi|\leq \tau_p$ such that $(v\xi)^{n-1}v\in \cL$ for all $n\in \nn$. By iterating the argument in the proof of Part i, we obtain~\eqref{eq:hyplemarray} with $c_{n,p} := (C_p')^{1-n}$, so Lemma~\ref{lem:pafromarray} establishes~\assref{pa}.
        
        \item Under the stated assumptions, $\P_n$ and $\Q_n$ are equivalent, and since the choice of $\xi$ in the proof of Part i only depends on the properties of the support, the same $\xi$ can be chosen to prove~\assref{sld} for both $\P$ and $\Q$, so~\assref{jsld} holds. 
        \qedhere
    \end{enumerate}
\end{proof}

\begin{remark}In the same way that we derived~\eqref{eq:Pbddbelow}, we conclude that if~\eqref{eq:pnpmx} holds, then $\P([u])\geq \delta^{|u|}$ for all $u\in \cL$, so the map $x\mapsto -\frac 1n\ln \P_n(x_1^n)$ is bounded on $\supp \P$. Thus, in this situation, the limit defining $\qlpp$ is finite whenever it exists. In the same way, if $\Q,\P\in \cP_{\textnormal{inv}}(\Omega)$ have common support and if $\Q$ satisfies~\eqref{eq:pnpmx}, then $\qlpq$ is finite whenever the limit exists. 
\end{remark}

\begin{remark}\label{rem:weakGibbsdeco}If both \eqref{eq:pnpmxud} and \eqref{eq:pnpmx} hold for all $x\in \supp \P$, then there exists a continuous potential $\varphi$ on $\supp \P$ such that the weak Gibbs condition \eqref{eq:weakGibbsprop} holds for all $x\in \supp\P$, by Theorem 1.2 of \cite{cuneo_asympt_2020} applied to $f_n(x) = \ln \P_n(x_1^n)$. Unfortunately, the weak Gibbs condition does not seem to imply \eqref{eq:pnpmxud} and \eqref{eq:pnpmx}, due to the lack of uniformity in $m$, so Lemmas~\ref{lem:specud} and \ref{lem:specsld} do not apply to weak Gibbs measures in general.
\end{remark}

\subsection{Decoupling properties of \textit{g}-measures}
\label{app:g-measures}

The following lemma is used in \secref{sec:gmeasures} in order to show that $g$-measures satisfy our decoupling assumptions.

\begin{lemma}
\label{lem:app-g-1}
    Let $\Omega'$ be a subshift of~$\Omega$. Suppose that $\P\in\cP_{\textnormal{inv}}(\Omega)$ is such that $\supp \P = \Omega'$ and that the sequence of functions on~$\Omega'$ defined by 
    \[
        g_n(x) := \frac{\P_n(x_1^n)}{\P_{n-1}(x_2^n)}
    \]
    converges uniformly to some continuous function~$g:\Omega'\to (0,1]$. 
    Then, there exists an $\Exp{o(n)}$-sequence $(K_n)_{n\in\nn}$ such that
    \begin{equation}\label{eq:WGgmeas}
            K_n^{-1} {\prod_{i=0}^{n-1} g (\shift^{i}x)} \leq \P_n(x_1^n) \leq K_n {\prod_{i=0}^{n-1} g (\shift^{i}x)}
    \end{equation}
    and
    \begin{equation}\label{eq:decouplgmeas}
        K_n^{-2} \P_n(x_1^n)\P_m(x_{n+1}^{n+m})
        \leq \P_{n+m}(x_1^{n+m}) 
        \leq K_n^2 \P_n(x_1^n)\P_m(x_{n+1}^{n+m})
    \end{equation}
    for all~$x \in \Omega'$. In particular, \eqref{eq:pnpmxud} and \eqref{eq:pnpmx} hold for all $x\in \Omega'$ with $C_n = K_n^2$, and $\P$ is weak Gibbs on $\Omega'$ for the potential $\potone = \ln g$, i.e.~\eqref{eq:weakGibbsprop} holds for all $x\in \Omega'$.  
\end{lemma}

\begin{proof}
    Throughout the proof, we use the convention that $\P_0(x_1^0) = 1$. Let $\Delta_n(x) := \ln g_n(x) - \ln g(x)$ for $n\in \nn$ and $x\in \Omega'$. 
    By compactness, positivity and uniform convergence, we find that $\Delta_n^* := \sup_{n'\geq n}\sup_{x\in\Omega'}|\Delta_{n'}(x)|$ defines a bounded sequence that converges to~$0$ as $n\to\infty$. As a consequence, $\tfrac{1}{n} \sum_{k=1}^{n} \Delta_n^* \to 0$, and $K_n:=\exp \sum_{k=1}^{n} \Delta_k^*$ defines an $\Exp{o(n)}$-sequence.

    For all $x\in \Omega'$, a telescoping argument yields
    \begin{equation}\label{eq:firsttelescopeg}
        \P_n(x_1^n) = \prod_{k=1}^n g_{n-k+1}(\shift^{k-1}x) =  \Exp{\sum_{k=1}^n\Delta_k(\shift^{n-k}x)}\prod_{k=1}^n g(\shift^{k-1}x),
    \end{equation}
    which establishes \eqref{eq:WGgmeas}.
    In order to prove \eqref{eq:decouplgmeas}, we remark that by another telescoping argument,
    \begin{align*}
        \P_{n+m}(x_1^{n+m}) &= \P_m(x_{n+1}^{n+m}) \prod_{k=1}^n g_{n+m-k+1}(\shift^{k-1}x)\\
        & = \P_m(x_{n+1}^{n+m})   \Exp{\sum_{k=1}^n\Delta_{k+m}(\shift^{n-k}x)} \prod_{k=1}^n g(\shift^{k-1}x).
    \end{align*}
    Combining the above with \eqref{eq:firsttelescopeg} and noting that $\Delta_{k+m}^*\leq \Delta_k^*$, we obtain \eqref{eq:decouplgmeas}.
\end{proof}

\subsection{Relation to mixing}
    
We recall that a measure $\P \in \cP_{\textnormal{inv}}$ is said to be \emph{$\psi$-mixing} if its $\psi$-mixing coefficients~$(\psi_{\P}(\tau))_{\tau\in\nn\cup\{0\}}$, defined by
$$
    \psi_{\P}(\tau) := \sup_{n,m\in\nn} \max 
        \left\{
            \left|\frac{\P([u] \cap \shift^{-n-\tau}[v])}{\P_n(u)\P_m(v)} - 1\right|
            :
            u \in \supp\P_n, v\in\supp\P_m
        \right\},
$$
satisfy $\lim_{\tau\to\infty} \psi_{\P}(\tau) = 0$.

\begin{lemma}\label{lem:mixingrel}
    If $\P$ is $\psi$-mixing, then
    $\P$ satisfies~\assref{sld} and~\assref{ud} with $\tau_n=O(1)$ and $C_n=O(1)$. 
\end{lemma}

\begin{proof}
    By assumption, there exists $\tau \in \nn\cup\{0\}$ such that $\psi_{\P}(\tau) < 1$. But then~\assref{ud} holds with $\tau_n^+:=\tau$ and $C_n^+ := 1 + \psi_{\P}(\tau)$, and~\assref{sld} holds with $\tau_n^- := \tau$ and $C_n^- := \frac{|\cA|^{\tau}}{1-\psi_{\P}(\tau)}$. Both~\assref{ud} and~\assref{sld} thus hold with $C_n := \max\{C_n^-,C_n^+\}$ and $\tau_n := \tau$.
\end{proof}

\begin{remark}
    We do not claim that $\psi$-mixing alone implies~\assref{pa} in general. It does, however, when $\psi_{\P}(0) < 1$. In this case, Lemma~\ref{lem:mixingrel} shows that~\assref{sld} holds with $\tau_n = 0$ for all $n$, which implies~\assref{pa}; recall Remark~\ref{rem:tau0pa}.\footnote{The reader familiar with the terminology in e.g.~\cite[\S{2.1}]{Bra05} will note that $\psi'_{\P}(0) > 0$ suffices.} In the same way, if~$\P$ and~$\Q$ are two $\psi$-mixing measures with common support, we do not claim that~\assref{jsld} holds in general, except once again if $\psi_{\P}(0),\psi_{\Q}(0) < 1$; recall Remark~\ref{rem:tau0sld}.
\end{remark}

\section{From weak LDPs to full LDPs}
\label{app:weak-to-full}

\subsection{General results for nonnegative random variables}

We gather in the following proposition a collection of results allowing to establish Legendre--Fenchel duality and promote the weak LDP to a full one in the case of nonnegative random variables. They rely on standard techniques, and we provide the details not only for completeness, but also because our setup is slightly unusual: we do not assume the full LDP nor goodness of the rate function, but we rely on the one-dimensional and convex nature of the problem, in the same spirit as the proof of Proposition~5.2 of \cite{CJPS19}. We again use the notation of the beginning of Section~\ref{sec:RLfuncts}.

\begin{proposition}\label{prop:abstractZn}
    Let $(Z_n)_{n\in\nn}$ be a sequence of $[0,\infty)$-valued random variables
    defined on a probability space $(\Omega_*, \P_*)$. Assume $(Z_n)_{n\in\nn}$ satisfies the weak LDP with a rate function $I$, and set
    \begin{equation}\label{eq:defuqoq}
     \underline q(\alpha) := \liminf_{n\to\infty}\frac 1n \ln \int \Exp{\alpha n Z_n}\dd \P_* 
     \qquad \text{and} \qquad
     \overline q(\alpha) := \limsup_{n\to\infty}\frac 1n \ln \int \Exp{\alpha n Z_n}\dd \P_*
    \end{equation} 
    for $\alpha\in\rr$.
    When $\overline q(\alpha) = \underline q(\alpha)$, let $q(\alpha)$ denote their common value.
    Then, the following hold:
    \begin{enumerate}[label=\roman*.,ref=\roman*]
        \item \label{part:qnondeg}$\underline q$, $\overline q$ are nondecreasing and satisfy $\underline q(0) = \overline q(0) = 0$.
        \item \label{part:propI} $I(s)=\infty$ for all $s<0$.
        \item \label{part:qgeqistar} $\underline q\geq I^*$.
        \item \label{part:qleqistarcond} $\overline q(\alpha) \leq I^*(\alpha)$ for every $\alpha$ satisfying the tail condition
         \begin{equation}\label{eq:limK}
            \lim_{K\to\infty} \limsup_{n\to \infty} \frac 1n \ln \int  \Exp{\alpha n Z_n}1_{Z_n> K} \dd \P_* = -\infty.
        \end{equation}
        \item \label{part:qleqistarneg} For every $\alpha <0$, $q(\alpha)$ exists and $q(\alpha) = I^*(\alpha)$.
        \item \label{part:qIfullcases2} If $\overline q(\alpha) < \infty$ for some $\alpha > 0$, then the full LDP holds, $I$ is a good rate function and the family $(Z_n)_{n\in\nn}$ is exponentially tight.
        \item \label{part:qIfullcases} Assume $I$ is convex. If $q(\alpha)$ exists (in $[-\infty, \infty]$) and $q(\alpha) = I^*(\alpha)$ for every $\alpha \in \rr$, then $I=q^*$ and the full LDP holds.
    \end{enumerate}
\end{proposition}

\begin{remark}\label{rem:weirdIQ}
    Part~\ref*{part:qleqistarneg} does not extend to $\alpha = 0$ automatically, because the weak LDP alone does not guarantee that $I^*(0)=0$. Indeed, we could in principle have, for example, $I(s) =  1 +\infty \cdot 1_{s< 0}$ and $q(\alpha) = -1_{\alpha<0} + \infty\cdot 1_{\alpha > 0}$.
\end{remark}
\begin{remark}
    The convexity assumption in Part~\ref*{part:qIfullcases} is required not only to derive $I=q^*$, but also in order to prove the full LDP. 
    In the situation where Part~\ref*{part:qIfullcases} applies but Part~\ref*{part:qIfullcases2} does not, that is when $I$ is convex but  $q(\alpha) = \infty$ for all $\alpha > 0$, we shall obtain in the proof that, due to convexity, $\lim_{s\to\infty} I(s) = 0$, which makes the large-deviation upper bound for unbounded sets somewhat trivial. In particular, in this situation $I$ is not a good rate function and $(Z_n)_{n\in \nn}$ is not an exponentially tight family; recall Figure~\ref{fig:HMC1}. 
\end{remark}

\begin{proof}
    In this proof, we denote by $\mu_n$ the distribution of $Z_n$ under $\P_*$, that is $\mu_n = \P_* \circ Z_n^{-1}$, which obviously satisfies $\mu_n([0,\infty)) = 1$. Parts~\ref*{part:qnondeg} and~\ref*{part:propI} are immediate consequences of the definitions \eqref{eq:defuqoq} and the fact that $Z_n \geq 0$. 

    \begin{enumerate}[i.]
        \setcounter{enumi}{2}
        \item
        This is the standard lower bound of Varadhan's lemma; see e.g.\ Lemma~4.3.4 in~\cite{DZ}. We include the proof for completeness.
        For any $s$, $\alpha \in \rr$ and $\epsilon > 0$, we have by the weak LDP
        \begin{align*}
           \underline q(\alpha) &\geq \liminf_{n\to\infty} \frac 1 n  \ln \int_{B(s,\epsilon)} 
           \Exp{\alpha n s'}\dd \mu_n(s') \\
            &  \geq (\alpha s - |\alpha|\epsilon) + \liminf_{n\to\infty} \frac 1 n\ln
            \mu_n(B(s,\epsilon))\\
            &\geq \alpha s - |\alpha|\epsilon-I(s).
        \end{align*}
        Letting $\epsilon \to 0$, we obtain $\underline q(\alpha) \geq \alpha s - I(s)$. Since
        $\alpha$ and $s$ are arbitrary, we indeed have $\underline q \geq I^*$. 
        
        \item
        We give here the classical covering argument used in the proof of the upper bound of Varadhan's lemma; see  e.g.\ Lemma~4.3.6 in~\cite{DZ} for the same result under slightly different assumptions.
        Let $\alpha$ verify~\eqref{eq:limK} and fix $\epsilon > 0$ and
        $K>0$. 
        For every $s\in [0, K]$, there exists an open neighborhood $G_s \ni s$
        such that
        \[
            \inf_{s'\in \overline{G}_s} I(s') \geq \min\left\{(I(s)-\epsilon), \epsilon^{-1}\right\}
            \qquad \text{and} \qquad 
            \sup_{s'\in G_s} \alpha s' \leq \alpha s+\epsilon.
        \]
        These neighborhoods cover the compact set $[0,K]$ and we can extract a finite subcover $\{G_{s_1}, G_{s_2}, \dots, G_{s_r}\}$. Now, by the weak LDP, 
        \begin{align*}
            \limsup_{n\to\infty} \frac 1 n	\ln \int_{G_{s_i}} \Exp{\alpha n s'}\dd\mu_n(s') 
            &\leq \limsup_{n\to\infty} \frac 1 n \ln \left( \Exp{\alpha n s_i +
            \epsilon n}\mu_n(G_{s_i})\right)\\
            & \leq \alpha s_i +\epsilon - \min\left\{(I(s_i)-\epsilon), \epsilon^{-1}\right\} \\
            & = \max\left\{\alpha s_i - I(s_i)+ 2\epsilon, \alpha s_i + \epsilon -
            \epsilon^{-1} \right\} \\
            & \leq \max\left\{I^*(\alpha)+ 2\epsilon, |\alpha| K + \epsilon -
            \epsilon^{-1} \right\}
        \end{align*}
        for each~$i=1,2,\dotsc, r$, and 
        it follows that
        \begin{align*}
        \overline q(\alpha) &\leq \limsup_{n\to\infty} \frac 1 n	\ln \left(\int_{(K,\infty)}
        \Exp{\alpha n s}\dd\mu_n(s) + \sum_{i=1}^r \int_{G_{s_i}} \Exp{\alpha n s}\dd\mu_n(s)
        \right) \\
        &   \leq \max\{R_K,I^*(\alpha)+ 2\epsilon, |\alpha| K + \epsilon -
        \epsilon^{-1}\},
        \end{align*}
        where $R_K = \limsup_{n\to\infty}  \int_{(K,\infty)} \Exp{\alpha n s}\dd\mu_n(s)$.
        Sending $\epsilon \to 0$ shows that
        \[
            \overline q(\alpha) \leq  \max\{R_K,I^*(\alpha)\}.
        \]
        Finally, sending $K\rightarrow \infty$ and using~\eqref{eq:limK} yields
        $\overline q(\alpha) \leq I^*(\alpha)$, as claimed.

        \item
        When $\alpha <0$,
        $$
            \limsup_{n\to\infty}  \frac 1n \ln\int_{(K,\infty)} \Exp{\alpha n s}\dd\mu_n(s)\leq \alpha K + \limsup_{n\to\infty}  \frac 1n  \ln\int_{(K,\infty)} \dd\mu_n(s)\leq \alpha K,
        $$
        so Part~\ref*{part:qleqistarcond} ensures that $\overline q(\alpha)\leq I^*(q)$. Since also $\underline q(\alpha)\geq I^*(q)$ by Part~\ref*{part:qgeqistar}, we conclude that $q(\alpha) = I^*(\alpha)$.

        \item
        If $\overline q(\alpha) < \infty$ for some $\alpha> 0$, then a standard application of Chebyshev's exponential inequality yields
        $$\limsup_{n\to\infty} \frac 1n \ln \mu_n((K,\infty))\leq \overline q(\alpha)-\alpha K,$$
        which establishes exponential tightness, and thus the full LDP and the goodness of the rate function (in fact $I(s) \geq \alpha s - \overline q(\alpha)$), see e.g.\ Lemma~1.2.18 in \cite{DZ}.

        \item
        Since $I$ is convex and lower semicontinuous, we obtain $I=I^{**} = q^*$. In view of Part~\ref*{part:qIfullcases2}, we need only prove the full LDP in the case where $q(\alpha) = \infty$ for all $\alpha > 0$. In this situation, $I(s)=q^*(s) = \sup_{\alpha \leq 0}(\alpha s - q(\alpha))$ for all $s\in \rr$, so clearly the function $I$ is nonincreasing. Thus, since $\inf_{s\geq 0}I(s) = -I^*(0) = -q(0) = 0$, we must have $\lim_{s\to\infty}I(s) = 0$.

        In order to extend the weak LDP to a full one, consider a closed set $\Gamma \subseteq \rr$. Since $\mu_n((-\infty, 0)) = 0$ and $I(s)=\infty$ when $s<0$, it suffices to prove the large-deviation upper bound for $\Gamma' :=\Gamma \cap [0,\infty)$. If $\Gamma'$ is compact, the weak LDP provides the desired upper bound. If $\Gamma'$ is not compact, then it is unbounded, and since $\lim_{s\to\infty}I(s) = 0$ we conclude that  $\inf_{s \in \Gamma} I(s) = 0$, making the large-deviation upper bound trivial.
        \qedhere
    \end{enumerate}
\end{proof}

\subsection{Proof of Theorem~\ref{thm:CJPS}}
\label{sec:proof-CJPS}

The proof relies on the weak LDP established in~\cite{CJPS19}, which we shall promote, using Proposition~\ref{prop:abstractZn}, to a full LDP with suitable Legendre--Fenchel duality relations.

Before we proceed, it should be noted that the arguments below can be simplified in the special case where $\Q= \P$. Indeed, in this case, exponential tightness is immediate (see Remark~\ref{rem:expotight-1}), and one can apply Theorem~2.8 of \cite{CJPS19} to the pair $(\P, \widehat{\P})$, with~$\widehat \P$ the uniform measure. Indeed, in this case, $\frac 1n\sigma_n(x) = \frac 1n\ln \P_n(x_1^n) + \ln |\cA|$; see Remark~2.11 in~\cite{CJPS19}. 

Returning to the general case, we note that the combination of~\assref{ud} and~\assref{jsld} implies, in view of Lemma~A.1 of \cite{CJPS19}, that the pair $(\P, \Q)$ satisfies assumption SSD of \cite{CJPS19}. By an immediate adaptation of Proposition~3.1.2 of \cite{CJPS19}, we conclude that the random variables $x\mapsto  - \ln \Q_n(x_1^n)$ are $\psi$-compatible in the sense of Definition 3.10 of \cite{CJPS19} (the map $\psi_{n,t}$ is constructed in Proposition~3.1.1 therein). Thus, Proposition~3.11 of \cite{CJPS19} implies that the sequence $(Z_n)_{n\in\nn}$ defined by $Z_n(x) := -\frac 1n  \ln \Q_n(x_1^n)$ satisfies the weak LDP with a convex rate function $\ilpq$. More specifically, Proposition~3.11.2 of \cite{CJPS19} ensures that  
there exists a sequence $(\gamma_n)_{n\in\nn}$ with $\gamma_n \to 0$ such
that for all $\epsilon > 0$, all $n\in\nn$ and all $s \in
\rr$,
\begin{equation}\label{eq:unifcontrolnx}
    \frac {1}{n} \ln \mu_n(B(s, \epsilon))  
    \leq
    \gamma_n - \inf_{y\in B(s, \epsilon + (1+|s|)\gamma_n)} \ilpq(y),
\end{equation}
where $\mu_n$ is the distribution of $Z_n$ with respect to $\P$. 
Moreover, Lemma~3.12 in~\cite{CJPS19} implies that the limit defining $\qlpq$ exists for all $\alpha\in \rr$ and defines a lower semicontinuous, convex function $\qlpq:\rr \to (-\infty, \infty]$. 
We are thus in a position to exploit Proposition~\ref{prop:abstractZn}, with $\Omega_* := \Omega$, $Z_n$ as above, and $I:=\ilpq$.

\medskip\noindent
{\em Proof of Parts i--ii}. 
    In view of the above discussion and of Proposition~\ref{prop:abstractZn}.i--ii, it only remains to promote the weak LDP to a full one and to establish the Legendre--Fenchel duality relations~\eqref{eq:LdualIpq}. These two conclusions will follow from Proposition~\ref{prop:abstractZn}.vii, once we have established the identity $\qlpq=\ilpq^*$.
    
    Since Parts~\ref*{part:qgeqistar} and~\ref*{part:qleqistarneg} of Proposition~\ref{prop:abstractZn} ensure that $\qlpq(\alpha)\geq \ilpq^*(\alpha)$ for all $\alpha \in \rr$ and that $\qlpq(\alpha) = \ilpq^*(\alpha)$ for all $\alpha <0$, it remains to prove that $\qlpq(\alpha)\leq \ilpq^*(\alpha)$ for all $\alpha \geq 0$.
    Values of $\alpha \geq 0$ will be split into two regions, separated by the limit
    \begin{equation*}
        \alpha_+ := \lim_{s\to\infty}\frac {\ilpq(s)}{s},
    \end{equation*}
    which exists in $[0, \infty]$ by convexity of~$\ilpq$.

    \medskip
    \noindent\textit{Case 1: $0 \leq \alpha < \alpha_+$.} 
        Let $\delta = \frac 12 \min\{1, \alpha_+-\alpha\}$. Then, there exists $c > 0$ such that
        \begin{equation*}
            \ilpq(s) \geq \alpha s + \delta  s  - c 
        \end{equation*}
        for all $s\geq 0$. Using this and~\eqref{eq:unifcontrolnx}, we deduce that 
        \begin{align*}
            \frac 1 n \ln \mu_n\bigl(B(k,1)\bigr)
                &\leq  \gamma_n - \inf_{y\in B(k, 1 + (1+k)\gamma_n)} \ilpq(y) \leq - \alpha k - \delta k + c' + c''k\gamma_n
        \end{align*}
        for all $k\geq 0$, where $\gamma_n \to 0$ as $n\to\infty$, while the constants~$c'$ and~$c''$ are independent of~$n$ and~$k$. 
        It follows that
        \begin{equation*}
            \mu_n \bigl(B(k,1)\bigr)
                \leq 
                \exp\left(n\left(- \alpha k - \frac{\delta k}{2} + c'\right)\right)
        \end{equation*}
        for all $n$ large enough so that $c''\gamma_n \leq \frac \delta 2$. We thus obtain, for all $K>0$, 
        \begin{align*}
            \int_{(K,\infty)} \Exp{\alpha n s}\dd\mu_n(s) & \leq \sum_{j=1}^\infty \Exp{\alpha n (K+j+1)}\mu_n(B(K+j,1)) \leq  \sum_{j=1}^\infty \Exp{- n (K+j)\frac \delta 2+n(c'+\alpha)}.
        \end{align*}
        By summing the geometric series, we conclude that
        \begin{equation*}
            \limsup_{n\to\infty} \frac 1 n	\ln \int_{(K,\infty)} \Exp{\alpha n s}\dd\mu_n(s)  \leq   -(K+1)\frac{\delta}2 + c'+\alpha.
        \end{equation*}
    Therefore, the condition~\eqref{eq:limK} is satisfied and Proposition~\ref{prop:abstractZn}.iv yields the bound $\qlpq(\alpha)\leq \ilpq^*(\alpha)$.

    \medskip
    \noindent\textit{Case 2: $\alpha \geq \alpha_+$.} 
        Of course, this is only to be considered if $\alpha_+ < \infty$.         
        Since $\ilpq(s)=\infty$ for all $s<0$, we have $\ilpq^*(\alpha) = \sup_{s\geq 0}(\alpha s - \ilpq(s))$ for all $\alpha\in \rr$, so the function $\ilpq^*$ is nondecreasing. Thus, both $\qlpq$ and $\ilpq^*$ are nondecreasing, lower semicontinuous functions, which implies that they are left continuous. It follows that the inequality $\qlpq(\alpha)\leq \ilpq^*(\alpha)$, which we have now established for all $\alpha\in (-\infty, \alpha_+)$, extends to $\alpha  = \alpha_+$.\footnote{While this extension to $\alpha_+$ may seem to be a technical detail, lower semicontinuity of $\qlpq$, which is ensured by a subadditivity argument in Lemma~3.12 of~\cite{CJPS19}, plays an important role here, preventing for example the situation described in Remark~\ref{rem:weirdIQ} (in that example $\alpha_+ = 0$).}
        
        Finally, if $\alpha > \alpha_+$, then
        obviously $\ilpq^*(\alpha) \geq \lim_{s\to\infty}(\alpha s - \ilpq(s)) = \infty$ by the definition of $\alpha_+$, so the inequality $\qlpq(\alpha)\leq \ilpq^*(\alpha)$ is trivial.

    \medskip

    \noindent Combining the two cases, we have proved that $\qlpq\leq \ilpq^*$, and thus the proof of Theorem~\ref{thm:CJPS}.i--ii is complete. 

\medskip\noindent
{\em Proof of Part~iii}. 
    In this part $\Q= \P$. The four claims are proved in an order which differs from that of the statement.

    \begin{itemize}
        \item [c.]         By the rightmost expression in~\eqref{eq:limitqpq},
        \begin{equation}\label{eq:qp1supppn}
            \qlpp(1) = \lim_{n\to\infty} \frac 1n \ln |\supp \P_n(u)| = \htop(\supp \P),
        \end{equation}
        so~\eqref{eq:qppbdd} is a consequence of monotonicity in~$\alpha$.

        \item[b.] We need to prove that
        \begin{equation}
        \label{eq:gammapmlimits}
            \gamma_+ = \lim_{n\to\infty}\frac 1n \sup_{u\in \cA^n}\ln \P_n(u)
            \qquad\text{and}\qquad
            \ \gamma_- = \lim_{n\to\infty}\frac 1n \inf_{u\in \supp \P_n}\ln \P_n(u).
        \end{equation}
    
        By~\assref{ud}, the sequence $(a_n)_{n\in \nn}$ defined by $a_n := \sup_{u\in \cA^n} \ln \P_n(u)$ satisfies $a_{n+m+\tau_n}\leq a_n+a_m+\ln C_n$ for all $n,m\in \nn$, and a variant of Fekete's subadditive lemma applies~\cite[\S{2}]{Ra22}.
        This proves the statement about~$\gamma_+$.
    
        To prove the statement about $\gamma_-$, set this time $a_n :=  \inf_{u\in \supp \P_n}\ln \P_n(u)$ and let $n,m \in \nn$ be arbitrary. Pick $u \in \cA^n$ such that $\ln \P_n(u) = a_n$ and $v \in \cA^m$ such that $\ln\P_m(v) = a_m$. Since $\P_m(v) > 0$, there exists $b\in \cA^{\tau_n}$ such that $\P_{\tau_n+m}(bv)>0$, and  by~\assref{sld}, there exist $0\leq \ell \leq \tau_n$ and $\xi\in \cA^\ell$ such that $\P_{n+\ell+\tau_n + m}(u\xi b v) > 0$. But then, in view of~\assref{ud},
        $$
        0<\P_{n+\ell+\tau_n+m}(u\xi bv)\leq C_n\P_n(u)\P_{m+\ell}(b^{\tau_n}_{\tau_n-\ell+1}v)\leq C_n\P_n(u)\P_m(v),
        $$
        with the convention that $\xi_{\tau_n+1}^{\tau_n}$ is the empty word. Therefore, $u\xi bv \in \supp \P_{n+\ell+\tau_n+m}$ and thus
        \[
            a_{n+\ell+m+\tau_n} \leq \ln \P_{n+\ell+\tau_n+m}(u\xi bv) \leq  a_n + a_m  + \ln C_n .
        \]
        Since the sequence $(a_n)_{n\in\nn}$ is nonincreasing, we conclude that $a_{n+m+2\tau_n}\leq a_n + a_m + \ln C_n$, so the same variant of Fekete's subadditive lemma yields the claim.

        \item [a.]   In view of~\eqref{eq:qppbdd}, Proposition~\ref{prop:abstractZn}.\ref*{part:qIfullcases2} applies, so $(Z_n)_{n\in\nn}$ is exponentially tight and $\ilpp$ is a good rate function.
        
        We now prove~\eqref{eq:Isinterval}.
        First, by \eqref{eq:qp1supppn} and since $\qlpp(1) = \ilpp^*(1)=\sup_{s\in \rr}(s - \ilpp(s))$, we conclude that $\ilpp(s) \geq s-\htop(\supp \P)$ for all $s\in \rr$. 

        By the Ruelle--Lanford representation of $\ilpp$, and since 
        $$
            \mu_n(B(s,\epsilon)) \leq \Exp{-n(s -\epsilon)} \left|\left\{u\in \cA^n:-\frac 1n \ln \P_n(u)\in B(s,\epsilon)\right\}\right|
        $$
        with a similar lower bound (with $\Exp{-n(s +\epsilon)} $), we obtain
        \begin{equation}\label{eq:representationRL}     
            \ilpp(s) = s - \lim_{\epsilon\to 0} \limsup_{n\to\infty}\frac 1n \ln \left|\left\{u\in \cA^n:-\frac 1n \ln \P_n(u)\in B(s,\epsilon)\right\}\right|.
        \end{equation}   
        The limit superior can be replaced by a limit inferior. Then, since the logarithm is either nonnegative or equal to $-\infty$,
        we conclude that $\ilpp(s)$ is either bounded above by $s$ or infinite. 
    
        If $s< -\gamma_+$ or $s>-\gamma_-$, then the definition of $\gamma_\pm$ implies that for $\epsilon>0$ small enough and $n$ large enough there is no $u\in \cA^n$ such that $-\frac 1n \ln \P_n(u)\in B(s,\epsilon)$, so  $\ilpp(s)=\infty$.  
    
        Next, by \eqref{eq:gammapmlimits}, we have for all $\epsilon>0$ that the set on the right-hand side of~\eqref{eq:representationRL} with $s=-\gamma_+$ is nonempty for all $n$ large enough, so $\ilpp(-\gamma_+)\leq -\gamma_+$. When $\gamma_- > -\infty$, a symmetric argument shows that $\ilpp(-\gamma_-)\leq -\gamma_-$, so by convexity, we indeed obtain $\ilpp(s)\leq s$ for all $s\in [-\gamma_+, -\gamma_-]$.
    
        Now, if $\gamma_- = -\infty$, then by \eqref{eq:gammapmlimits} we can pick, for each $n\in \nn$, some $u^{(n)}\in \cA^{n}$ so that $s_n := -\frac 1n \ln \P_{n}(u^{(n)}) \to  \infty$ as $n\to\infty$. By~\eqref{eq:unifcontrolnx} with $\epsilon=1$, we obtain
        $$
            \inf_{s'\in B(s_n, 1 + (1+s_n)\gamma_{n})} \ilpp(s') \leq -\frac {1}{{n}} \ln  \mu_n(B(s_n, 1)) +  \gamma_{n}\leq  s_n+  \gamma_{n}.
        $$
        As a consequence, for each $n$, there is some $s_n'\in B(s_n, 1 + (1+s_n)\gamma_{n})$ such that $\ilpp(s_n')\leq s_n + \gamma_n + 1<\infty$, and thus $\ilpp(s_n')\leq s_n'$ by~\eqref{eq:representationRL}. Noting that $s_n' \to \infty$, and recalling that $\ilpp(-\gamma_+)\leq-\gamma_+$, convexity allows to conclude that $\ilpp(s)\leq s$ for all $s\geq -\gamma_+$. The proof of \eqref{eq:Isinterval} is complete.

        \item[d.]  If $\gamma_- > -\infty$, then~\eqref{eq:Isinterval} yields $
        \qlpp(\alpha) = \sup_{s\in [-\gamma_+, \gamma_-]}(\alpha s - \ilpp(s)) < \infty
        $ for all $\alpha \in \rr$.
        If on the contrary  $\gamma_- = -\infty$, then~\eqref{eq:Isinterval} implies that $\ilpp(s) \leq s$ for all $s\in [-\gamma_+, \infty)$, 
        so that, for $\alpha > 1$, we obtain $\qlpp(\alpha) \geq \lim_{s\to\infty}(\alpha s - \ilpp(s)) = \infty$. \hfill\qed
    \end{itemize}

\newcommand{\etalchar}[1]{$^{#1}$}

\end{document}